\documentclass[11pt, draft]{amsart}
\usepackage{amssymb, amstext, amscd, amsmath, amssymb}
\usepackage{mathtools, xypic, paralist, color, dsfont, rotating}
\usepackage{verbatim}
\usepackage{enumerate}
\usepackage{enumitem}
\usepackage{calc}
\usepackage{float}

\numberwithin{equation}{section}

\usepackage[a4paper]{geometry}
\usepackage{quoting}
\quotingsetup{vskip=.1in}
\quotingsetup{leftmargin=.17in}
\quotingsetup{rightmargin=.17in}

\makeatletter
\def\@cite#1#2{{\m@th\upshape\bfseries%
[{#1\if@tempswa{\m@th\upshape\mdseries, #2}\fi}]}}
\makeatother
\theoremstyle{plain}
\newtheorem{theorem}{Theorem}[section]
\newtheorem{corollary}[theorem]{Corollary}
\newtheorem{proposition}[theorem]{Proposition}

\theoremstyle{definition}
\newtheorem{definition}[theorem]{Definition}
\newtheorem{example}[theorem]{Example}

\newtheorem{remark}[theorem]{Remark}

\newtheorem*{acknow}{Acknowledgements}
\theoremstyle{remark}

\mathtoolsset{centercolon}
  \newcommand{\A}{{\mathcal{A}}}
  \newcommand{\B}{{\mathcal{B}}}

  \newcommand{\F}{{\mathcal{F}}}
  \newcommand{\G}{{\mathcal{G}}}

  \newcommand{\K}{{\mathcal{K}}}
\renewcommand{\L}{{\mathcal{L}}}

\renewcommand{\O}{{\mathcal{O}}}

\renewcommand{\S}{{\mathcal{S}}}
  \newcommand{\T}{{\mathcal{T}}}

\newcommand{\eps}{\varepsilon}
\def\al{\alpha}
\def\be{\beta}
\def\ga{\gamma}
\def\de{\delta}
\def\ze{\zeta}

\def\ka{\kappa}
\def\la{\lambda}
\def\La{\Lambda}
\def\om{\omega}

\def\si{\sigma}

\newcommand\vphi{\varphi}

\newcommand{\bC}{\mathbb{C}}

\newcommand{\bF}{\mathbb{F}}
\newcommand{\bN}{\mathbb{N}}
\newcommand{\bT}{\mathbb{T}}
\newcommand{\bZ}{\mathbb{Z}}
\newcommand{\bR}{\mathbb{R}}

\newcommand{\AND}{\text{ and }}
\newcommand{\FOR}{\text{ for }}

\newcommand{\foral}{\text{ for all }}
\newcommand{\qand}{\quad\text{and}\quad}

\newcommand{\qfor}{\quad\text{for}\; }

\newcommand{\ca}{\mathrm{C}^*}

\newcommand{\ol}{\overline}

\newcommand{\ad}{\operatorname{ad}}

\newcommand{\Aut}{\operatorname{Aut}}

\newcommand{\Avt}{\operatorname{AVT}}

\newcommand{\diag}{\operatorname{diag}}

\newcommand{\End}{\operatorname{End}}

\newcommand{\Eq}{\operatorname{E}}
\newcommand{\fty}{\operatorname{fin}}

\newcommand{\id}{{\operatorname{id}}}
\newcommand{\im}{{\operatorname{Im}}}

\newcommand{\mt}{\emptyset}

\newcommand{\spn}{\operatorname{span}}

\newcommand{\sumoplus}{\operatornamewithlimits{\sum \strut^\oplus}}
\newcommand{\supp}{\operatorname{supp}}
\newcommand{\Tr}{\operatorname{T}}

\newcommand{\sca}[1]{\langle#1\rangle} 
\newcommand{\nor}[1]{\Vert #1\Vert} 

\addtocontents{toc}{\protect\setcounter{tocdepth}{1}}

\begin{document}

\title[The equilibrium states of Pimsner algebras]{Entropy theory for the parametrization of the equilibrium states of Pimsner algebras}

\author[E.T.A. Kakariadis]{Evgenios T.A. Kakariadis}
\address{School of Mathematics, Statistics and Physics\\ Newcastle University\\ Newcastle upon Tyne\\ NE1 7RU\\ UK}
\email{evgenios.kakariadis@ncl.ac.uk}



\begin{abstract}
We consider Pimsner algebras that arise from C*-corres\-pondences of finite rank, as dynamical systems with their rotational action.
We revisit the Laca-Neshveyev classification of their equilibrium states at positive inverse temperature along with the parametrizations of the finite and the infinite parts simplices by tracial states on the diagonal.
The finite rank entails an entropy theory that shapes the KMS-structure.
We prove that the infimum of the tracial entropies dictates the critical inverse temperature, below which there are no equilibrium states for all Pimsner algebras.
We view the latter as the entropy of the ambient C*-correspondence.
This may differ from what we call strong entropy, above which there are no equilibrium states of infinite type.
In particular, when the diagonal is abelian then the strong entropy is a maximum critical temperature for those.
In this sense we complete the parametrization method of Laca-Raeburn and unify a number of examples in the literature.
\end{abstract}

\maketitle

\section{Introduction}

The Fock space construction gives a concrete quantization of systems in terms of Hilbertian operators.
Originating from Quantum Mechanics, it has seen an important generalization to Hilbert bimodules over C*-algebras, better known as C*-correspondences.
The key element is the existence of a C*-algebra $A$ acting ``externally'' on $X$ and of an $A$-valued inner product.
Rieffel \cite{Rie74} originally envisioned C*-correspondences as a tool to identify C*-algebras in terms of their representation theory.
Pimsner \cite{Pim97} much later extended the theory to accommodate a range of examples of Operator Algebras arising from C*-dynamics and graphs.
The Pimsner algebras generalize the well known Toeplitz- and Cuntz-algebras and they have been under considerable study since their introduction.
By now they form a topic in its own respect with several interactions with graph theory and ring theory.
The C*-correspondence machinery is now viewed as an effective way for quantizing geometric structures that evolve in discrete time.

Nevertheless, the interplay of C*-algebras with Quantum Statistical Mechanics goes well beyond that point.
Taking motivation from ideal gases, there is an analogue of a Kubo-Martin-Schwinger condition for states of C*-algebras that admit an $\bR$-action, even when moving beyond the trace class operators.
See for example the seminal monographs of Bratelli-Robinson \cite{BraRob87, BraRob97}.
The parametrization of equilibrium states has been an essential task in the past 30 years, as they can serve as an invariant for $\bT$-equivariant isomorphisms.
To give only but a fragment of a very long list we mention the Cuntz-algebra \cite{Eva80, OlePed78}, C*-algebras of different types of dynamical systems \cite{BEK80, Exe03, ExeLac03, IKW07, Kak14b, KumRen06, LRR11, LRRW13, PWY00, Tho11, Tho12}, graph C*-algebras \cite{HLRS13, HLRS15, KajWat13, RRS15}, C*-algebras related to number systems \cite{BosCon95, CDL12, LacNes11, LacRae10} and to subshifts \cite{EFW84, MWY98}, and Pimsner algebras \cite{Kaj09, LacNes04}.

The major steps for classifying the equilibrium states of Pismner algebras were established in the seminal paper of Laca and Neshveyev \cite{LacNes04}.
Their arguments were further refined by Laca and Raeburn \cite{LacRae10} in their study of C*-algebras arising from number systems.
The approach of Laca-Raeburn has been very influential, and effectively applicable in a big variety of examples, e.g. \cite{HLRS13, HLRS15, Kak14b, LRR11, LRRW13}.
However in each occasion ad-hoc data is used to trigger the algorithm.
The aim of this paper is to show how these ideas combine with the notion of entropy of Pinzari, Watatani and Yotetani \cite{PWY00} that is induced when the ambient C*-correspon\-den\-ces have finite rank; an assumption that holds in the aforementioned cases.
The KMS-structure of the Pimsner algebras in \cite{HLRS13, HLRS15, KajWat13, Kak14b, LRR11, LRRW13} follows as an application of this analysis.

\subsection{Motivation}

The \emph{Toeplitz-Pimsner algebra} $\T_X$ is the C*-algebra generated by the left creation operators of $X$ and $A$ acting on the Fock space $\F X$.
In addition, there is a range of Pimsner algebras that encodes desirable redundancies.
Every quotient of $\T_X$ by $\bT$-equivariant relations gives rise to a \emph{$J$-relative Cuntz-Pimsner algebra} $\O(J,X)$, where $J \subseteq \phi_X^{-1}(\K X)$ for the left action $\phi_X$ of $A$ and the compact operators $\K X$.
Among those the \emph{Cuntz-Pimsner algebra} $\O_X$ is of central importance and arises when $J$ is Katsura's ideal \cite{Kat04}.
It is the smallest $\bT$-equivariant quotient of $\T_X$ that admits a faithful copy of $A$ and $X$ \cite{Kak14a}.
In general $\O_X \neq \O(\phi^{-1}_X(\K X), X)$ but they coincide with $\O(A, X)$ when $X$ is injective and $\phi_X(A) \subseteq \K X$.

Laca and Neshveyev \cite{LacNes04} studied actions implemented by one-parameter unitaries for injective C*-correspondences.
Their main tool was the use of induced traces from \cite{ComZet83, CunPed79, Ped69}.
In this way they were able to classify the equilibrium states of $\T_X$ in terms of their restrictions on the diagonal by using iterations of the inducing map at each level of the Fock space.
Following Exel-Laca \cite{ExeLac03}, they proved a Wold decomposition into a finite part (given by a series of iterations of a tracial state on the diagonal) and an infinite part (where iterations are stable).
They showed that $\T_X$ admits a rich KMS-structure from which they could derive that of $\O_X$ (when $X$ is injective) through the infinite part.
A characterization was also given for ground states.

Later Laca-Raeburn \cite{LacRae10} refined the main tools of \cite{LacNes04} for a specific class of Pimsner algebras coming from number systems.
From then on the interest was restricted to dynamics implemented by the rotational action.
Most notably they use the statistical approximations of \cite{LacNes04} to parameterize the finite part by tracial states on $A$.
As we shall explain later there is a difference between the parametrizations in \cite{LacNes04} and in \cite{LacRae10}.
Likewise, weak*-homeomorphic parametrizations were given for both ground states and KMS${}_\infty$-states in \cite{LacRae10}.

In turn, a number of subsequent works, e.g. \cite{HLRS13, HLRS15, Kak14b, LRR11, LRRW13}, were greatly influenced by the parametri\-zations of \cite{LacNes04, LacRae10} and applied their method to other examples of Pimsner algebras.
A re-appearing theme is the existence of two critical temperatures $\be_c \geq \be_c'$ for which:

\vspace{2.5pt}
\textbf{(a)} for $\be > \be_c$ the algorithm of \cite{LacRae10} gives all equilibrium states for $\T_X$;

\vspace{1.5pt}
\textbf{(b)} for $\be = \be_c'$ there is an association with averaging states; and

\vspace{1.5pt}
\textbf{(c)} there are no equilibrium states below $\be_c'$.

\vspace{2.5pt}
\noindent
At the other extreme $\O(A, X)$ is not amenable to the construction of (a) but it provides the states for (b).
Such an example is the averaging state on the Cuntz-algebra $\O_d$ which is the only possible equilibrium state (and it appears at $\be = \log d$).

The critical temperatures often coincide and can be associated to structural data of the original construct. 
For example, an Huef-Laca-Raeburn-Sims \cite{HLRS13} show that $\be_c = \be_c'$ is the logarithm of the Perron-Fr\"{o}benius eigenvalue when the graph is irreducible.
In a continuation \cite{HLRS15} the authors also show that a more rich structure appears for general graphs.
That was also verified by Kajiwara-Watatani \cite{KajWat13} who studied the KMS-structure of Cuntz-Krieger C*-algebras.
In the process they achieve also a parametrization of the infinite part of $\O(J, X)$ for $J$ inside Katsura's ideal under some assumptions on the C*-correspondence.
However this does not cover the infinite part in the non-injective case, i.e., it does not cover the case $J = \phi_X^{-1}(\K X)$.
These works motivate the following question:

\vspace{-2pt}
\begin{quoting}
\noindent
\textbf{Q.} How $A$ and $X$ dictate the critical temperature(s) beyond which we don't have equilibrium states of Pimsner algebras?
\end{quoting}

\vspace{-2pt}
\noindent
In the current paper we show how this is done under the assumption that $X$ attains a finite set $\{x_1, \dots, x_d\}$ of vectors in its unit ball such that $1_X = \sum_{i \in [d]} \theta_{x_i, x_i}$.
Equivalently, when the adjointable operators of $X$ are compact.
This is satisfied in the aforementioned examples, and sometimes on the much stronger side of the vectors being orthogonal.
We are not assuming orthogonality here.

Also, we mention that we consider just the dynamics coming from the rotational $\bT$-action for which there is a physical interpretation. 
Recall that the starting point for Gibbs states is the action implemented by $r \mapsto e^{ir(H - \ka N)}$, where $H$ is the \emph{Hamiltonian}, $N$ is the \emph{number operator} and $\ka$ is the \emph{chemical potential}.
When $H$ is the Hamiltonian of a Quantum Harmonic Oscillator then it admits the solution $H = h\om(1/2 + N)$ for the energy dimension $h\om/2$ of the ground state, and the action is implemented by $r \mapsto e^{ir h \om/2} e^{ir(h\om - \ka)N}$.
Since $N$ is unbounded some effort is required to make the action precise.
This can be seen for example in Proposition \ref{P:real} where it is shown that $r \mapsto \ga_{e^{irs}}$ realizes any action implemented by $r \mapsto e^{ir(c + s N)}$ for $c \in \bC$ and $s \in \bR$.
In what follows we make the normalization $h\om - \ka = 1$.
Recall that $\ka < h\om/2$ for any Quantum Harmonic Oscillator and thus substituting $\be$ by $(h \om -\ka) \be$ covers all cases.

\subsection{Decomposition and parametrization}

We write $\Eq_\be(\O(J,X))$ for the $(\si,\be)$-KMS states of the $J$-relative Cuntz-Pimsner algebra $\O(J,X)$ with respect to the action $r \mapsto \si_r := \ga_{e^{ir}}$.
Every $\O(J,X)$ is the quotient of $\T_X$ by a $\bT$-equivariant ideal and hence in order to understand $\Eq_\be(\O(J,X))$ it suffices to do so for $\Eq_\be(\T_X)$.
We need to revisit in detail the main points of \cite{KajWat13, LacNes04} and in particular see how the method of \cite{LacRae10} extends to unify \cite{HLRS13, Kak14b, LRR11, LRRW13}.

In what follows fix $\{x_1, \dots, x_d\}$ be a finite unit decomposition.
Then $\{x_\mu \mid |\mu| = n\}$ yields a unit decomposition for $X^{\otimes n}$, where we write $x_{\mu_n \cdots \mu_1} = x_{\mu_n} \otimes \cdots \otimes x_{\mu_1}$ for a word $\mu = \mu_n \cdots \mu_1$ on the $d$ symbols.
Consequently the projections $p_n \colon \F X \to X^{\otimes n}$ and the compacts $\K(\F X)$ are in $\T_X$.
The finite and the infinite parts of the Wold decomposition at $\be >0$ of \cite{LacNes04} form respectively the convex sets:
\begin{equation}
\Eq_\be^{\fty}(\T_X) := \{\vphi \in \Eq_\be(\T_X) \mid \sum_{k=0}^\infty \vphi(p_k) = 1\}
\qand
\Eq_\be^\infty(\T_X) := \{\vphi \in \Eq_\be(\T_X) \mid \vphi(p_0) = 0\}.
\end{equation}
In particular $\Eq_\be^\infty(\T_X)$ corresponds to the states annihilating $\K(\F X)$ (and thus to those that factor through $\O(A,X)$), and $\Eq_\be^{\fty}(\T_X)$ corresponds to those that restrict to states on $\K(\F X)$ (see Theorem \ref{T:con}).
We then construct the parametrization of each convex set by a specific convex set in the tracial states $\Tr(A)$ of $A$.
This is linked to the formal series
\begin{equation}
c_{\tau,\be} := \sum_{k=0} e^{-k\be} \sum_{|\mu| = k} \tau(\sca{x_\mu, x_\mu})
\FOR \tau \in \Tr(A)
\AND \be > 0.
\end{equation}
We thus need to consider the sets that arise from two extreme cases:
\begin{equation}
\Tr_\be(A) := \{\tau \in \Tr(A) \mid c_{\tau,\be} < \infty\}
\AND
\Avt_\be(A) := \{\tau \in \Tr(A) \mid e^\be \tau(\cdot) = \sum_{i \in [d]} \tau(\sca{x_i, \cdot x_i})\}.
\end{equation}
Notice that $c_{\tau,\be} = \sum_{k=0}^\infty 1$ for every $\tau \in \Avt_\be(A)$.

The parametrization of $\Eq_\be^{\fty}(\T_X)$ is constructive and follows from \cite{LacNes04, LacRae10}.
In Theorem \ref{T:para} we show that there is a bijection
\begin{equation}
\Phi \colon \Tr_\be(A) \to \Eq_\be^{\fty}(\T_X) \text{ such that } \Phi(\tau)(p_0) = c_{\tau,\be}^{-1}.
\end{equation}
In particular $\Phi$ can be reconstructed by
\begin{equation}\label{eq:stat app}
\Phi(\tau) (t(\xi^{\otimes n}) t(\xi^{\otimes m})^*)
=
\de_{n,m} c_{\tau,\be} \sum_{k=0}^\infty e^{-(k+n) \be} 
\sum_{|\mu| = k} \tau(\sca{\eta^{\otimes m} \otimes x_\mu, \xi^{\otimes n} \otimes x_\mu})
\end{equation}
for all $\xi^{\otimes n} \in X^{\otimes n}$ and $\eta^{\otimes m} \in X^{\otimes m}$.
When $\Eq_\be^{\fty}(\T_X)$ is weak*-closed then $\Phi$ is a weak*-homeomorphism.
As a new outcome of this analysis we derive that the map $\Phi$ preserves convex combinations (by weighting over the $c_{\tau,\be}$), and thus it preserves the extreme points.

Theorem \ref{T:para} uses the crux of the arguments of \cite[proof of Theorem 2.1]{LacNes04} but as with \cite{HLRS13, HLRS15, Kak14b, LacRae10, LRR11, LRRW13} there are slight differences.
First of all the correspondence between $\Eq_\be(\T_X)$ and a subset of $\Tr(A)$ is provided in \cite[Theorem 2.1]{LacNes04}, yet as a correspondence $\Eq_\be(\T_X) \to \Tr(A)$ given by restriction $\Phi \mapsto \Phi|_A$, and it is not linked to $\Tr_\be(A)$.
In the comments preceding \cite[Definition 2.3]{LacNes04} it is hinted how a $\tau$ might be obtained from $\Phi$ but the suggested map requires normalization (by the possibly non-constant $c_{\tau, \be}$).
The constructive approach we take here tackles this point.
Secondly, $\Phi$ is obtained through induced representations of Toeplitz-Pimsner algebras rather than the theory of induced traces from \cite{ComZet83, CunPed79, Ped69}.

The infinite part is dealt with in Theorem \ref{T:para 0} where an affine weak*-homeomor\-phism  is constructed:
\begin{equation}
\Psi \colon \{\tau \in \Avt_\be(A) \mid \tau|_I = 0\} \to \Eq_\be^\infty(\T_X) \text{ such that } \Psi(\vphi)|_{A} = \tau,
\end{equation}
for the ideal of $A$
\begin{equation}
I := \{a \in A \mid \lim_n \nor{\phi_X(a) \otimes \id_{X^{\otimes n-1}}} = 0\}.
\end{equation}
The ideal $I$ is the kernel of the canonical quotient $q \colon \T_X \to \O(A,X)$ and arises from the fact that every $\vphi \in \Eq_\be^\infty(\T_X)$ factors through $q$.
The proof follows the lines of \cite[Theorem 3.18]{KajWat13} with the additional use of $I$.
The main tool is that the fixed point algebra is the inductive limit of the $\K X^{\otimes n}$ when $X$ is injective.
It has been implicitly applied in \cite{MWY98, Kak14b} to obtain equilibrium states at the critical temperature.

The affine weak*-homeomorphism has been obtained in \cite[Theorem 2.1 and Theorem 2.5]{LacNes04}, when $X$ is injective and non-degenerate, but with an entirely different line of attack.
At the end of \cite[proof of Theorem 2.1]{LacNes04} it is shown that any equilibrium state can be given as a limit of finite type states on $\si^\eps$ perturbed actions so that $\lim_{\eps \to 0} \si^\eps = \si$.
Hence they verify that \cite[Formula (2.2)]{LacNes04} gives a well defined extension of a state from $A$ to $\T_X$.
Then \cite[Theorem 2.5]{LacNes04} asserts that \cite[Formula (2.2)]{LacNes04} gives a state $\vphi$ of infinite type if and only if $\vphi|_A \in \Avt_\be(A)$.
Theorem \ref{T:para 0} on the other hand constructs directly the extension within the same action $\si$ without any conditions on $X$.
Moreover this method applies to parametrize the gauge-invariant tracial states on $\O(A,X)$.

By passing to a $\bT$-equivariant quotient we derive a similar characterization for any $J$-relative Cuntz-Pimsner algebra $\O(J, X)$ through the following scheme:
\[
\xymatrix@R=.05pt{
\Tr_\be(A) \cap \{\tau \in \Tr(A) \mid \tau|_J = 0\} \ar[r]^{\phantom{wwwwwww} \Phi} & \Eq_\be^{\fty}(\O(J,X)) \\
& \hspace{4cm} \oplus_{\textup{-convex}} \; = \; \Eq_\be(\O(J,X)). \\
\Avt_\be(A) \cap \{\tau \in \Tr(A) \mid \tau|_I = 0\} \ar[r]^{\phantom{wwwwwww} \Psi} & \Eq_\be^\infty(\O(J,X))
}
\]
\begin{center}
{\footnotesize Figure. Parametrization of equilibrium states of $\O(J,X)$.}
\end{center}
\vspace{6pt}

\noindent
Of course this has to be taken with care as it may be that $\Eq_\be^\infty(\O(J,X))$ or $\Eq_\be^{\fty}(\O(J,X))$ is empty for some choices of $\be$ and $J$.
This brings us to the main point of the discussion that captured our interest in the first place.

\subsection{Entropy}

Taking motivation from the classical case, entropy has been used in various guises.
See the excellent monograph of Neshveyev-St\o rmer in this respect \cite{NesSto06}.
Our approach is closer to that of Pinzari-Watatani-Yotetani \cite{PWY00} who considered imprimitivity bimodules with finite left and right unit decompositions.
The starting point is that the statistical approximation (\ref{eq:stat app}) works only when $c_{\tau,\be} < \infty$.
The ratio test may not be conclusive for all formal series $c_{\tau,\be}$ but it can be used to define the following notions of entropies\footnote{\ As we are concerned about convergence of series we will make the convention that $\limsup_k k^{-1} \log a_k = 0$ if $a_k = 0$ eventually.}.
The \emph{entropy of a tracial state $\tau$ of $A$} is given by
\begin{equation}
h_X^\tau := \limsup_k \frac{1}{k} \log \sum_{|\mu| = k} \tau(\sca{x_\mu, x_\mu}).
\end{equation}
Notice that $h_X^\tau \leq \be$ if $\tau \in \Tr_\be(A)$, and that $h_X^\tau = \be$ if $\tau \in \Avt_\be(A)$.
Moreover $h_X^\tau$ is independent of the choice of the unit decomposition.
On the other hand for a fixed unit decomposition $x = \{x_1, \dots, x_d\}$ we can define
\begin{equation}
h_X^x := \limsup_k \frac{1}{k} \log \nor{\sum_{|\mu| = k} \sca{x_\mu, x_\mu}}_{A}
\end{equation}
where the $\limsup$ is actually a limit.
The \emph{strong entropy of $X$} is then given by
\begin{equation}
h_X^s := \inf\{ h_X^x \mid x = \{x_1, \dots, x_d\} \text{ is a unit decomposition for $X$} \}.
\end{equation}
If $A$ is abelian then $h_X^s$ is the same for all unit decompositions.
We define the \emph{entropy of $X$} as the critical temperature below which we do not attain equilibrium states for any Pimsner algebra, i.e.,
\begin{equation}
h_X := \inf\{ \be > 0 \mid \Eq_\be(\T_X) \neq \mt\} \quad (\text{with } \inf \mt := \infty).
\end{equation}
By weak*-compactness the infimum is actually a minimum. 
In Proposition \ref{P:hx}, Corollary \ref{C:com ent}, Proposition \ref{P:not bigger} and Corollary \ref{C:eq ent} we show that:
\begin{enumerate}[labelindent=20pt,labelwidth=\widthof{\ref{last-item}},leftmargin=!]
\item $h_X^\tau \leq h_X^s  \leq \log d$ for every $\tau \in \Tr(A)$.
\item If $\be > h_X^s$ then $\Tr_\be(A) = \Tr(A)$ and thus $\Eq_\be^\infty(\T_X) = \mt$ and $\Eq_\be(\T_X) = \Eq_\be^{\fty}(\T_X)$.
\end{enumerate}
An essential application of \cite{PWY00} gives also that $\Eq_{h_X^s}^\infty(\T_X) \neq \mt$ when $A$ is abelian and $h_X^s > 0$.
In Corollary \ref{C:eq ent} we provide one of the main conclusions of this analysis; namely, that the entropy of $X$ can be recovered from the state entropies in the following way:
\begin{equation}
h_X = \max \big\{ 0, \inf\{ h_X^\tau \mid \tau \in \Tr(A)\} \big\}.
\end{equation}
In fact, if $h_X >0$ or if $h_\tau \geq 0$ for all $\tau \in \Tr(A)$ then $h_X = \min \{ h_X^\tau \mid \tau \in \Tr(A)\}$.
Consequently, if $h_X^\tau = h_X^s$ for all $\tau \in \Tr(A)$ then $\Eq_\be^\infty(\T_X) = \mt$ whenever $\be > h_X$.
This gives the KMS-states theory of $\O_d$ in a nutshell, and reflects what is known for specific cases in the literature.
In Section \ref{S:ex} we emphasize by examples that:
\begin{enumerate}[resume, labelindent=20pt,labelwidth=\widthof{\ref{last-item}},leftmargin=!]
\item The infimum over all $h_X^x$ is required in the definition of $h_X^s$, as the notion of an orthonormal basis is not well defined for Hilbert modules over non-abelian C*-algebras.
\item If $A$ is abelian and $X$ attains an orthonormal basis then $h_X^\tau = h_X = h_X^s = \log d$ for all $\tau \in \Tr(A)$.
Moreover $\Eq_\be(\T_X) = \Eq_\be^{\fty}(\T_X) \neq \mt$ for all $\be > \log d$, and $\Eq_{\log d}(\T_X) = \Eq_{\log d}^\infty(\T_X) \neq \mt$.
\item There may be both finite and infinite parts for $\T_X$ when $\be \in (h_X, h_X^s)$.
\end{enumerate}
As a second application we show how the entropy theory fully recovers the KMS-structure of Pimsner algebras of irreducible graphs \cite{HLRS13, KajWat13}, and that of Pimsner algebras related to dynamical systems or self-similar actions of \cite{Kak14b, LRR11, LRRW13}.
For these examples we derive item (v) above, where the value $d$ is specified by the intrinsic data of the related C*-correspondence.
We also take a look at the KMS-simplices for reducible graphs that have been identified in \cite{HLRS15}.
The main tool there is to study hereditary closures of connected components and quotients by passing to subgraphs.
As a third application we show how the theory of \cite{HLRS15} is recovered just by using entropies, and let us provide a description here.

Let $G_1, \dots, G_m$ be the irreducible components of $G$ with respect to which $G$ takes up an upper triangular form\footnote{\ We write $G_{ij}$ for the edges from $i$ to $j$.
Although this is the transpose of \cite{HLRS15}, it facilitates some computations in the proofs.}.
Each component comes with its Perron-Fr\"{o}benius eigenvalue $\la_{G_s}$ and it is well known that
\begin{equation}
h_G:= \limsup_k \frac{1}{k} \log (\sum_{i,j \in G} (G^k)_{ij}) = \max\{ \log \la_{G_1}, \dots, \log \la_{G_m}\}
\end{equation}
for the entropy $h_G$ of the graph.
The first step is to identify $h_X$ and $h_X^s$ (Theorem \ref{T:graph entropy}).
We verify that $h_G$ equals the strong entropy of the graph C*-correspondence.
Next we say that an irreducible component $G_s$ is a \emph{sink} if there are no paths emitting from $G_s$ that end at a vertex outside $G_s$.
If there exists a zero sink component then $h_X = 0$; otherwise 
\begin{equation}
h_X
=
\min\{\log \la_{G_s} \mid G_s \textup{ is a non-zero sink irreducible component of } G \}.
\end{equation}
A direct computation shows that the averaging traces correspond exactly to positive eigenvectors of the transpose matrix $G^t$.
In Proposition \ref{P:tr en gr} it is shown en passant that
\begin{equation}
h_X^ \tau 
=
\max \{\log \la_{G_s} \mid \textup{ $G_s$ is communicated by some $v_{r}$ in the support of $\tau$} \},
\end{equation}
and that $\tau \in \Tr_\be(A)$ if and only if $h_X^\tau < \be$.
The classification of the KMS-simplices of \cite{HLRS15} then follows in the following way (Theorem \ref{T:phase G}).
There are phase transitions exactly at the numbers
\begin{equation}
\La:= \{\log \la_{G_s} \mid \textup{ $\la_{G_s} \geq \la_{G_r}$ whenever $G_r$ is communicated by $G_s$}\}.
\end{equation}
The $G_s$ that contribute to $\La$ are called \emph{$\la_{G_s}$-maximal}, and correspond to the minimal components described in \cite{HLRS15}.
For convenience we order $\La$ by
\begin{equation}
h_X = \log \la_1 < \cdots < \log \la_q = h_X^s,
\end{equation}
so that for every $n$ there are $G_{n, 1}, \dots, G_{n, k_n}$ that are $\la_n$-maximal.
Let us set
\begin{equation}
V_n : = \{v \in G \mid \textup{ $v$ communicates with some $G_{n, s}$, $s = 1, \dots, k_n$}\}.
\end{equation}
Since $h_X^\tau > \log \la_n$ if and only if the support of $\tau$ communicates with $G_{j,s}$ for some $j > n$ we get
\begin{equation}
\Tr_\be(A)
=
\{\tau \in \Tr(A) \mid  \tau|_{V_j} = 0 \foral j > n\}
\foral
\be \in (\log \la_n, \log \la_{n+1}].
\end{equation}
In particular every $\Tr_\be(A)$ is weak*-closed.
On the other hand the averaging traces $\Avt_\be(A)$ at $\be = \log \la_{G_{s_0}}$ correspond to the $\ell^1$-normalized eigenvectors of the transpose of
\begin{equation}
H_{s_0}
=
\begin{bmatrix}
G_{s_0} & \ast & \cdots & \ast \\
0 & G_{s_1} & \cdots & \ast \\
\vdots & \vdots & \cdots & \vdots \\
0 & 0 & \cdots & G_{s_q}
\end{bmatrix}
\end{equation}
where $G_{s_1}, \dots, G_{s_q}$ are the components communicated by $G_{s_0}$.
We thus conclude that for every $\be \in (\log \la_n, \la_{n+1}]$ with $\log \la_n, \log \la_{n+1} \in \La$ we have an affine weak*-homeomorphism
\[
\Phi \colon \{\tau \in \Tr(A) \mid \tau|_{V_j} = 0 \foral j > n \} \to \Eq_\be^{\fty}(\T_X),
\]
and for every $\be  = \log \la \in \La$ we have an affine weak*-homeomorphism
\[
\Psi \colon \{ \tau \in \Tr(A) \mid H_{s}^t \tau = \la \tau \textup{ for some $\la$-maximal $G_s$}\} \to \Eq_\be^{\infty}(\T_X).
\]
From this we can easily read the structure for the graph C*-algebra $\O_X$.
The map $\Psi$ descends as is since $A \hookrightarrow \O_X$ while for $\Phi$ we restrict to the $\tau \in \Tr_\be(A)$ that have support entirely on sources.
We provide a variety of examples for which we compute the above ad-hoc in order to highlight the methods of the proofs.
In Section \ref{S:ex hlrs} we square our results with \cite[Examples 6.1--6.7]{HLRS15} by showing how the KMS-simplices can recovered by using entropies.

\subsection{States at the upper half plane}

We follow \cite{LacRae10} and make a distinction between states that are bounded on the upper half plane (ground states) and states that arise at the limit of $\be \uparrow \infty$ (KMS${}_\infty$-states).
The parametrization in Theorem \ref{T:infty rel} resembles that of \cite{HLRS13, Kak14b, LRR11, LRRW13}, which in turn are inspired by \cite[Theorem 2.2]{LacNes04}.
Namely, the mapping $\tau \mapsto \vphi_\tau$ given by
\begin{equation}
\vphi_\tau(f) := 
\begin{cases}
\tau(f) & \text{ if } f \in q_J(A) \subseteq \O(J,X),\\
0 & \text{ otherwise},
\end{cases}
\end{equation}
defines an affine weak*-homeomorphism from the states $\S(A)$ of $A$ (resp. from $\Tr(A)$) that vanish on $J$, onto the ground states of $\O(J,X)$ (resp. the KMS${}_\infty$-states of $\O(J,X)$).

\section{Preliminaries}

\subsection{Kubo-Martin-Schwinger states}

Let $\si \colon \bR \to \Aut(\A)$ be an action on a C*-algebra $\A$.
Then there exists a norm-dense $\si$-invariant $*$-subalgebra $\A_{\textup{an}}$ of $\A$ such that for every $f \in \A_{\textup{an}}$ the function $\bR \ni r \mapsto \si_r(f) \in \A$ is analytically continued to an entire function $\bC \ni z \mapsto \si_z(f) \in \A$ \cite[Proposition 2.5.22]{BraRob87}.
If $\be > 0$, then a state $\vphi$ of $\A$ is called a \emph{$(\si,\be)$-KMS state} (or \emph{equilibrium state at $\be$}) if it satisfies the KMS-condition:
\begin{equation}
\vphi(f g) = \vphi(g \si_{i\be}(f)) \, \text{ for all $f, g$ in a norm-dense $\si$-invariant $*$-subalgebra of $\A_{\text{an}}$}.
\end{equation}
If $\be=0$ or if the action is trivial then a KMS-state is a tracial state on $\A$.
The KMS-condition follows as an equivalent for the existence of particular continuous functions \cite[Proposition 5.3.7]{BraRob97}. More precisely, a state $\vphi$ is an equilibrium state at $\be >0$ if and only if for any pair $f, g \in \A$ there exists a complex function $F_{f, g}$ that is analytic on $D = \{ z \in \bC \mid 0 < \im(z) < \be\}$ and continuous (hence bounded) on $\ol{D}$ such that
\[
F_{f, g}(r) = \vphi(f \si_r(g)) \text{ and } F_{f, g}(r + i \be) = \vphi(\si_r(g) f) \foral t \in \bR.
\]
A state $\vphi$ of $\A$ is called a \emph{KMS$_\infty$-state} if it is the weak*-limit of $(\si,\be)$-KMS states as $\be \uparrow \infty$.
A state $\vphi$ of a C*-algebra $\A$ is called a \emph{ground state} if the function $z \mapsto \vphi(f \si_{z}(g))$ is bounded on $\{z \in \bC \mid \text{Im}z >0\}$ for all $f, g$ inside a dense analytic subset of $\A$.
The distinction between ground states and KMS${}_\infty$-states is not apparent in \cite{BraRob97} and is coined in \cite{LacRae10}.

\subsection{C*-correspondences}

The reader should be familiar with the theory of C*-correspondences, e.g. \cite{Kat04}.
A C*-correspondence $X$ over $A$ is a right Hilbert $A$-module with a left action given by a $*$-homomorphism $\phi_X \colon A \to \L X$.
We write $\K X$ for the ideal of compact operators and we denote the rank one compacts by
\[
\theta_{\xi, \eta} \colon X \to X : \ze \mapsto \xi \sca{\eta, \ze}.
\]
For $n >1$ we write $X^{\otimes n} = X^{\otimes n-1} \otimes X$ for the stabilized $n$-tensor product, with the left action given by $\phi_n = \phi_X \otimes \id_{X^{\otimes n-1}}$.
We write $\xi^{\otimes n} := \xi_1 \otimes \cdots \otimes \xi_n$ for the elementary tensors of $X^{\otimes n}$.

We fix $(\pi,t)$ be the Fock representation of $X$.
That is, on $\F X : = \sumoplus^n X^{\otimes n}$ we define the adjointable operators given on the elementary tensors $\eta^{\otimes n} \in X^{\otimes n}$ by
\[
\pi(a) \eta^{\otimes n} = \phi_n(a) \eta^{\otimes n} 
\FOR a  \in A 
\qand t(\xi) \eta^{\otimes n} = \xi \otimes \eta^{\otimes n}
\FOR \xi \in X.
\]
In order to reduce the use of superscripts we will abuse notation and write $t(\xi^{\otimes n})$ instead of the more appropriate $t^n(\xi^{\otimes n})$, and $t(\xi^{\otimes 0}) = \pi(a)$ for $a = \xi^{\otimes 0} \in A$.
We write $\T_X$ for the \emph{Toeplitz-Pimsner} C*-algebra that is generated by $\pi(A)$ and $t(X)$.
It follows that
\[
\T_X = \ol{\spn} \{t(\xi^{\otimes n}) t(\eta^{\otimes m})^* \mid \xi^{\otimes n} \in X^{\otimes n}, \eta^{\otimes m} \in X^{\otimes m}, n, m \in \bZ_+\}
\]
with the understanding that $X^{\otimes 0} = A$.
It is clear that $\T_X$ admits a gauge action $\ga_z : = \ad_{u_z}$ given by the unitaries
\[
u_z (\xi_n) = z^n \xi_n \foral \xi_n \in X^{\otimes n}.
\]
The Gauge-Invariant-Uniqueness-Theorem (in the full generality obtained by Katsura \cite{Kat04}) asserts that $\T_X$ is the universal C*-algebra with respect to pairs $(\rho, v)$ such that
\[
v(\xi)^* v(\eta) = \rho(\sca{\xi, \eta}) \qand \rho(a) t(\xi) = t(\phi_X(a) \xi).
\]
Any such pair induces a map $\psi_v$ on $\K X$ such that $\psi_v(\theta_{\xi, \eta}) = v(\xi) v(\eta)^*$.
In fact $(\rho, v)$ induces a faithful representation of $\T_X$ if and only if it admits a gauge action and $\rho(A) \cap \psi_v(\K X) = (0)$ (hence $\rho$ is injective).
We also fix the projections
\[
p_n \colon \F X \to X^{\otimes n}.
\]
It is straightforward that the $p_n$ commute with the diagonal operators of $\L(\F X)$ and thus with the elements in the fixed point algebra $\T_X^\ga$.

Let $J \subseteq \phi_X^{-1}(\K X)$.
The \emph{$J$-relative Cuntz-Pimsner algebra $\O(J, X)$} is defined as the quotient of $\T_X$ by the ideal generated by
\[
\pi(a) - \psi_t(\phi_X(a)) \foral a \in J.
\]
As such it inherits the gauge action from $\T_X$.
In particular $\O(J,X)$ is the universal C*-algebra with respect to pairs $(\rho, v)$ that in addition satisfy the \emph{$J$-covariance} $\rho(a) = \psi_v(\phi_X(a))$ for all $a \in J$.
If $J = J_X$ for \emph{Katsura's ideal}
\[
J_X :=\ker\phi_X^\perp \bigcap \phi_X^{-1}(\K X)
\]
then the quotient is the \emph{Cuntz-Pimsner} algebra $\O_X$ \cite{Kat04}.
It is shown in \cite{Kak14a} that $A$ embeds in $\O(J, X)$ if and only if $J \subseteq J_X$.
In this case the Gauge-Invariant-Uniqueness-Theorem asserts that a pair $(\rho, v)$ defines a faithful representation of $\O(J, X)$ if and only if it is $J$-covariant, it admits a gauge action, $\rho$ is injective and $J = \{a \in A \mid \rho(a) \in \psi_v(\K X)\}$.

The Fock space itself admits Hilbert spaces quantizations.
For convenience we take Hilbert spaces to be conjugate linear in the first entry (so that they are right Hilbert $\bC$-modules).
Suppose that $\rho_0 \colon A \to \B(H_0)$ is a $*$-representation and form the Hilbert module $\F X \otimes_{\rho_0} H_0$.
It is a Hilbert space with the inner product be given by
\[
\sca{\xi^{\otimes n} \otimes x, \eta^{\otimes m} \otimes y} := \sca{x, \rho_0(\sca{\xi^{\otimes n}, \eta^{\otimes m}}_{\F X}) y}_{H_0}
\]
and the induced pair $(\rho, v) := (\pi \otimes I_{H_0}, t \otimes I_{H_0})$ defines a representation of $\T_X$.
Now if we consider $(H_\tau, x_\tau, \rho_{\tau})$ be the GNS-representation of $A$ and $(H_u, \rho_u)$ be the universal representation of $A$ then $(\pi \otimes I_{H_u}, t \otimes I_{H_u})$ defines a faithful representation of $\T_X$ on
\[
\F X \otimes_{\rho_u} H_u \simeq \sumoplus_{\tau \in \S(A)} \F X \otimes_{\rho_\tau} H_\tau.
\]

\subsection{The KMS-simplex and the number operator}

Fix $s \in \bR$.
We use the gauge action to define $\si \colon \bR \to \Aut(\T_X)$ by $\si_r = \ga_{e^{ir s}}$.
It is standard to see then that it extends to an entire function on the analytic elements $f = t(\xi^{\otimes n}) t(\eta^{\otimes m})^*$ of $\T_X$ by setting
\[
\si_z(t(\xi^{\otimes n}) t(\eta^{\otimes m})^*) = e^{(n-m)i z s} t(\xi^{\otimes n}) t(\eta^{\otimes m})^*.
\]
We emphasize here that we consider just elementary tensors.
The $(\si, \be)$-KMS condition for a state $\vphi$ is thus written as
\begin{equation}\label{eq:kms 1}
\vphi(t(\xi^{\otimes n}) t(\eta^{\otimes m})^* \cdot t(\ze^{\otimes k}) t(y^{\otimes l})^*)
=
e^{-(n-m)\be s} \vphi(t(\ze^{\otimes k}) t(y^{\otimes l})^* \cdot t(\xi^{\otimes n}) t(\eta^{\otimes m})^*)
\end{equation}
Likewise we get the $(\si, \be)$-KMS condition for the relative Cuntz-Pimsner algebras $\O(J, X)$.

\begin{definition}
Let $X$ be a C*-correspondence and $J \subseteq \phi^{-1}_X(\K X)$.
For $\be > 0$ we write $\Eq_\be(\O(J,X))$ for the set of the $(\si,\be)$-KMS states of $\O(J,X)$ where $\si \colon \bR \to \Aut(\O(J,X))$ is given by $r \mapsto \ga_{e^{ir}}$ for the gauge action $\ga$ of $\O(J,X)$.
\end{definition}

The rotational action formalizes the distribution $e^{- \be N}$ for the \emph{number operator} $N$ given by $N \xi^{\otimes n} = n \xi^{\otimes n}$.
This is similar to what is done in Quantum Mechanics and let us include some details here.

\begin{proposition}\label{P:real}
Let $X$ be a C*-correspondence over $A$ and let $c \in \bC$ and $s \in \bR$.
If $N$ is the number operator on $\F X$ then $e^{i(c + sN)} = e^{ic} u_{e^{is}}$.
Consequently the action $\si \colon \bR \to \Aut(\T_X)$ with
\[
\si_r(f) := f \mapsto e^{ir(c + s N)} f e^{-ir(c + s N)} \foral f \in \T_X
\]
is realized by the rotational action $\bR \ni r \mapsto \ga_{e^{irs}} \in \Aut(\T_X)$.
\end{proposition}

\begin{proof}
Let $\tau$ be a state of $A$ and form the Hilbert space $\F X \otimes_{\rho_\tau} H_\tau$ for the GNS-representa\-tion $(H_\tau, x_\tau, \rho_\tau)$ of $\tau$.
Then $N \otimes I_{H_\tau} = \sum_{k=0}^\infty k p_k \otimes I_{H_\tau}$ is an unbounded selfadjoint operator on $\F X \otimes_{\rho_\tau} H_{\tau}$.
It suffices to show that 
\[
e^{i(c + sN)} \otimes I_{H_\tau} = e^{ic} u_{e^{is}} \otimes I_{H_\tau}.
\]
For convenience let us we write $p_{k,\tau} = p_k \otimes I_{H_\tau}$.
By the Spectral Theorem for unbounded normal operators we deduce that 
\[
e^{i(c + sN)} \otimes I_{H_\tau} = \text{sot-}\lim_m e^{ic} \prod_{k=0}^m e^{i s k p_{k,\tau}}.
\]
For any $z \in \bC$ we can use the functional calculus to approximate $e^{z p_{k,\tau}}$ by $P_\ell(p_{k,\tau})$ such that the $P_\ell(x) = \sum_{j} \al_{\ell, j} x^j$ converge to $e^{z x}$ for $x \in \{0,1\}$.
Then we get
\[
P_\ell(p_{k,\tau}) (\xi^{\otimes n} \otimes x_\tau)
=
\begin{cases}
\sum_j \al_{\ell, j} (\xi^{\otimes n} \otimes x_\tau) & \text{ if } n = k,\\
\al_{\ell, 0} (\xi^{\otimes n} \otimes x_\tau) & \text{ if } n \neq k,
\end{cases}
=
\begin{cases}
P_\ell(1) (\xi^{\otimes n} \otimes x_\tau) & \text{ if } n = k,\\
P_\ell(0) (\xi^{\otimes n} \otimes x_\tau) & \text{ if } n \neq k.
\end{cases}
\]
and so $e^{z p_{k,\tau}} = e^{z} p_{k,\tau} + \sum_{m \neq k} p_{m,\tau}$. 
Therefore
\begin{align*}
e^{i(c + sN)} \otimes I_{H_\tau} (\xi^{\otimes n} \otimes x_\tau)
& =
\lim_m e^{ic} \prod_{k=0}^m e^{i s k p_{k, \tau}} (\xi^{\otimes n} \otimes x_\tau) \\
& =
e^{ic} e^{i s n} (\xi^{\otimes n} \otimes x_\tau) 
 =
e^{ic} (u_{e^{i s}} \otimes I_{H_\tau}) (\xi^{\otimes n} \otimes x_\tau),
\end{align*}
and the proof is complete.
\end{proof}

Henceforth we focus on the case where $s = 1$.
Substituting $\be$ by $s \be$ in what follows yields the results for any $s \in \bR^+$.

\section{Characterization of equilibrium states}

We start by giving an equivalent characterization of the KMS-condition.

\begin{proposition}\label{P:char}
Let $X$ be a C*-correspondence and let $\be \in \bR$.
Then $\vphi \in \Eq_\be(\T_X)$ if and only if
\begin{equation}\label{eq:kms 2}
\vphi(t(\xi^{\otimes n}) t(\eta^{\otimes m})^*)= \de_{n,m} e^{-n\be} \vphi(t(\eta^{\otimes m})^* t(\xi^{\otimes n}))
\end{equation}
for all elementary tensor vectors $\xi^{\otimes n} \in X^{\otimes n}$, $\eta^{\otimes m} \in X^{\otimes m}$, with $n,m \in \bZ_+$.
Consequently two $(\si, \be)$-KMS states coincide if and only if they agree on $\pi(A)$.

An analogous description holds for the states in $\Eq_\be(\O(J,X))$ for any relative Cuntz-Pimsner algebra $\O(J,X)$.
\end{proposition}

\begin{proof}
Suppose that $\vphi \in \Eq_\be(\T_X)$.
If $n = m$ then the KMS-condition in (\ref{eq:kms 1}) directly gives that
\[
\vphi(t(\xi^{\otimes n}) t(\eta^{\otimes n})^*) = e^{-n\be} \vphi(t(\eta^{\otimes n})^* t(\xi^{\otimes n})).
\]
If $n \neq m$ then we use that $\vphi$ is $\si$-invariant and therefore for every $r \in \bR$ we get
\[
\vphi(t(\xi^{\otimes n}) t(\eta^{\otimes m})^*)
=
\vphi \si_r (t(\xi^{\otimes n}) t(\eta^{\otimes m})^*)
=
e^{(n-m)ir} \vphi(t(\xi^{\otimes n}) t(\eta^{\otimes m})^*).
\]
As $(n-m) \neq 0$ we must have that $\vphi(t(\xi^{\otimes n}) t(\eta^{\otimes m})^*) = 0$.

Conversely suppose that $\vphi$ is a state on $\T_X$ satisfying (\ref{eq:kms 2}).
It will be convenient to refer to elements of the form $t(\xi^{\otimes n}) t(\eta^{\otimes m})^*$ as \emph{$(n,m)$-products}.
We have to verify equation (\ref{eq:kms 1}), i.e.,
\[
\vphi(t(\xi^{\otimes n}) t(\eta^{\otimes m})^* \cdot t(\ze^{\otimes k}) t(y^{\otimes l})^*)
=
e^{-(n-m)\be} \vphi(t(\ze^{\otimes k}) t(y^{\otimes l})^* \cdot t(\xi^{\otimes n}) t(\eta^{\otimes m})^*).
\]
We will proceed by considering cases on $n, m, k, l$.
The left hand side of (\ref{eq:kms 1}) gives either an $(n, m-k+l)$-product or an $(n +k - m, l)$-product, depending on whether $m \geq k$ or $m \leq k$.
Similarly the right hand side gives either a $(k, l-n+m)$-product or a $(k + n - l, m)$-product.
In each case we get that $\vphi$ is zero on these products, and thus equation (\ref{eq:kms 1}) holds when $n+k \neq l+m$.
Now suppose that $n + k = l + m$.
Without loss of generality we may assume that $m \geq k$ and so $n \geq l$ (otherwise take adjoints).
Let us write
\[
\eta^{\otimes m} = \eta^{\otimes k} \otimes \eta^{\otimes m-k}
\qand
\xi^{\otimes n} = \xi^{\otimes l} \otimes \xi^{\otimes n-l}.
\]
By using (\ref{eq:kms 2}), the left hand side of (\ref{eq:kms 1}) equals to
\begin{align*}
\vphi(t(\xi^{\otimes n}) t(\eta^{\otimes m})^* \cdot t(\ze^{\otimes k}) t(y^{\otimes l})^*)
& =
\vphi(t(\xi^{\otimes n}) t(y^{\otimes l} \otimes \sca{\zeta^{\otimes k}, \eta^{\otimes k}} \eta^{\otimes m-k})^*) \\
& =
e^{-n\be}
\vphi(t(y^{\otimes l} \otimes \sca{\zeta^{\otimes k}, \eta^{\otimes k}} \eta^{\otimes m-k})^* t(\xi^{\otimes n})) \\
& =
e^{-n \be}
\vphi(t(\eta^{\otimes m})^* t(\ze^{\otimes k}) t(y^{\otimes l})^* t(\xi^{\otimes n})).
\end{align*}
Likewise, the right hand side of equation \ref{eq:kms 1} equals to
\begin{align*}
e^{-(n-m)\be} \vphi(t(\ze^{\otimes k}) t(y^{\otimes l})^* \cdot t(\xi^{\otimes n}) t(\eta^{\otimes m})^*)
& =
e^{-(n-m)\be} \vphi(t(\ze^{\otimes k} \otimes \sca{y^{\otimes l}, \xi^{\otimes l}} \xi^{\otimes n-l}) t(\eta^{\otimes m})^*) \\
& = 
e^{-(n-m)\be} e^{-m \be} \vphi( t(\eta^{\otimes m})^* t(\ze^{\otimes k} \otimes \sca{y^{\otimes l}, \xi^{\otimes l}} \xi^{\otimes n-l})) \\
& =
e^{-n \be} \vphi(t(\eta^{\otimes m})^* t(\ze^{\otimes k}) t(y^{\otimes l})^* t(\xi^{\otimes n})).
\end{align*}
Therefore equation \ref{eq:kms 1} is satisfied, and the proof is complete.
\end{proof}

The following proposition allows us to consider just unital C*-correspon\-den\-ces from now on.
When $\phi_X$ is not unital, we define $X^1$ be the space $X$ which becomes a C*-corresponden\-ce over $A^1 = A + \bC$ by extending the operations $\phi_X(1) \xi = \xi = \xi 1$.
Note here that $A^1 = A \oplus \bC$ when $A$ is already unital but $\phi_X(1_A) \neq 1_X$.

\begin{proposition}\label{P:unital}
Let $X$ be a C*-correspondence over $A$.
Then $\vphi$ is a $(\si,\be)$-KMS state for $\T_{X^1}$ if and only if it restricts to a $(\si,\be)$-KMS state on $\T_X$.
\end{proposition}

\begin{proof}
If $\phi_X \colon A \to \L X$ is unital then there is nothing to show.
Otherwise let $(\pi, t)$ be the Fock representation of $X^1$ and notice that $(\pi|_A, t)$ defines a faithful representation of $\T_X$ by the Gauge-Invariant-Uniqueness-Theorem.
Indeed it admits a gauge action and if $\pi(a) \in \psi_t(\K X)$ then
\[
a = p_0 \pi(a) p_0 \in p_0 \psi_t(\K X) p_0 \subseteq p_0\psi_t(\K X^1)p_0 = (0)
\]
for the projection $p_0$ on $A^1 \subset \F X^1$.
Therefore $\T_X \subseteq \T_{X^1}$.
In fact we see that $\T_{X^1}$ is the unitization of $\T_X$.
As the $(\si,\be)$-KMS condition is the same for both $\T_X$ and $\T_{X^1}$ then the equivalence follows by the unitization of states.
Notice here that $\si$ is the same action spatially implemented by the corresponding unitaries.
\end{proof}

\begin{remark}
Henceforth we will assume that the C*-correspondence is unital for our proofs.
However the statements will be given for possibly non-unital C*-correspondences.
\end{remark}

\section{Wold decomposition}

We will consider C*-correspondences that admit a finite decomposition of unit.
By Kasparov's Stabilization Theorem this is equivalent to having $\L X = \K X$.

\begin{definition}
A C*-correspondence $X$ over $A$ will be of \emph{finite rank} if there is a finite collection $x := \{x_1, \dots, x_d\}$ of vectors in the unit ball of $X$ such that $\sum_{i \in [d]} \theta_{x_i, x_i} = 1_{X}$.
\end{definition}

\begin{remark}
For any non-trivial word $\mu = \mu_n \cdots \mu_1 \in \bF_+^d$ we write
\[
x_\mu := x_{\mu_n} \otimes \cdots \otimes x_{\mu_1} \in X^{\otimes n}.
\]
We reserve the notation $x_\mt = 1_A \in X^{\otimes 0}$ when $A$ is unital.
It follows that $X^{\otimes n}$ has finite rank with respect to the collection $\{x_\mu \mid \mu \in \bF_+^d, |\mu| = n\}$.
\end{remark}

\begin{remark}
By construction, $\K(\F X)$ is an ideal in $\T_X$.
When $X$ is of finite rank then we can write every projection $p_k \colon \F X \to A$ with $k \geq 1$ by
\[
p_k = \sum_{|\mu| = k} t(x_\mu) t(x_\mu)^* - \sum_{|\nu| = k+1} t(x_\nu)t(x_\nu)^*,
\]
and thus $p_k \in \T_X$ for all $k \geq 1$.
Moreover we see that
\begin{equation}\label{eq:it id}
p_0 = 1_{\F X} - \sum_{i \in [d]} t(x_i) t(x_i)^* \in \T_{X^1}.
\end{equation}
Hence by using the unitization we have that $p_k \in \T_X^1$ for all $k \in \bZ_+$.
It is straightforward that the $p_n$ commute with all elements in $\T_{X^1}^\ga$, as the latter are supported on the diagonal of $\F X$.
Moreover if $\vphi \in \Eq_\be(\T_{X^1})$ then
\begin{equation}\label{eq:zero}
\vphi(p_k)
=
\sum_{|\mu| = k} \vphi(t(x_\mu) p_0 t(x_\mu)^*) 
 =
\sum_{|\mu| = k} e^{-k \be}\vphi(p_0 \pi(\sca{x_\mu, x_\mu}) p_0) 
 \leq 
\sum_{|\mu| = k} e^{-k \be} \vphi(p_0).
\end{equation}
Therefore if $\vphi(p_0) = 0$ then $\vphi(p_k) = 0$ for all $k \in \bZ_+$.
\end{remark}

This triggers the following definition.
We will be using the same symbol for the extension of a state from $\T_X$ to $\T_{X^{1}}$ from Proposition \ref{P:unital}.

\begin{definition}
Let $X$ be a C*-correspondence of finite rank over $A$.
For $\be > 0$ we define 
\begin{equation}
\Eq_\be^{\fty}(\T_X) := \{\vphi \in \Eq_\be(\T_X) \mid \sum_{k=0}^\infty \vphi(p_k) = 1\}
\AND
\Eq_\be^\infty(\T_X) := \{\vphi \in \Eq_\be(\T_X) \mid \vphi(p_0) = 0\}.
\end{equation}
Likewise we define $\Eq_\be^\infty(\O(J,X))$ and $\Eq_\be^{\fty}(\O(J,X))$ for any $J$-relative Cuntz-Pimsner algebra with respect to the projections $q_J(p_k)$, for the canonical $*$-epimorphism $q_J \colon \T_X \to \O(J,X)$.
\end{definition}

\begin{remark}
We note that $\Eq_\be^\infty(\cdot)$ and $\Eq_\be^{\fty}(\cdot)$ may be trivial in some cases.
For if $q \colon \T_X \to \O(A, X)$ is the canonical $*$-epimorphism, then its kernel $\K(\F X)$ is generated by $p_0$.
This automatically implies that $\Eq_\be^{\fty}(\O(A,X)) = \mt$.
As another example, in Proposition \ref{P:hx} we will show that $\Eq_\be^\infty(\T_X) = \mt$ for sufficiently large $\be$.
\end{remark}

Notice that $\K( \F X) \subseteq \T_X$ and thus it inherits the gauge action by restriction.
Therefore we also get equilibrium states for $\K(\F X)$.
Let us give an alternative proof of \cite[Proposition 2.4]{LacNes04} of the Wold decomposition into a finite and an infinite part.

\begin{theorem}\label{T:con}
Let $X$ be a C*-correspondence of finite rank over $A$ and let $\be > 0$.
Then for any $\vphi \in \Eq_\be(\T_X)$ we have:
\begin{enumerate}[leftmargin=30pt, itemindent=0pt, itemsep=0pt]
\item $\vphi \in \Eq_\be^{\fty}(\T_X)$ if and only if $\vphi|_{\K(\F X)} \in \Eq_\be(\K(\F X))$;
\item $\vphi \in \Eq_\be^\infty(\T_X)$ if and only if $\vphi|_{\K(\F X)} = 0$ if and only if $\vphi$ factors through $\O(A,X)$;
\item There are unique $\vphi_{\fty} \in \Eq_\be^{\fty}(\T_X)$ and $\vphi_\infty \in \Eq_\be^\infty(\T_X)$ such that
\[
\vphi = \la \vphi_{\fty} + (1-\la) \vphi_\infty, \FOR \la := \sum_{k=0}^\infty \vphi(p_k).
\]
\end{enumerate}
\end{theorem}

\begin{proof}
By positivity we have $\sum_{k=0}^n \vphi(p_k) \leq 1$ for every $n \in \bN$ and so $\sum_{k=0}^\infty \vphi(p_k) < \infty$.
Moreover equation (\ref{eq:zero}) implies that $\sum_{k=0}^\infty \vphi(p_k) = 0$ if and only if $p_0 = 0$.
Now both items (i) and (ii) follow by using $(\sum_{k=0}^n p_k)_n$ as a contractive approximate identity of $\K(\F X)$.
For item (iii) use the Wold decomposition with respect to the quotient map $\T_X \to \O(A,X)$ and the KMS-condition on $\vphi$ to define the positive functional $\psi_{\fty} \colon \T_X \to \bC$ by
\[
\psi_{\fty}(f) 
:= \sum_{k, \ell=0}^{\infty} \vphi(p_k f p_\ell)
= \sum_{k, \ell=0}^{\infty} \vphi(p_k f p_\ell p_k)
=  \sum_{k=0}^{\infty} \vphi(p_k f p_k),
\]
and let
\[
\psi_{\infty}(f) := \vphi(f) - \psi_{\fty}(f)
\]
for $f \in \T_X$.
If $\psi_{\fty} \neq 0$ then $\la := \sum_{k=0}^\infty \vphi(p_k) = \nor{\psi_{\fty}}$, and so $\nor{\psi_{\infty}} = 1 - \la$.
Hence if $\la \in (0,1)$ we obtain the states
\[
\vphi_{\fty} := \la^{-1} \psi_{\fty} 
\qand
\vphi_{\infty} := (1 - \la)^{-1} \psi_{\infty}.
\]
Since there is a unique extension of a state from $\K(\F X)$ to $\T_X$ we get uniqueness of this decomposition.
As $\vphi_\infty(p_0) = 0$ it remains to show that $\vphi_{\fty}$ and $\vphi_{\infty}$ satisfy the KMS-condition.
Equivalently that $\psi_{\fty}$ does so.
By definition we have that $\psi_{\fty}(t(\xi^{\otimes n}) t(\eta^{\otimes m})^*) = 0$ when $n \neq m$.
Now if $n = m$ then we get
\begin{align*}
t(\eta^{\otimes n})^* p_k t(\xi^{\otimes n})
&=
\begin{cases}
p_{k-n} t(\eta^{\otimes n})^* t(\xi^{\otimes n}) p_{k-n} & \text{ if } k \geq n,\\
0 & \text{ otherwise}.
\end{cases}
\end{align*}
Therefore for all $n,m \in \bZ_+$ we obtain
\begin{align*}
\psi_{\fty}(t(\xi^{\otimes n}) t(\eta^{\otimes m})^*)
& =
\de_{n,m} \sum_{k=0}^\infty \vphi(p_k t(\xi^{\otimes n}) t(\eta^{\otimes m})^* p_k) \\
& =
\de_{n,m} e^{-n\be} \sum_{k=0}^\infty \vphi(t(\eta^{\otimes m})^* p_k t(\xi^{\otimes n})) \\
& =
\de_{n,m} e^{-n\be} \sum_{k \geq n}^\infty \vphi(p_{k-n} t(\eta^{\otimes m})^* t(\xi^{\otimes n})  p_{k-n}) \\
& =
\de_{n,m} e^{-n \be} \sum_{k =0}^\infty \vphi(p_k t(\eta^{\otimes m})^* t(\xi^{\otimes n}) p_k) 
 =
\de_{n,m} \psi_{\fty}(t(\eta^{\otimes m})^* t(\xi^{\otimes n}))
\end{align*}
and thus $\psi_{\fty}$ satisfies equation (\ref{eq:kms 2}).
\end{proof}

\begin{remark}
The convex decomposition is not weak*-continuous.
For example, for fixed $\vphi_\infty \in \Eq_\be^\infty(\T_X)$ and $\vphi_{\fty} \in \Eq_\be^{\fty}(\T_X)$, the states $\vphi_n = n^{-1}\vphi_{\fty} + (1 - n^{-1}) \vphi_\infty$ weak*-converge to $\vphi_\infty$.
However the infinite and the finite parts of all $\vphi_n$ stay the same.
\end{remark}

\section{Entropy}

We start with a remark that ensures that the quantities we are to introduce are independent of the choice of the unit decomposition.

\begin{remark}\label{R:ind un dec}
Let $\{x_1, \dots, x_d\}$ and $\{y_1, \dots, y_{d'}\}$ be two unit decompositions.
Then for $\tau \in \Tr(A)$ we get that
\begin{align*}
\sum_{|\mu| = k} \tau(\sca{x_\mu, x_\mu})
& = 
\sum_{|\mu| = k} \sum_{|\nu| = k} \tau( \sca{x_\mu, y_\nu} \sca{y_\nu, x_\mu}) 
 = 
\sum_{|\nu| = k} \sum_{|\mu| = k} \tau( \sca{y_\nu, x_\mu} \sca{x_\mu, y_\nu}) 
 =
\sum_{|\nu| = k} \tau( \sca{y_\nu, y_\nu}).
\end{align*}
That is, the value $\sum_{|\mu| = k} \tau(\sca{x_\mu, x_\mu})$ is independent of the unit decomposition.
\end{remark}

\begin{definition}\label{D:ctau}
Let $X$ be a C*-correspondence of finite rank over $A$ with respect to $\{x_1, \dots, x_d\}$ and let $\be > 0$.
For any $\tau \in \Tr(A)$ we define the formal series
\begin{equation}\label{eq:ctau}
c_{\tau,\be} := \sum_{k=0}^\infty e^{-k \be} \sum_{|\mu| = k} \tau(\sca{x_\mu, x_\mu}).
\end{equation}
We write $\Tr_\be(A) := \{ \tau \in \Tr(A) \mid c_{\tau,\be} < \infty\}$.
\end{definition}

Remark \ref{R:ind un dec} implies that $c_{\tau,\be}$ and $\Tr_\be(A)$ do not depend on the unit decomposition.
Since $x_\mt = 1_A$ then we see that $c_{\tau,\be} \geq 1$.
Moreover the set $\Tr_\be(A)$ is convex.
On the other extreme we have the notion of averages.

\begin{definition}
Let $X$ be a C*-correspondence of finite rank over $A$ with respect to $\{x_1, \dots, x_d\}$ and let $\be > 0$.
Let $\Avt_\be(A)$ be the set of the tracial states $\tau$ of $A$ that satisfy
\[
\tau(a) = e^{-\be} \sum_{i \in [d]} \tau(\sca{x_i, a x_i}) \foral a \in A.
\]
\end{definition}

As in Remark \ref{R:ind un dec} we have that $\Avt_\be(A)$ does not depend on the decomposition $\{x_1, \dots, x_d\}$ of the unit.
The next proposition marks that $\Tr_\be(A) \cap \Avt_\be(A) = \mt$.

\begin{proposition}\label{P:hav}
Let $X$ be a C*-correspondence of finite rank over $A$ and let $\be > 0$.
If $\tau \in \Avt_\be(A)$ then $c_{\tau,\be} = \infty$.
\end{proposition}

\begin{proof}
Induction yields an average formula for all words of length $k$, i.e.,
\begin{equation}
\tau(a) = e^{-k \be} \sum_{|\mu| = k} \tau(\sca{x_\mu, a x_\mu}) \foral a \in A, k \in \bN. \qedhere
\end{equation}
\end{proof}

The root test implies a notion of entropy for $\tau \in \Tr(A)$ that connects with convergence of $c_{\tau,\be}$.
We are going to use also two notions of entropy for $X$.
As we use entropy for convergence of $c_{\tau, \be}$ we set $\limsup_k k^{-1} \log a_k = 0$ if $a_k = 0$ eventually.

\begin{definition}
Let $X$ be a C*-correspondence of finite rank over $A$ with respect to a unit decomposition $x = \{x_1, \dots, x_d\}$.\\
(1) The \emph{entropy} of a $\tau \in \Tr(A)$ is given by
\[
h_X^\tau := \limsup_k \frac{1}{k} \log \sum_{|\mu| = k} \tau(\sca{x_\mu, x_\mu}).
\]
(2) The \emph{entropy} of $x = \{x_1, \dots, x_d\}$ is defined by
\[
h_X^x := \limsup_k \frac{1}{k} \log \nor{\sum_{|\mu| = k} \sca{x_\mu, x_\mu}}_{A}.
\]
(3) The \emph{strong entropy} of $X$ is defined by
\[
h_X^s := \inf\{ h_X^x \mid x =\{x_1, \dots, x_d\} \textup{ is a unit decomposition for $X$}\}.
\]
(4) The \emph{entropy} of $X$ is defined by
\[
h_X := \inf\{ \be>0 \mid \Eq_\be(\T_X) \neq \mt\}.
\]
\end{definition}

\begin{remark}
Due to remark \ref{R:ind un dec}, the entropy $h_X^\tau$ is independent of the unit decomposition.
Likewise $h_X^s = h_X^x$ for any unit decomposition when $A$ is abelian.
Furthermore the $\limsup$ in $h_X^x$ is actually the limit of a decreasing sequence.
Indeed for $k_1, k_2 \in \bN$ with $k_1 + k_2 = k$ we get
\begin{align*}
\sum_{\mu = k} \sca{x_\mu, x_\mu}
& =
\sum_{|\nu_1| = k_1, |\nu_2| = k_2} \sca{x_{\nu_1} \otimes x_{\nu_2}, x_{\nu_1} \otimes x_{\nu_2}} \\
& =
\sum_{|\nu_1| = k_1} \sca{x_{\nu_1}, (\sum_{|\nu_2| = k_2}\sca{x_{\nu_2},x_{\nu_2}}) (x_{\nu_1})} 
 \leq
\nor{\sum_{|\nu_2| = k_2} \sca{x_{\nu_2},x_{\nu_2}}}_{A} \sum_{|\nu_1| = k_1} \sca{x_{\nu_1}, x_{\nu_1}}.
\end{align*}
Therefore the sequence $\nor{\sum_{|\mu| = k} \sca{x_\mu, x_\mu}}_A$ is submultiplicative.
\end{remark}

We close this section with a connection between entropies and $\Eq_\be(\T_X)$.
We shall see later that Proposition \ref{P:hx}(iv) can follow from the complete paramet\-ri\-za\-tion of $\Eq_\be^{\fty}(\T_X)$ and $\Eq_\be^\infty(\T_X)$.
Item (v) below is basically a rewording of \cite[Theorem 2.5 and Corollary 2.6]{PWY00}.

\begin{proposition}\label{P:hx}
Let $X$ be a C*-correspondence of finite rank over $A$ and let $\be > 0$.
\begin{enumerate}[leftmargin=30pt, itemindent=0pt, itemsep=1pt]
\item If $\tau \in \Tr_\be(A) \cup \Avt_\be(A)$ then $h_X^\tau \leq \be$.
\item For every $\tau \in \Tr(A)$ we have that $h_X^\tau \leq h_X^s \leq \log d$.
\item $\Tr_\be(A) = \Tr(A)$ whenever $\be > h_X^s$.
\item $\Eq_\be^\infty(\T_X) = \mt$ and $\Eq_\be(\T_X) = \Eq_\be^{\fty}(\T_X)$ whenever $\be > h_X^s$.
\item If $A$ is abelian and $h_X^s > 0$ then $\Avt_{h_X^s}(A) \neq \mt$.
\end{enumerate}
\end{proposition}

\begin{proof}
Let $x=\{x_1, \dots, x_d\}$ be a decomposition of the unit.
Item (i) follows directly from the root test when $\tau \in \Tr_\be(A)$ and from Proposition \ref{P:hav} when $\tau \in \Avt_\be(A)$.
Moreover it is straightforward to check that if $\tau \in \Tr(A)$ then
\begin{align*}
\sum_{|\mu| = k} \tau(\sca{x_\mu, x_\mu}) 
\leq
\nor{\sum_{|\mu| = k} \sca{x_\mu, x_\mu}}_{A} 
\leq d^k.
\end{align*}
As the left hand side does not depend on $x$, taking infimum over all unit decompositions gives that $h_X^\tau \leq h_X^s \leq \log d$.
For item (iii) suppose that $\be \in (h_X^s, \infty)$ and choose a unit decomposition $x = \{x_1, \dots, x_d\}$ such that $h_X^s \leq h_X^x < \be$.
Then for any $\tau \in \Tr(A)$ we have that
\[
\limsup_k (e^{-k \be} \sum_{|\mu| = k} \tau(\sca{x_\mu,x_\mu})^{1/k}
\leq
e^{-\be} e^{h_X^x} < 1
\]
giving that $c_{\tau,\be} < \infty$.
For item (iv), if $\vphi \in \Eq_\be^\infty(\T_X)$ then $\vphi(p_0) = 0$ and thus $\vphi(p_k) = 0$ for all $k \in \bZ_+$ by equation (\ref{eq:zero}).
But then the KMS-condition yields
\begin{align*}
1 
& = 
\sum_{|\mu| = k} \vphi( t(x_\mu) t(x_\mu)^*)
 = 
e^{-k \be} \sum_{|\mu| = k} \vphi \pi(\sca{x_\mu, x_\mu})
\leq
e^{-k \be} \nor{ \sum_{|\mu| = k} \sca{x_\mu, x_\mu}}_{A},
\end{align*}
as $\vphi\pi \in \Tr(A)$.
Hence $\be \leq k^{-1} \log \nor{ \sum_{|\mu| = k} \sca{x_\mu, x_\mu}}_{A}$ for all $k \in \bZ_+$, which gives that $\be \leq h_X^x$.
Taking the infimum over all unit decompositions yields $\be \leq h_X^s$.
Finally, if $A$ is abelian then the arguments of \cite[Theorem 2.5 and Corollary 2.6]{PWY00} apply to give that $\Avt_{h_X^s}(A) \neq \mt$.
In short let the map 
\[
\psi \colon A \to A
\; \textup{ such that } \;
\psi(a) = \sum_{i \in [d]} \sca{x_i, a x_i}.
\]
As $\psi$ is a positive map we have $\nor{\psi^k} = \nor{\psi^k(1)} = \nor{\sum_{|\mu| = k} \sca{x_\mu, x_\mu}}_A$ for all $k \in \bN$.
Therefore
\[
h_X^s = \lim_k \log \nor{\psi^k}^{1/k} = \log \la_\psi,
\]
where $\la_\psi$ is the spectral radius of $\psi$.
Then \cite[Theorem 2.5 and Corollary 2.6]{PWY00} imply that $\la_\psi$ is an eigenvalue of the adjoint of $\psi$ on the states of $A$, i.e., there is $\tau \in \S(A)$ such that $\tau \psi = \la_\psi \tau \psi$ (the fullness condition of \cite{PWY00} is not required here).
Hence $\tau$ gives a tracial state in $\Avt_{h_X^s}(A)$.
\end{proof}

\section{The finite part of the equilibrium states}

In this section we parametrize the states in $\Eq_\be^{\fty}(\T_X)$ for $\be > 0$ and consequently we show how this induces a parametrization for all $\Eq_\be^{\fty}(\O(J,X))$.
Passing from $\Tr_\be(A)$ to $\Eq_\be^{\fty}(\T_X)$ uses essentially \cite[proof of Theorem 2.1]{LacNes04}.
Showing that this construction is a bijection generalizes the corresponding arguments from \cite{LacRae10}.

\begin{theorem}\label{T:para}
Let $X$ be a C*-correspondence of finite rank over $A$ and let $\be > 0$.
Then there is a bijection
\[
\Phi \colon \Tr_\be(A) \to \Eq_\be^{\fty}(\T_X) \text{ such that } \Phi(\tau)(p_0) = c_{\tau,\be}^{-1}.
\]
If $x = \{x_1, \dots, x_d\}$ is a decomposition of the unit then $\Phi$ is given by
\begin{equation}\label{eq:kms 3}
\Phi(\tau) (t(\xi^{\otimes n}) t(\xi^{\otimes m})^*)
=
\de_{n,m} c_{\tau,\be}^{-1} \sum_{\mu \in \bF_+^d} e^{-(|\mu|+n) \be} \tau(\sca{\eta^{\otimes m} \otimes x_\mu, \xi^{\otimes n} \otimes x_\mu})
\end{equation}
for all $\xi^{\otimes n} \in X^{\otimes n}$ and $\eta^{\otimes m} \in X^{\otimes m}$.
If, in addition, $\Eq_\be^{\fty}(\T_X)$ is weak*-closed then $\Phi$ is a weak*-homeomorphism between weak*-compact sets.
\end{theorem}

\begin{proof}
Equation (\ref{eq:kms 3}) is independent of the unit decomposition for $\tau \in \Tr_\be(A)$.
Indeed let $y = \{y_1, \dots, y_{d'}\}$ be a second decomposition.
If $n \neq m$ then there is nothing to show.
For $n = m$ we directly verify that
\begin{align*}
\sum_{|\mu| = k} \tau(\sca{\eta^{\otimes n} \otimes x_\mu, \xi^{\otimes n} \otimes x_\mu})  
& = 
\sum_{|\mu| = k} \sum_{|\nu| = k} 
\tau(\sca{\eta^{\otimes n} \otimes x_\mu, \xi^{\otimes n} \otimes \theta_{y_\nu, y_\nu} x_\mu} \\
& =
\sum_{|\mu| = k} \sum_{|\nu| = k}
\tau(\sca{x_\mu, \sca{\eta^{\otimes n}, \xi^{\otimes n}} y_\nu} \sca{y_\nu, x_\mu}) \\
& =
\sum_{|\nu| = k} \sum_{|\mu| = k} 
\tau(\sca{y_\nu, x_\mu} \sca{x_\mu, \sca{\eta^{\otimes n}, \xi^{\otimes n}} y_\nu}) \\
& =
\sum_{|\nu| = k} \tau(\sca{\eta^{\otimes n} \otimes y_\nu, \xi^{\otimes n} \otimes y_\nu}).
\end{align*}

Now we proceed to the construction of $\Phi$.
First we show that $\vphi_\tau \equiv \Phi(\tau)$ exists and is in $\Eq_\be^{\fty}(\T_X)$ when $\tau \in \Tr_\be(A)$.
Let $(H_\tau, x_\tau, \rho_\tau)$ be the GNS-representation associated to $\tau$ and consider  the induced pair $(\rho, v) := (\pi \otimes I, t \otimes I)$ for $\T_X$ acting on $\F X \otimes_{\rho_{\tau}} H_\tau$.
For any word $\mu$ on the $d$ symbols define the positive vector state $\vphi_{\tau, \mu}$ of $\T_X$ be given by
\[
\vphi_{\tau, \mu}(f) = \sca{x_\mu \otimes x_\tau, (\rho \times v)(f) x_\mu \otimes x_\tau}_H \text{ for } f \in \T_X.
\]
We then define
\[
\vphi_\tau := c_{\tau,\be}^{-1} \sum_{k=0}^\infty e^{-k\be} \sum_{|\mu| = k} \vphi_{\tau, \mu}.
\]
To see that it is indeed well defined (and a state) on $\T_X$ first check that
\begin{align*}
c_{\tau,\be}^{-1} \sum_{k=0}^\infty e^{-k\be} \sum_{|\mu| = k} \vphi_{\tau, \mu}(\pi(1_A))
& =
c_{\tau,\be}^{-1} \sum_{k=0}^\infty e^{-k\be} \sum_{|\mu| = k} \tau(\sca{x_\mu, x_\mu})
=
1.
\end{align*}
Likewise we have $\vphi_{\tau, \mu}(f) \leq \nor{f} \vphi_{\tau, \mu}(\pi(1_A))$ for all $0 \leq f \in \T_X$, and thus
\begin{align*}
c_{\tau,\be}^{-1} \sum_{k=0}^\infty e^{-k\be} \sum_{|\mu| = k} \vphi_{\tau, \mu}(f)
& \leq
c_{\tau,\be}^{-1} \sum_{k=0}^\infty e^{-k\be} \sum_{|\mu| = k} \nor{f} \cdot \vphi_{\tau, \mu}(\pi(1_A)) 
=
\nor{f}.
\end{align*}
Next we show that $\vphi_\tau$ satisfies equation (\ref{eq:kms 3}).
If $n \neq m$ then for all $\mu$ we get that
\[
\vphi_{\tau, \mu}(t(\xi^{\otimes n}) t(\eta^{\otimes m})^*)
=
\tau(\sca{t(\xi^{\otimes n})^* x_\mu, t(\eta^{\otimes m})^* x_\mu}_{\F X})
= 0,
\]
and thus $\vphi_\tau(t(\xi^{\otimes n}) t(\eta^{\otimes m})^*) = 0$.
If $n = m$ and $k \geq n$, then for all $x_\mu$ with $|\mu| < n$ we get that
\[
\vphi_{\tau, \mu}(t(\xi^{\otimes n}) t(\eta^{\otimes n})^*)
=
\tau(\sca{t(\xi^{\otimes n})^* x_\mu, t(\eta^{\otimes n})^* x_\mu}_{\F X}) = 0.
\]
On the other hand if $|\mu| = k \geq n$ then recall that $\sum_{|\mu| = k} t(x_\mu) t(x_\mu)^*$ acts as a unit on $t(X^{\otimes \ell})$ for all $\ell \geq k$.
Thus we get
\begin{align*}
\sum_{|\mu| = k} \vphi_{\tau, \mu}(t(\xi^{\otimes n}) t(\eta^{\otimes n})^*)
& = 
\sum_{|\mu| = k} \tau \pi^{-1}( t(x_\mu)^* t(\xi^{\otimes n}) t(\eta^{\otimes n})^* t(x_\mu) ) \\
& =
\sum_{|\mu| = k} \sum_{|\nu| = k - n } \tau \pi^{-1} ( t(x_\mu)^* t(\xi^{\otimes n}) t(x_\nu) t(x_\nu)^* t(\eta^{\otimes n})^* t(x_\mu)) \\
& = 
\sum_{|\nu| = k - n } \sum_{|\mu| = k} \tau \pi^{-1} (t(x_\nu)^* t(\eta^{\otimes n})^* t(x_\mu)  t(x_\mu)^* t(\xi^{\otimes n}) t(x_\nu)) \\
& = 
\sum_{|\nu| = k - n }\tau \pi^{-1} (t(x_\nu)^* t(\eta^{\otimes n})^* t(\xi^{\otimes n}) t(x_\nu)) \\
& = 
\sum_{|\nu| = k - n} \vphi_{\tau, \nu}(t(\eta^{\otimes n})^* t(\xi^{\otimes n}) ).
\end{align*} 
Hence we obtain
\begin{align*}
\vphi_\tau(t(\xi^{\otimes n}) t(\eta^{\otimes n})^*)
& =
c_{\tau,\be}^{-1} \sum_{k= 0}^\infty e^{-k \be} \sum_{|\mu| = k} \vphi_{\tau, \mu}(t(\xi^{\otimes n}) t(\eta^{\otimes n})^*) \\
& =
c_{\tau,\be}^{-1} \sum_{k= n}^\infty e^{-k \be} \sum_{|\mu| = k - n} \vphi_{\tau, \mu}(t(\eta^{\otimes n})^* t(\xi^{\otimes n})) \\
& =
c_{\tau,\be}^{-1} \sum_{k= 0}^\infty e^{-(k+n) \be} \sum_{|\mu| = k} \vphi_{\tau, \mu}(t(\eta^{\otimes n})^* t(\xi^{\otimes n})) \\
& =
c_{\tau,\be}^{-1} \sum_{k= 0}^\infty e^{-(k+n) \be} \sum_{|\mu| = k} \tau(\sca{\eta^{\otimes m} \otimes x_\mu, \xi^{\otimes n} \otimes x_\mu}).
\end{align*}
We verify that $\vphi_\tau \in \Eq_\be(\T_X)$ by using Proposition \ref{P:char}.
By definition we have that if $n \neq m$ then $\vphi_\tau(t(\xi^{\otimes n}) t(\eta^{\otimes m})^*) = 0$.
Now if $n = m$ then we directly compute
\begin{align*}
\vphi_\tau(t(\xi^{\otimes n}) t(\eta^{\otimes n})^*)
& =
c_{\tau,\be}^{-1} \sum_{k= 0}^\infty e^{-(k+n) \be} \sum_{|\mu| = k} \vphi_{\tau, \mu}(t(\eta^{\otimes n})^* t(\xi^{\otimes n})) \\
& = 
e^{-n\be} c_{\tau,\be}^{-1} \sum_{k= 0}^\infty e^{-k \be} \sum_{|\mu| = k} \vphi_{\tau, \mu}(t(\eta^{\otimes n})^* t(\xi^{\otimes n})) \\
& =
e^{-n \be} \vphi_\tau(t(\eta^{\otimes n})^* t(\xi^{\otimes n})).
\end{align*}
In order to show that $\vphi_\tau \in \Eq_\be^{\fty}(\T_X)$ we compute
\begin{align*}
\sum_{k=0}^n \vphi_\tau(p_k)
& =
1 - \sum_{|\nu| = n+1} \vphi_\tau(t(x_\nu) t(x_\nu)^*) \\
& =
1 - c_{\tau,\be}^{-1} \sum_{k=0}^\infty e^{-(n+1 + k)\be} \sum_{|\nu| = n+1} \sum_{|\mu| = k} \tau(\sca{x_\nu \otimes x_\mu, x_\nu \otimes x_\mu}) \\
& =
c_{\tau,\be}^{-1} \sum_{k=0}^n e^{-k \be} \sum_{|\mu| = k} \tau(\sca{x_\mu, x_\mu}),
\end{align*}
Applying for $n=0$ yields $\vphi_\tau(p_0) = c_{\tau,\be}^{-1}$.
Taking the limit $n \rightarrow \infty$ gives $\sum_{k=0}^\infty \vphi_\tau(p_k) = c_{\tau,\be}^{-1} c_{\tau, \be} = 1$, and so $\vphi_\tau \in \Eq_\be^{\fty}(\T_X)$.

Secondly we show that this correspondence is surjective.
To this end fix $\vphi \in \Eq_\be^{\fty}(\T_X)$.
Inequality (\ref{eq:zero}) gives that $\vphi(p_0) \neq 0$ and thus we can define the state $\tau_\vphi$ on $A$ by
\[
\tau_{\vphi}(a) := \vphi(p_0)^{-1} \vphi(p_0 \pi(a) p_0) \foral a \in A.
\]
Moreover $\tau_\vphi$ is in $\Tr(A)$ since
\begin{align*}
\vphi(p_0) \tau_\vphi(ab)
& =
\vphi(p_0 \pi(a) \pi(b) p_0) 
 =
\vphi(\pi(b) p_0 \pi(a)) 
 =
\vphi(p_0 \pi(b) \pi(a) p_0) 
 =
\vphi(p_0) \tau(ba),
\end{align*}
where we used that $p_0 \in \pi(A)'$ and $\si_{i\be}(\pi(a)) = \pi(a)$.
In order to show that $\tau_\vphi \in \Tr_\be(A)$ it suffices to show that
\[
\vphi(p_0)^{-1} = \sum_{k=0}^\infty e^{-k\be} \sum_{|\mu| = k} \tau_\vphi(\sca{x_\mu, x_\mu}).
\]
However a direct computation yields
\begin{align*}
\vphi(p_0) \sum_{|\mu| = k} \tau_\vphi(\sca{x_\mu, x_\mu})
& =
\sum_{|\mu| = k} \vphi(p_0 t(x_\mu)^* t(x_\mu) p_0) 
 = 
e^{k \be} \sum_{|\mu| = k} \vphi(t(x_\mu) p_0 t(x_\mu)^*) 
 =
e^{k\be} \vphi(p_k).
\end{align*}
Since $\vphi \in \Eq_\be^{\fty}(\T_X)$ we have
\[
\sum_{k=0}^\infty e^{-k \be} \sum_{|\mu| = k} \tau_{\vphi}(\sca{x_\mu, x_\mu})
=
\vphi(p_0)^{-1} \sum_{k=0}^\infty \vphi(p_k)
=
\vphi(p_0)^{-1}.
\]
Surjectivity now follows by showing that $\vphi = \Phi(\tau_\vphi)$.
Since both are $(\si, \be)$-KMS states, by Proposition \ref{P:char} it suffices to show that they agree on $\pi(A)$.
Since $\vphi$ is implemented by a state on $\K(\F X)$, for every $a \in A$ we have that
\begin{align*}
\vphi(\pi(a))
& = 
\lim_m \sum_{k, l=0}^m \vphi(p_k \pi(a) p_l) 
 = 
\lim_m \sum_{k = 0}^m \vphi(p_k \pi(a)) \\
& = 
\lim_m \sum_{k=0}^m \sum_{|\mu| = k} \vphi\big(t(x_\mu) p_0 t(x_\mu)^* \pi(a) \big) \\
& =
\lim_m \sum_{k=0}^m \sum_{|\mu| = m} e^{-k\be} \vphi\big(p_0 t(x_\mu)^* \pi(a) t(x_\mu) p_0\big) \\
& = 
c_{\tau_\vphi, \be}^{-1} \sum_{k=0}^\infty e^{-k\be} \sum_{|\mu| = k} \tau_\vphi(\sca{ x_\mu, a x_\mu)})
=
\Phi(\tau_\vphi)(\pi(a)).
\end{align*}
To show injectivity let $\tau \in \Tr_\be(A)$ and use the vector states $\vphi_{\tau, \mu}$ to get
\[
\sum_{|\mu| = k} \vphi_{\tau, \mu}(p_0 \pi(a) p_0)
= 
\begin{cases}
\tau(a) & \text{ if } k = 0,\\
0 & \text{ otherwise}.
\end{cases}
\]
Therefore we have
\[
\vphi_\tau(p_0 \pi(a) p_0)
=
c_{\tau,\be}^{-1} \sum_{k=0}^\infty e^{-k \be} \sum_{|\mu| = k} \vphi_{\tau, \mu}(p_0 \pi(a) p_0)
=
\vphi_\tau(p_0) \tau(a)
\]
showing that $\tau$ is uniquely identified by $\vphi_\tau$.

Finally we show that $\Phi^{-1}$ is weak*-continuous.
To this end let $\vphi_j, \vphi \in \Eq_\be^{\fty}(\T_X)$ such that $\vphi_j \longrightarrow \vphi$ in the weak*-topology.
As $\vphi_j(p_0) \neq 0$ and $\vphi(p_0) \neq 0$ we get that $\tau_{\vphi_j}(a) \longrightarrow \tau_{\vphi}(a)$ for all $a \in A$.
Hence $\Phi^{-1}$ is a continuous bijection from the compact space $\Eq_\be^{\fty}(\T_X)$ onto the Hausdorff space $\Tr_\be(A)$.
Thus if $\Eq_\be^{\fty}(\T_X)$ is weak*-closed then $\Phi$ is a homeomorphism.
\end{proof}

\begin{corollary}\label{C:ext}
Let $\Phi \colon \T_\be(A) \to \Eq_\be^{\fty}(\T_X)$ be the map of Theorem \ref{T:para}.
If $\tau = \la \tau_1 + (1- \la) \tau_2$ for $\tau_1, \tau _2 \in \Tr_\be(A)$ and $\la \in [0,1]$ then
\[
\Phi(\tau) = \la \frac{c_{\tau_1,\be}}{c_{\tau,\be}} \Phi(\tau_1) + (1 - \la) \frac{c_{\tau_2, \be}}{c_{\tau,\be}} \Phi(\tau_2).
\]
Conversely, if $\vphi = \la \vphi_1 + (1 - \la) \vphi_2$ for $\vphi_1, \vphi_2 \in \Eq_\be^{\fty}(\T_X)$ and $\la \in [0,1]$ then
\[
\Phi^{-1}(\vphi) 
= 
\la \frac{\vphi_1(p_0)}{\vphi(p_0)} \Phi^{-1}(\vphi_1) + (1 - \la) \frac{\vphi_2(p_0)}{\vphi(p_0)} \Phi^{-1}(\vphi_2).
\]
Consequently, the parametrization $\Phi$ fixes the extreme points.
\end{corollary}

\begin{proof}
For the forward direction it is clear that $c_{\tau,\be} = \la c_{\tau_1,\be} + (1 - \la) c_{\tau_2,\be}$.
Therefore the state
\[
\vphi' =  \la \frac{c_{\tau_1,\be}}{c_{\tau,\be}} \Phi(\tau_1) + (1 - \la) \frac{c_{\tau_2,\be}}{c_{\tau,\be}} \Phi(\tau_2)
\]
is in $\Eq_\be^{\fty}(\T_X)$ as a convex combination of states in $\Eq_\be^{\fty}(\T_X)$.
Now for every $a \in A$ we have
\begin{align*}
\vphi_\tau(\pi(a))
& =
\la c_{\tau,\be}^{-1} \sum_{k=0}^\infty e^{-k \be} \sum_{|\mu| = k} \tau_1 (\sca{x_\mu, a x_\mu}) 
+ 
(1 - \la) c_{\tau,\be}^{-1} \sum_{k=0}^\infty e^{-k \be} \sum_{|\mu| = k} \tau_2 (\sca{x_\mu, a x_\mu}) \\
& = 
\la \frac{c_{\tau_1,\be}}{c_{\tau,\be}} \Phi(\tau_1)(\pi(a)) + (1-\la) \frac{c_{\tau_2,\be}}{c_{\tau,\be}} \Phi(\tau_2)(\pi(a))
=
\vphi'(\pi(a)).
\end{align*}
As both $\vphi'$ and $\vphi_\tau$ are in $\Eq_\be(\T_X)$, Proposition \ref{P:char} implies that they are equal.
For the converse set $\tau_1 = \Phi^{-1}(\vphi_1)$, $\tau_2 = \Phi^{-1}(\vphi_2)$ and $\tau = \Phi^{-1}(\vphi)$.
Then by construction, for every $a \in A$ we get that
\begin{align*}
\vphi(p_0) \tau(a)
& = 
\la \vphi_1(p_0 \pi(a) p_0) + (1-\la) \vphi_2(p_0 \pi(a) p_0) 
 =
\la \vphi_1(p_0) \tau_1(a) + (1-\la) \vphi_2(p_0) \tau_2(a).
\end{align*}
Applying for $a = 1_A$ also gives that $\vphi(p_0) = \la_1 \vphi_1(p_0) + \la_2 \vphi_2(p_0)$.
Finally to see that $\Phi$ fixes the extreme points just notice that the $c$-constants are all non-zero and the equations for $\Phi(\tau)$ and $\Phi^{-1}(\vphi)$ are convex combinations of states.
\end{proof}

\begin{corollary}\label{C:com ent}
If $X$ is a C*-correspondence of finite rank over $A$ then $h_X \leq \max\{0, h_X^s\}$.
\end{corollary}

\begin{proof}
If $\be > h_X^s$ then Proposition \ref{P:hx}(iii) implies that $\Tr_\be(A) = \Tr(A)$.
Therefore Theorem \ref{T:para} gives that $\Eq_\be(\T_X) \neq \mt$ and so $h_X \leq \be$.
\end{proof}

The gauge action of $\T_X$ is inherited by the $J$-relative Cuntz-Pimsner algebras.
Thus we can use the previous parametrization for their equilibrium states.
For convenience let us write here $(\rho, v) = (q_J \pi, q_J t)$ for the faithful representation of $\O(J, X)$ where $q_J \colon \T_X \to \O(J,X)$ is the canonical quotient map.
Hence $\ker q_J$ is the ideal generated by $\pi(a) p_0$ for all $a \in J$ when $p_0 \in \T_X$.
We will write simply $q$ when $J = A$.

\begin{theorem}\label{T:para rel}
Let $X$ be a C*-correspondence of finite rank over $A$ and let $\be > 0$.
Suppose that $J \subseteq \phi_X^{-1}(\K X)$.
Then there is a bijection 
\[
\Phi \colon \{\tau \in \Tr_\be(A) \mid \tau|_J = 0\} \to \Eq_\be^{\fty}(\O(J,X)).
\]
If $x = \{x_1, \dots, x_d\}$ is a decomposition of the unit then $\Phi$ is given by
\begin{equation}
\Phi(\tau)(v(\xi^{\otimes n}) v(\xi^{\otimes m})^*)
=
\de_{n,m} c_{\tau,\be}^{-1} \sum_{k=0}^\infty e^{-(k+n) \be} 
\sum_{|\mu| = k} \tau(\sca{\eta^{\otimes m} \otimes x_\mu, \xi^{\otimes n} \otimes x_\mu})
\end{equation}
for all $\xi^{\otimes n} \in X^{\otimes n}$ and $\eta^{\otimes m} \in X^{\otimes m}$.
Moreover $\Phi$ satisfies the convex combination of Corollary \ref{C:ext} and thus it preserves extreme points.
If, in addition, $\Eq_\be^{\fty}(\O(J,X))$ is weak*-closed then $\Phi$ is a weak*-homeomorphism between weak*-compact sets.
\end{theorem}

\begin{proof}
Fix $q_J \colon \T_X \to \O(J, X)$ be the canonical $*$-epimorphism.
Then the $\ker q_J = \K(\F(X) J)$ is generated by $\pi(a)p_0$ for $a \in J$.
Let $\tau \in \Tr_\be(A)$ and fix $\vphi_\tau$ be the associated state in $\Eq_\be^{\fty}(\T_X)$ given by Theorem \ref{T:para}.
Then we get that $\vphi_\tau(\pi(a) p_0) = c_{\tau,\be}^{-1} \tau(a)$.
Therefore, if $\tau$ vanishes on $J$ then $\vphi_\tau$ vanishes on $\ker q_J$ and so it induces a state on $\O(J,X)$.
As the unital quotient map intertwines the gauge actions the induced state is in $\Eq_\be^{\fty}(\O(J, X))$.
Conversely if $\vphi \in \Eq_\be^{\fty}(\O(J, X))$ then $\bar\vphi q_J \in \Eq_\be^{\fty}(\T X)$ and it defines $\tau_{\bar\vphi} \in \Tr_\be(A)$ by Theorem \ref{T:para}.
By construction $\tau_{\bar\vphi}$ vanishes on $J$ as $q_J(p_0 \pi(a) p_0) = q_J(\pi(a) p_0) = 0$ for all $a \in J$.
\end{proof}

\section{The infinite part of the equilibrium states}

Let us now see how we can parametrize $\Eq_\be^\infty(\O(A,X))$ (and thus all $\Eq_\be^\infty(\O(J,X))$).
The main point here is that these states are given by extending tracial states on $A$ rather than by taking statistical approximations.
When $X$ is non-degenerate and injective then the existence of such a $\Psi$ can be derived by combining \cite[Theorem 2.1 and Theorem 2.5]{LacNes04}.
However the attack therein is essentially different, as the well definedness of the extension is verified by using perturbations of the action.
Following \cite[Theorem 3.18]{KajWat13} we can directly construct the extension within the fixed point algebra (and by keeping the same action).

\begin{theorem}\label{T:para 0}
Let $X$ be a C*-correspondence of finite rank over $A$ and let $\be > 0$.
Let 
\[
I := \{a \in A \mid \lim_n \nor{\phi_X(a) \otimes \id_{X^{\otimes n - 1}}} = 0\}.
\]
Then there is an affine weak*-homeomorphism
\[
\Psi \colon \{\tau \in \Avt_\be(A) \mid \tau|_I = 0\} \to \Eq_\be^\infty(\O(A,X)) \text{ such that } \Psi(\vphi)|_{\rho(A)} = \tau.
\]
In particular $\Psi$ induces an affine weak*-homeomorphism onto $\Eq_\be^\infty(\T_X)$.
\end{theorem}

\begin{proof}
By Theorem \ref{T:con}, $\Eq_\be^\infty(\T_X)$ consists exactly of the $(\si,\be)$-KMS states that factor through $\Eq_\be(\O(A,X))$.
The ideal $I$ is the kernel of $q|_{\pi(A)}$ for the canonical $*$-epimorphism $q \colon \T_X \to \O(A, X)$.
Suppose first that $X$ is not injective and fix $(\rho, v)$ such that $\O(A, X) = \ca(\rho, v)$.
Then $I = \ker \rho$ and we claim that $\O_X$ is canonically $*$-isomorphic to $\O_Y$ for $Y = v(X)$ and $B = \rho(A)$.
To this end first notice that $Y$ is injective and of finite rank so that $J_Y = B$.
Indeed the covariance gives that
\[
\rho(a) = \rho(a) \sum_{i \in [d]} v(x_i) v(x_i)^* \foral a \in A.
\] 
Hence if $a + I \in \ker \phi_Y$ then $\rho(a) = \rho(a) \sum_{i \in [d]} v(x_i) v(x_i)^* = 0$, so that $a \in \ker \rho = I$.
Secondly it is clear that $(\id_B, \id_Y)$ defines a $B$-covariant pair for $Y$ since for $b = \rho(a)$ we have
\[
\psi_{\id_Y}(\phi_Y(b)) = \rho(a) = \id_B(b).
\]
Moreover it inherits a gauge action and trivially $\id_B$ is injective on $B$.
Thus the Gauge-Invariant-Uniqueness-Theorem asserts that $\O_Y = \ca(\id_B, \id_v) = \ca(\rho, v)$.

Therefore without loss of generality we may assume that $X$ is injective so that $\O(A, X) = \O_X$ and $I = (0)$.
We have to produce a weak*-homeomorphism $\Psi \colon \Avt_\be(A) \to \Eq_\be^\infty(\O_X)$ such that $\Psi^{-1}(\vphi) = \vphi \rho$. 
Let $\vphi \in \Eq_\be^\infty(\O_X)$ and set $\tau := \vphi \rho \in \Tr(A)$.
Therefore the KMS-condition yields
\[
\tau(a)
= \vphi( \rho(a) \sum_{i \in [d]} v(x_i) v(x_i)^*) 
= e^{-\be} \sum_{i \in [d]} \vphi(v(x_i)^* \rho(a) v(x_i))
= e^{-\be} \sum_{i \in [d]} \tau( \sca{x_i, a x_i} )
\]
and thus $\tau \in \Avt_\be(A)$.
Now fix $\tau \in \Avt_\be(A)$ and we will construct a $\vphi_\tau \in \Eq_\be^\infty(\O(A,X))$.
To this end we use a well known construction, that goes as back as \cite{Pim97}.
Namely, when $X$ is injective and the left action is by compacts then the fixed point algebra $\O_X$ can be identified with the direct limit
\[
\xymatrix{
A \ar[r]^{\phi_X} & \K X \ar[r]^{\otimes \id_X} & \K X^{\otimes 2} \ar[r]^{\otimes \id_X} & \cdots
}
\]
where $[\otimes \id_X] (t) = t \otimes \id _X$.
In our case this identification is given by
\[
\theta_{\xi^{\otimes n}, \eta^{\otimes n}} \otimes \id_X
=
\sum_{i \in [d]} \theta_{\xi^{\otimes n} \otimes x_i, \eta^{\otimes n} \otimes x_i},
\]
as a direct computation on elementary tensors shows.
Therefore $\O_X^\ga$ is the inductive limit of the increasing sequence $\{(\O_X^\ga)_n\}_{n \in n\bN}$ for
\[
(\O_X^\ga)_n := \ol{\spn} \{ v(\xi^{\otimes n}) v(\eta^{\otimes n})^* \mid \xi^{\otimes n}, \eta^{\otimes n} \in X^{\otimes n}\} = \psi_{v, n}(\K X^{\otimes n})
\]
by writing
\[
v(\xi^{\otimes n}) v(\eta^{\otimes n})^* = \sum_{i \in [d]} v(\xi^{\otimes n}) v(x_i) v(x_i)^* v(\eta^{\otimes n})^*.
\]
We define the functionals $\vphi_n$ on $(\O_X^\ga)_n$ by
\[
\vphi_n(\psi_{v, n}(k_n)) := e^{-n \be} \sum_{|\mu| = n} \tau(\sca{x_\mu, k_n x_\mu}) \FOR k_n \in \K X^{\otimes n}.
\]
To see that it is well defined notice that for a positive $k_n$ we have
\[
e^{-n \be} \sum_{|\mu| = n} \tau(\sca{x_\mu, k_n x_\mu})
\leq
\nor{k_n} e^{-n \be} \sum_{|\mu| = n} \tau(\sca{x_\mu, x_\mu})
\leq
\nor{k_n}.
\]
In particular we have $\vphi_n(\rho(1_A))
=
e^{-n \be} \sum_{|\mu| = n} \tau(\sca{x_\mu, x_\mu})
=
1$
and so each $\vphi_n$ is a state.
By construction we have that
\begin{align*}
\vphi_n(\psi_{v,n}(\theta_{\xi^{\otimes n}, \eta^{\otimes n}}))
& =
e^{-n\be} \sum_{|\mu| = n} \tau(\sca{x_\mu, \xi^{\otimes n}} \sca{\eta^{\otimes n}, x_\mu}) \\
& =
e^{-n\be} \sum_{|\mu| = n} \tau(\sca{\eta^{\otimes n}, x_\mu} \sca{x_\mu, \xi^{\otimes n}}) 
 =
e^{-n \be} \tau(\sca{\eta^{\otimes n}, \xi^{\otimes n}}).
\end{align*}
We see that the collection $\{\vphi_n \mid n \in \bN\}$ is compatible with the direct limit structure since
\begin{align*}
\vphi_{n+1}(\psi_{v, n+1}(\theta_{\xi^{\otimes n}, \eta^{\otimes n}} \otimes \id_X))
& =
\sum_{i \in [d]} \vphi_{n+1}(v(\xi^{\otimes n}) v(x_i) v(x_i)^* v(\eta^{\otimes n})^*) \\
& =
e^{-(n+1) \be} \sum_{|\mu| = n+1} \sum_{i \in [d]} \tau (\sca{x_\mu, \xi^{\otimes n} \otimes x_i} \sca{ \eta^{\otimes n} \otimes x_i, x_\mu}) \\
& =
e^{-(n+1) \be} \sum_{i \in [d]} \sum_{|\mu| = n+1} \tau (\sca{\eta^{\otimes n} \otimes x_i, x_\mu} \sca{x_\mu, \xi^{\otimes n} \otimes x_i}) \\
& =
e^{-(n+1) \be} \sum_{i \in [d]} \tau (\sca{x_i, \sca{\eta^{\otimes n}, \xi^{\otimes n}} x_i}) \\
& =
e^{-n \be} \tau (\sca{\eta^{\otimes n}, \xi^{\otimes n}}) 
 = 
\vphi_n(\psi_{v,n}(\theta_{\xi^{\otimes n}, \eta^{\otimes n}})).
\end{align*}
Therefore it defines a state $\vphi_\tau$ in the limit which extends $\tau$ such that
\[
\vphi_\tau(v(\xi^{\otimes n}) v(\eta^{\otimes n})^*)
=
\vphi_n(v(\xi^{\otimes n}) v(\eta^{\otimes n})^*)
= 
e^{-n\be} \tau(\sca{\eta^{\otimes n}, \xi^{\otimes n}}).
\]
Let $E \colon \O_X \to \O_X^\ga$ be the conditional expectation coming from the gauge action.
Then Proposition \ref{P:char} yields that the induced state $\vphi_\tau E$ is a $(\si,\be)$-KMS state on $\O_X$.
The same proposition implies that $\vphi_\tau E$ is the unique $(\si,\be)$-KMS state with restriction $\tau$ on $A$.
Therefore $\Psi$ is injective.

It is immediate that $\Psi^{-1}$ is weak*-continuous and affine.
Since $\Eq_\be^\infty(\O(A,X))$ is weak*-compact and $\Avt_\be(A)$ is Hausdorff, it follows that $\Psi$ is a weak*-homeomorphism.
\end{proof}

The same method applies to parametrize the KMS-states at $\be=0$, i.e., the gauge-invariant tracial states on $\O(A,X)$.
However we cannot have arbitrarily large $\be>0$ for $\O(A,X)$.
Obviously $q(p_0) = 0$ and so $\Eq_\be(\O(A,X)) = \Eq_\be^\infty(\O(A,X))$.
Therefore Theorem \ref{T:para rel} is void for $\O(A,X)$ at $\be \leq h_X^s$; and there is a good reason for this.

\begin{proposition}\label{P:not bigger}
If $X$ is a C*-correspondence of finite rank over $A$ then $\Eq_\be(\O(A,X)) = \mt$ for all $\be > h_X^s$.
If, in addition, $A$ is abelian and $h_X^s >0$ then $\O(A,X)$ attains equilibrium states at $h_X^s$.
\end{proposition}

\begin{proof}
Since $\vphi q \in \Eq_\be^\infty(\T_X)$ for $\vphi \in \Eq_\be(\O(A,X))$, Proposition \ref{P:hx} yields $\be \leq h_X^s$.
The same proposition and Theorem \ref{T:para 0} gives the second part of the statement.
\end{proof}

\begin{corollary}\label{C:eq ent}
Let $X$ be a C*-correspondence of finite rank over $A$.
Then:
\begin{enumerate}[leftmargin=30pt, itemindent=0pt, itemsep=1pt]
\item $h_X = \max \big\{0, \inf \{ h_X^\tau \mid \tau \in \Tr(A)\} \big\}$.
\item If $0 < h_X$, or if $0 \leq h_\tau$ for all $\tau \in \Tr(A)$, then $h_X = \min \{ h_X^\tau \mid \tau \in \Tr(A)\}$.
\item If $h_X^\tau = h_X^s$ for all $\tau \in \Tr(A)$ then $\Eq_\be^\infty(\T_X) = \mt$ for all $\be > h_X$.
\end{enumerate}
\end{corollary}

\begin{proof}
For item (i) let $\be \geq h_X$ so that $\Eq_\be(\T_X) \neq \mt$.
Due to the decomposition and the parametrization we get that $\Tr_\be(A) \neq \mt$ or $\Avt_\be(A) \neq \mt$.
In any case there is a $\tau \in \Tr(A)$ such that $h_X^\tau \leq \be$.
Therefore
\[
\inf\{ h_X^\tau \mid \tau \in \Tr(A)\} \leq h_X.
\]
Suppose there were a $\tau \in \Tr(A)$ such that $h_X^\tau < h_X$.
If $h_X^\tau < 0$ then it is clear that $\tau \in \Tr_\be(A)$ for all $\be > 0$ in which case $h_X = 0$.
If $h_X^\tau >0$ then choose $\be \in (h_X^\tau, h_X)$.
Then the root test gives that $c_{\tau,\be} < \infty$ and thus the contradiction $\Eq_\be^{\fty}(\T_X) \neq \mt$.
For item (ii), weak*-compactness gives that $\Eq_{h_X}(\T_X) \neq \mt$.
We consider two possible cases for $h_X$:

\smallskip

\noindent
{\bf Case (a).}
If $h_X > 0$ then item (i) implies that $h_X = \inf \{ h_X^\tau \mid \tau \in \Tr(A)\}$.
Now we can decompose a $\vphi \in \Eq_{h_X}(\T_X)$ and use the parametrization of each component to get a $\tau_0 \in \Tr_{h_X}(A) \cup \Avt_{h_X}(A)$ with $0 \leq h_X^{\tau_0} \leq h_X$.
However by item (i) we have that $h_X^{\tau_0} \geq h_X$ and thus we have equality, i.e., a minimum at $h_X^{\tau_0}$. 

\smallskip

\noindent
{\bf Case (b).}
If $h_X = 0$ but $h_X^\tau \geq 0$ for all $\tau \in \Tr(A)$,  then $\T_X$ admits a tracial state $\vphi$ such that
\[
\sum_{|\mu| = k} \vphi \pi (\sca{x_\mu, x_\mu}) = \vphi(\sum_{|\mu| = k} t(x_\mu) t(x_\mu)^*) \leq \vphi(\pi(1_A)) = 1.
\]
As this holds for all $k  \in \bZ_+$ we have that $0 \leq h_X^{\tau_0} \leq \log 1$ and so $h_X^{\tau_0} = 0 = h_X$ for $\tau_0 := \vphi \pi \in \Tr(A)$.
The third item follows by Proposition \ref{P:not bigger}.
\end{proof}

\section{Comments and applications}\label{S:ex}

\subsection{Unit decompositions}\label{Ss:dif s ent}

The strong entropy requires taking the infimum over all possible unit decompositions.
This is because the notion of basis is not well defined for C*-correspondences over non-commutative C*-algebras.
Let us give such an example here.

\begin{example}
Let $A = C(K) \oplus \O_2$ for a compact and Hausdorff space $K$ and the Cuntz algebra $\O_2 = \ca(s_1, s_2)$.
Let $\al \in \End(A)$ be given by $\al(a,b) = (a, s_1 b s_1^* + s_2 b s_2^*)$ and let $X$ be the induced C*-correspondence ${}_\al A$.
That is $X = A$ as a vector space and
\[
\sca{\xi, \eta} = \xi^* \eta \qand (a,b) \cdot \xi \cdot (c,d) = \al(a,b) \xi (c,d).
\]
We chose $A$ to have a commutative part so that $\Tr(A) \neq \mt$.
In \cite{Kak14b} it is shown that $\mt \neq \Eq_\be(\T_X) = \Eq_\be^{\fty}(\T_X)$ for all $\be \in (0, \infty)$.

Now ${}_\al A$ admits at least two unit decompositions $x = \{(1,1)\}$ and $y = \{(1, s_1), (1, s_2)\}$.
It is clear that $\sum_{|\mu| = k} \sca{x_\mu, x_\mu} = 1$.
On the other hand we have that $y_1 = (1, s_1)$ and $y_2 = (1, s_2)$ are orthonormal and so $\sca{y_\mu, y_\mu} = (1,1)$ for all $\mu \in \bF_+^2$.
We then see that they have different entropies as
\[
h_X^x = \limsup_k \frac{1}{k} \log \nor{\sum_{|\mu| = k} \sca{x_\mu, x_\mu}}
=
0 < \log 2
=
\limsup_k \frac{1}{k} \log \nor{\sum_{|\mu| = k} \sca{y_\mu, y_\mu}}
=
h_X^y.
\]
\end{example}

\subsection{Orthogonal bases}\label{Ss:ob}

In several examples, the C*-correspondence is over an abelian $A$ and admits a finite orthonormal basis.
From our analysis, and in particular from Corollary \ref{C:eq ent}, we get directly the KMS-structure in these cases:
\begin{enumerate}[leftmargin=30pt, itemindent=0pt, itemsep=1pt]
\item $h_X^\tau = \log d$ for all $\tau \in \Tr(A)$ and $h_X = h_X^s  = \log d$;
\item $\Eq_\be(\T_X) = \Eq_\be^{\fty}(\T_X) \neq \mt$ for all $\be > \log d$, and $\Eq_{\log d}(\T_X) = \Eq_{\log d}^\infty(\T_X) \neq \mt$.
\end{enumerate}
Indeed suppose that $X$ admits a finite orthonormal basis $x = \{x_1, \dots, x_d\}$, i.e., $\sca{x_i, x_j} = \de_{i,j}$.
Then $\sca{x_\mu, x_\nu} = \de_{\mu, \nu}$ when $|\mu| = |\nu|$, so that
\[
h_X^\tau = h_X^s = \log d \foral \tau \in \Tr(A).
\]
Corollary \ref{C:eq ent} yields that $h_X = \log d$ and thus $\Eq_\be(\T_X) = \Eq_\be^{\fty}(\T_X)$ for all $\be > \log d$.
Moreover we see that $c_{\tau, \log d} = \sum_{k=0}^\infty 1$ for all $\tau \in \Tr(A)$, so that
\[
\Eq_{\log d}(\T_X) = \Eq_{\log d}^\infty(\T_X).
\]
As we noted in Proposition \ref{P:not bigger} we have that $\Eq_{\log d}^\infty(\T_X) \neq \mt$.
As applications we get the full KMS-structure for the Pimsner algebras:
\begin{enumerate}[leftmargin=30pt, itemindent=0pt, itemsep=1pt]
\item[(a)] In \cite{Kak14b}, by applying for $d$ the multiplicity of the dynamical system;
\item[(b)] In \cite{LRR11}, by applying for $d = |\det A|$ and using \cite[Lemma 2.6]{EHR11};
\item[(c)] In \cite{LRRW13}, by applying for $d = |X|$ and using \cite[Equation 3.1]{LRRW13}.
\end{enumerate}
We will see below that Corollary \ref{C:eq ent} gives also the KMS-structure of \cite{HLRS13, KajWat13} for C*-algebras of irreducible graphs.
Notice that in addition to that we provide a clear parametrization of all the equilibrium states at the critical temperature $\be = \log d$.

\subsection{Irreducible graphs}\label{Ss:gr}

C*-corres\-pon\-dences of finite graphs is the first step away from orthonormal bases.
Their KMS-structure gives a nice mixing of cases to distinguish between entropies, i.e.,
\begin{enumerate}[leftmargin=30pt, itemindent=0pt, itemsep=1pt]
\item It may be the case that $h_X < h_X^s$, and in particular $\Eq_\be^{\fty}(\T_X) \neq \mt$ for every $\be \in (h_X, h_X^s)$.
\item It may be the case that $\Eq_\be^{\fty}(\T_X) \neq \mt$ and $\Eq_\be^\infty(\T_X) \neq \mt$ for some $\be > 0$.
\item It may be the case that $\Eq_\be^{\fty}(\O_X) \neq \mt$ at every $\be > 0$.
\end{enumerate}
The equilibrium states of this category have been extensively investigated in \cite{HLRS13, HLRS15, KajWat13}.
Pimsner algebras of irreducible graphs had been considered in \cite{HLRS13} were the Perron-Fr\"{o}benius Theorem is used in an essential way.
We will see here how the entropy theory we have developed applies and recovers the results therein.

To fix notation, the C*-correspondence $X \equiv X_G$ of a graph $G = (G^{(0)}, G^{(1)}, s, r)$ is the linear span of $\{x_e \mid e \in G^{(1)}\}$ over the abelian C*-algebra generated by orthogonal projections $\{p_v \mid v \in G^{(0)}\}$ such that
\[
\sca{x_e, x_f} = \de_{e,f} p_{s(f)} \quad p_v x_e = \de_{v, r(e)} x_e \quad x_e p_v = \de_{s(e), v} x_e.
\]
Then $X$ admits the unit decomposition given by the basis $\{x_e \mid e \in G^{(1)}\}$.
We will write
\[
\{p_v \mid v \in \G^{(0)}\}
\qand
\{L_e \mid e \in G^{(1)}\}
\]
for the induced Toeplitz-Cuntz-Krieger family, and
\[
\{p_v \mid v \in G^{(0)}\}
\qand
\{S_e \mid e \in G^{(1)}\}
\]
for the induced Cuntz-Krieger family.
We use the same symbol for the projections as they have an injective copy in $\T_X$ and $\O_X = \ca(G)$.

In what follows we write $G$ for both the graph and its adjacency matrix with the understanding that $G_{ij}$ denotes the edges from $i$ to $j$.
We will be loose in differentiating between the entry $i$ and the vertex $v_i$ corresponding to it.
Let us start with an auxiliary proposition.

\begin{proposition}\label{P:irr graph}
Let $G$ be a finite non-zero irreducible graph and $p > 0$.
Then for every $v_{r} \in G$ we have
\[
\limsup_k \frac{1}{k} \log(p \sum_{j \in G} (G^k)_{r j}) = \log \la_G,
\]
where $\la_G$ is the Perron-Fr\"{o}benius eigenvalue of $G$.
\end{proposition}

\begin{proof}
Suppose that $G$ has $n$ vertices and let $w = [w_1, \dots, w_n]$ be the positive Perron-Fr\"{o}benius eigenvector corresponding to $\la_G$.
Without loss of generality assume that $w_1 = \max\{w_1, \dots, w_n\}$ and compute
\begin{align*}
p \sum_{j \in G} (G^k)_{rj}
& \geq
\frac{p}{w_1} \sum_{j \in G} (G^k)_{r j} w_j
=
\frac{p}{w_1} \la_G^k w_{r}
=
\frac{p \cdot w_{r}}{w_1} \la_G^k.
\end{align*}
On the other hand we have that
\[
p \sum_{j \in G} (G^k)_{r j}
\leq
p \sum_{i, j \in G} (G^k)_{ij}.
\]
Therefore
\begin{align*}
\log \la_G 
&= \limsup_k \frac{1}{k} \log(\frac{p \cdot w_{r}}{w_1} \la_G^k)
\leq
\limsup_k \frac{1}{k} \log(p \sum_{j \in G} (G^k)_{r j}) \\
& \leq
\limsup_k \frac{1}{k} \log(p \sum_{i,j \in G} (G^k)_{ij})
=
\limsup_k \frac{1}{k} \log(\sum_{i,j \in G} (G^k)_{ij})
=
\log \la_G
\end{align*}
where the last equality follows from the Perron-Fr\"{o}benius Theorem.
\end{proof}

\begin{theorem}\label{T:HLRS13} \cite{HLRS13}
Let $X$ be the C*-correspondence associated to a finite non-zero irreducible graph $G$.
Let $\la_G$ be the Perron-Fr\"{o}benius eigenvalue of $G$ (and thus of its transpose $G^t$).
Then:
\begin{enumerate}[leftmargin=30pt, itemindent=0pt, itemsep=1pt]
\item $h_X = \log \la_G$;
\item $\Eq_\be(\T_X) = \Eq_\be^{\fty}(\T_X)$ for all $\be > h_X$;
\item $\Eq_{h_X}(\T_X) = \Eq_{h_X}^\infty(\T_X)$ is a singleton on the Perron-Fr\"{o}benius eigenvector of $G^t$.
\end{enumerate}
\end{theorem}

\begin{proof}
Let $\tau \in \Tr(A)$.
Then by definition we have that
\[
h_X^s \leq \limsup_k \frac{1}{k} \log \big( \sum_{|\mu|=k} \|\sca{x_{\mu}, x_{\mu}}\| \big) = \limsup_k \frac{1}{k} \log( \sum_{i,j \in G} (G^k)_{ij})
=
\log \la_G,
\]
by the Perron-Fr\"{o}benius Theorem.
On the other hand set
\[
P := \diag\{p_{i} \mid i \in G\} \textup{ where } p_i := \tau(p_{v_i}).
\]
That is $p_i$ is the evaluation of $\tau$ at the projection corresponding to $v_i$.
Thus $P$ is a diagonal matrix whose entries $\{p_i \mid i \in G\}$ sum up to one.
Let a $v_r \in G$ such that $p_{r} \neq 0$.
Let $w = [w_1, \dots, w_n]$ be the Perron-Fr\"{o}benius eigenvector and suppose without loss of generality that $w_1 = \max\{w_1, \dots, w_n\}$.
As in Proposition \ref{P:irr graph} we get
\begin{align*}
\sum_{|\mu| = k} \tau(\sca{x_\mu, x_\mu})
 =
\sum_{i,j \in G} (P G^k)_{ij}
\geq
p_{r} \sum_{j \in G} (G^k)_{r j} 
 \geq
\frac{p_{r}}{w_1} \sum_{j \in G} (G^k)_{r j} w_j
=
\frac{p_{r} \cdot w_{r}}{w_1} \la_G^k.
\end{align*}
Therefore $h_X^\tau \geq \log \la_G$ and so
\begin{align*}
\log \la_G \leq h_X^\tau \leq h_X^s \leq \log \la_G.
\end{align*}
Then Corollary \ref{C:eq ent}(iii) gives items (i) and (ii).
We also see that
\[
c_{\tau, \log \la_G} \geq \frac{p_{r} \cdot w_{r}}{w_1} \sum_{k=0}^\infty e^{-k \log \la_G} \la_G^k = \infty 
\]
so that $\Tr_{\log \la_G}(A) = \mt$.
Hence $\Eq_{\log \la_G}(\T_X) = \Eq_{\log \la_G}^\infty(\T_X)$ and for item (iii) it remains to show that $\Avt_{\log \la_G}(A)$ is a singleton.
A direct computation gives that
\begin{equation}\label{eq:phase}
\sum_{e \in G^{(1)}} \tau(\sca{x_e, p_i x_e})
=
\sum_{e \in r^{-1}(v_i)} \tau(\sca{x_e, x_e})
=
\sum_{e \in r^{-1}(v_i)} \tau(p_{s(e)}) 
=
\sum_{j \in G} g_{ji} p_j
=
\sum_{j \in G^t} (G^t)_{ij} p_j
\end{equation}
Thus $\tau \in \Avt_{\be}(A)$ if and only if $[p_1, \dots, p_n]$ is an $e^\be$-eigenvector of $G^t$ with $\ell^1$-norm equal to one.
By uniqueness of the Perron-Fr\"{o}benius eigenvector we derive that there is only one $\tau \in \Avt_{\log \la_G}(A)$.
\end{proof}

\begin{remark}
The class of irreducible graphs showcases why the convex decomposition between the finite and the infinite part is not weak*-continuous with respect to $\be > 0$.
For let $\tau$ be the trace corresponding to the Perron-Fr\"{o}benius eigenvalue $\la_G$.
Then it defines a finite state for every $\be > \log \la_G$.
Letting $\be \to \log \la_G$ gives the infinite KMS-state at $\log \la_G$.
\end{remark}

\subsection{Reducible graphs}

Let us now show how the entropy theory can recover the results of \cite{HLRS15}.
By removing sources one at a time we can write the adjacency matrix of a graph $G$ as an upper triangular matrix
\[
G
=
\begin{bmatrix} 
G_1 & \ast & \cdots & \ast \\
0 & G_2 & \cdots & \ast \\
\vdots & \vdots & \vdots & \vdots \\
0 & 0 & \cdots & G_m 
\end{bmatrix}
\]
where the $G_1, \dots, G_m$ are its irreducible components (including some be possibly equal to $[0]$).
It is well known that the entropy of the graph $G$ can be given by the Perron-Fr\"{o}benius eigenvalues of the irreducible components in the sense that
\begin{equation}\label{eq:gr en}
h_G := \limsup_k \frac{1}{k} \log (\sum_{i,j \in G} (G^k)_{ij}) = \max\{\log \la_{G_1}, \dots, \log \la_{G_m}\};
\end{equation}
see for example \cite[Theorem 4.4.4]{LM95}.

\begin{definition}
We say that a component of $G$ is \emph{communicated} by a $v \in G$ if there is a path from $v$ inside that component.
We say that $G_r$ is \emph{communicated} by $G_s$ if there exists a path starting at some $v_s \in G_s$ that ends in  some $v_r \in G_r$.
\end{definition}

The proof of the following proposition uses the main idea of the proof of \cite[Theorem 4.4.4]{LM95}.
We say that a trace $\tau$ has \emph{support in $G_s$} if there is a vertex $v \in G_s$ such that $\tau(p_v) \neq 0$.

\begin{proposition}\label{P:tr en gr}
Let $G$ be a finite graph with irreducible components $G_1, \dots, G_m$.
Then for every $\tau \in \Tr(A)$ we have that
\begin{align*}
h_X^ \tau 
=
\max \{\log \la_{G_s} \mid \textup{ $G_s$ is communicated by some $v_{r}$ in the support of $\tau$} \}.
\end{align*}
\end{proposition}

\begin{proof}
For convenience let us set
\[
\la = \max\{ \la_{G_s} \mid \textup{ $G_s$ is communicated by some $v_{r}$ in the support of $\tau$} \}.
\]
Let $G_s$ be an irreducible component with which a vertex $v_{r}$ in the support of $\tau$ communicates.
Thus there is a path of length $N_0$ from $v_{r}$ to some $v_{i_s} \in G_s$.
Suppose that $G_s \neq [0]$.
Then the paths of length $k+N_0$ emitting from $v_{r}$ are more than the paths of length $k$ which start at $v_{i_s}$ and end inside $G_s$.
Thus we have
\begin{align*}
\sum_{i,j \in G} (PG^{k+N_0})_{ij}
\geq
p_{r} \sum_{j \in G} (G^{k+N_0})_{r j}
\geq
p_{r} \sum_{j \in G_s} (G_s^{k})_{i_s j}.
\end{align*}
Proposition \ref{P:irr graph} applied for $v_{i_s} \in G_s$ and $p = p_r$ yields
\begin{align*}
h_X^\tau
& =
\limsup_k \frac{1}{k+N_0} \log \big( \sum_{i,j \in G} (PG^{k+N_0})_{ij} \big) \\
& \geq
\limsup_k \frac{k}{k+N_0} \cdot \frac{1}{k} \log \big( p_{r} \sum_{j \in G_s} (G_s^{k})_{i_s j} \big)
=
\log \la_{G_s}.
\end{align*}
The same holds trivially if $G_s = [0]$.
Therefore $h_X^\tau \geq \log \la$.

On the other hand recall that the proof of Perron-Fr\"{o}benius Theorem asserts that there exists an $\al_1 > 0$ such that
\[
\sum_{i,j \in G_1} (G_1^k)_{ij} \leq \al_1 \la_{G_1}^k;
\]
see for example the comments preceding \cite[Proposition 4.2.1]{LM95}.
Likewise for $G_2, \dots, G_m$ and set
\[
\al := \max\{\al_1, \dots, \al_m\}.
\]
First let $v_r$ be in the support of $\tau$ and we want to estimate from above the number $\sum_{j \in G} (G^k)_{rj}$ of paths $\mu$ starting at $v_r$ with length $k$.
Such a path is of the form (read from left to right)
\[
\mu = v_r \mu_1 e_1 \cdots e_q \mu_q 
\textup{ with }
|\mu_1| + \cdots + |\mu_q| + q =k,
\]
with $\mu_1, \dots, \mu_q$ be paths entirely in some irreducible components $G_{s_1}, \dots, G_{s_q}$, respectively, while the $e_1, \dots, e_q$ are transitional edges between different components.
Of course $v_r$ communicates with all these components so that $\la \geq \la_{G_{s_1}}, \dots, \la_{G_{s_q}}$.
It is clear that the number of the transitional edges in $\mu$ cannot be more than the number of the components, that is
\[
0 \leq q \leq m.
\]
Set $M$ be the number of all transitional edges in $G$.
For each of the $e_1, \dots, e_q$ we have at most $M$ choices and at most $k$ places to insert it.
Hence the choices for all the $e_1, \dots, e_q$ in $\mu$ cannot exceed $(M k)^q$.
On the other hand suppose that $\mu_1$ has length $k(s_1)$.
We have at most $\sum_{i,j \in G_{s_1}} (G_{s_1}^k)_{ij}$ choices for such a path, and
\[
\sum_{i,j \in G_{s_1}} (G_{s_1}^k)_{ij}
\leq
\al_1 \cdot \la_{G_1}^{k(s_1)}
\leq
\al \cdot \la^{k(s_1)}.
\]
Likewise for $G_{s_2}, \dots, G_{s_q}$ and we note that 
\[
k(s_1) + k(s_2) + \dots + k(s_q) \leq k.
\]
Therefore we have that
\begin{align*}
\sum_{j \in G} (G^k)_{r j} 
 \leq
(M k)^q \cdot (\al \cdot \la^{k(s_1)}) \cdots (\al \cdot \la^{k(s_q)}) 
 \leq
(\al M k)^q \cdot \la^k
\leq
(\al M k)^m \cdot \la^k.
\end{align*}
Applying for all $v_{r_1}, \dots, v_{r_\ell}$ in the support of $\tau$ we get that
\begin{align*}
\sum_{i,j \in G} (PG^k)_{ij}
& =
\sum_{i= r_1}^{r_\ell} p_i \sum_{j \in G} (G^k)_{ij}
\leq
\ell \cdot (\al M k)^m \cdot \la^k.
\end{align*}
Consequently we get that $h_X^\tau \leq \log \la$ and the proof is complete.
\end{proof}

We can further identify the tracial components required for the KMS-structure.
Note that the root test is always conclusive for those.

\begin{corollary}\label{C:tr en}
Let $G$ be a finite graph and $\be > 0$.
Then we have that
\[
\Tr_\be(A) = \{\tau \in \Tr(A) \mid h_X^\tau < \be \}
\qand
\Avt_\be(A) = \{\tau \in \Tr(A) \mid G^t \tau = e^{\be} \tau \}.
\]
\end{corollary}

\begin{proof}
The part on averaging traces follows from the same computation as in equation (\ref{eq:phase}).
For $\Tr_\be(A)$ it suffices to show that $c_{\tau, \be} < \infty$ if and only if $h_X^\tau < \be$.
To this end suppose that $\tau$ is supported on the vertices $v_{r_1}, \dots, v_{r_{\ell}}$.
For the associated diagonal matrix $P$ set
\[
p_{\min} := \min\{ p_{r_1}, \dots, p_{r_\ell} \}.
\]
Let $G_{s_1}, \dots, G_{s_q}$ be all the irreducible components of $G$ with which a vertex from the support of $\tau$ communicates, and set $G_{s_0}$ for the one that corresponds to
\[
\la_{G_{s_0}} = \max \{ \la_{G_{s_1}}, \dots, \la_{G_{s_q}} \}. 
\]
By Proposition \ref{P:tr en gr} we have that $h_X^\tau = \log \la_{G_{s_0}}$.
As before there exist $\al_s, \al_s' > 0$ such that
\[
\al_s' \la_{G_s}^k \leq \sum_{j \in G_s} (G_s^k)_{i j} \leq \al_s \la_{G_s}^k.
\]
Set
\[
\al' = \min \{\al_{s_1}', \dots, \al_{s_q}' \}
\qand
\al = \max \{\al_{s_1}, \dots, \al_{s_q} \}.
\]

If all components $G_{s_1}, \dots, G_{s_q}$ are zero then the paths emitting from the support of $\tau$ are only on transitional edges.
As in the proof of Proposition \ref{P:tr en gr} we get that $\sum_{|\mu| = k} \tau(\sca{x_{\mu}, x_{\mu}}) \leq (Mk)^m$ for the total number of transitional edges $M$.
Then $h_X^\tau \leq 0$ and trivially $h_X^\tau < \be$.

So suppose there exists at least one component which is not zero.
Then $G_{s_0} \neq [0]$ with $\la_{G_{s_0}} \geq 1$.
Let $v_{r}$ be in the support of $\tau$ that communicates with some $v_{s_0} \in G_{s_0}$.
As in Proposition \ref{P:tr en gr}, after some finite step $N_0$ we get
\[
\sum_{i,j \in G} (P G^k)_{ij}
\geq
p_r \sum_{j \in G_{s_0}} (G^{k-N_0}_{s_0})_{s_0 j}
\geq
p_{\min} \cdot \al' \cdot \la_{G_{s_0}}^{k-N_0}.
\]
On the other hand we have seen in the proof of Proposition \ref{P:tr en gr} that
\[
\sum_{i,j \in G} (PG^k)_{ij} \leq \ell \cdot (\al M k)^m \cdot \la_{G_{s_0}}^k
\]
where $M$ is the total number of transitional edges.
Hence there are constants $M_1$ and $M_2$ such that
\[
M_1 \cdot e^{-k\be} \la_{G_{s_0}}^k \leq e^{-k\be} \sum_{|\mu| = k} \tau(\sca{x_\mu, x_\mu}) \leq M_2 \cdot e^{-k\be} \la_{G_{s_0}}^k.
\]
Thus $c_{\tau, \be} < \infty$ if and only if $h_X^\tau = \log \la_{G_{s_0}} < \be$.
\end{proof}

Next we compute the entropies and the phase transitions.
We require some notation and terminology.
Let $G_s$ be a non-zero irreducible component of $G$.
Then we can identify the components that communicate with $G_s$, say $G_{s_1}, \dots, G_{s_q}$ with the understanding that $s_1 < s_2 < \cdots < s_q$.
As ``communicating'' is transitive we can write
\[
G 
=
\begin{bmatrix}
\ast & \ast & \ast \\
0 & H_s & 0 \\
0 & 0 & \ast
\end{bmatrix}
\quad
\textup{for}
\quad
H_s
:=
\begin{bmatrix}
G_{s} & \ast & \cdots & \ast \\
0 & G_{s_1} & \cdots & \ast \\
\vdots & \vdots & \cdots & \vdots \\
0 & 0 & \cdots & G_{s_q}
\end{bmatrix},
\]
up to a permutation of the vertices.
We say that $H_s$ is the \emph{communicating graph on $G_s$}.
We also need the notion of sink for the irreducible component.

\begin{definition}
Let $G$ be a finite graph with irreducible components $G_1, \dots, G_m$.
A component $G_s$ is said to be a \emph{sink} if there are no paths emitting from $G_s$, i.e., if $G_{ij} = 0$ for all $i \in G_s$ and all $j \notin G_s$.
Equivalently when $H_s = G_s$.
\end{definition}

The definition of a sink component includes the case where $G_s = [0]$.
Note that when $G_s$ is a sink component then 
\[
\sum_{i \in G_s} \sum_{j \in G} (G^k)_{ij} = \sum_{i \in G_s} \sum_{j \in G_s} (G^k)_{ij} = \sum_{i \in G_s} \sum_{j \in G_s} (G_s^k)_{ij}.
\]
Moreover $G_m$ will always be a sink component.

\begin{theorem}\label{T:graph entropy}
Let $G$ be a finite graph with irreducible components $G_1, \dots, G_m$.
Then
\[
h_X^s = \max\{\log \la_{G_1}, \dots, \log \la_{G_m}\}.
\]
If there exists a zero sink component then $h_X = 0$; otherwise
\[
h_X
=
\min\{\log \la_{G_s} \mid G_s \textup{ is a non-zero sink irreducible component of } G \}.
\]
\end{theorem}

\begin{proof}
For the first part, on one hand we can use equation (\ref{eq:gr en}) to get
\begin{align*}
h_X^s
& \leq 
\limsup_k \frac{1}{k} \log( \sum_{|\mu| =k} \| \sca{x_{\mu}, x_{\mu}} \|) \\
& =
\limsup_k \frac{1}{k} \log( \sum_{i,j \in G} (G^k)_{ij})
=
\max\{\log \la_{G_1}, \dots, \log \la_{G_m}\}.
\end{align*}
On the other hand notice that the projections $p_v$ are orthogonal and so for any $s = 1, \dots, m$ and $r \in G_s$ we get
\begin{align*}
\| \sum_{|\mu|=k} \sca{x_\mu, x_\mu} \|
& =
\max_{i \in G} (\#\{x_\mu \mid |\mu| = k \textup{ starting at } v_i\}) \\
& =
\max\{ \sum_{j \in G} (G^k)_{ij} \mid i \in G\}
\geq 
\sum_{j \in G_s} (G^k_s)_{r j}.
\end{align*}
Applying Proposition \ref{P:irr graph} gives that $h_X^s \geq \log \la_{G_s}$ and thus
\[
h_X^s \geq \max\{\log \la_{G_1}, \dots, \log \la_{G_m}\}.
\]

For the second part first suppose that $G$ has a zero sink component.
Thus its adjacency matrix has a row of zeroes, i.e., there is a vertex $v_0$ that emits no edges.
If $\tau_0$ is the Dirac measure on that vertex then we see that $\sum_{|\mu|=k} \tau_0(\sca{x_\mu, x_\mu}) = 0$ for all $k$, as there are no paths from that vertex.
Hence $c_{\tau_0, \be} = 1$ and so $\tau_0 \in \Tr_\be(A)$ for all $\be > 0$.
In particular we get that $h_X = 0$.

So now suppose that all sink irreducible components are non-zero.
Fix an $s_0$ such that $\la_{G_{s_0}}$ corresponds to the minimum of their Perron-Fr\"{o}benius.
First let a trace $\tau$ supported entirely inside $G_{s_0}$ and set $P$ for its corresponding diagonal matrix.
Then we have
\[
\sum_{|\mu| = k} \tau(\sca{x_\mu, x_\mu})
=
\sum_{i,j \in G} (P G^k)_{ij}
=
\sum_{i \in G_{s_0}} \sum_{j \in G} (P G^k)_{ij}
=
\sum_{i,j \in G_{s_0}} (P G_{s_0}^k)_{ij}.
\]
Proposition \ref{P:irr graph} (as used in the proof of Theorem \ref{T:HLRS13}) yields \[
h_X^\tau = \log \la_{G_{s_0}}.
\]
On the other hand let $\tau$ be a trace that is not supported entirely on a sink component.
That is, there is a vertex $v_r$ that is not in a sink component with $\tau(p_{v_r}) = p_r \neq 0$.
Hence there is a connecting edge from $v_r$ into another irreducible component.
That may not be a sink component, but moving inductively we have that there is a sink component $G_s$ with which $v_r$ communicates. 
Then Proposition \ref{P:tr en gr} gives
\[
h_X^\tau \geq \log \la_{G_s} \geq \log \la_{G_{s_0}}.
\]
As equality holds when $\tau$ is supported entirely on $G_{s_0}$ taking the infimum over all $\tau$ gives the required $h_X = \log \la_{G_{s_0}}$.
\end{proof}

As a corollary we have the full classification of the equilibrium states and their phase transitions.
The parametrizations are weak*-homeomorphisms in this case.
We show this in three steps.
The following is the analogue of the minimal components in \cite{HLRS15}.

\begin{definition}
Let $G$ be a finite graph with irreducible components $G_1, \dots, G_m$ and let $\la > 1$.
We say that a $G_s$ is \emph{$\la$-maximal} if $\la_{G_s} = \la$ and $\la_{G_s} \geq \la_{G_r}$ for any $G_r$ that is communicated by $G_s$.
\end{definition}

\begin{proposition}\label{P:constant}
Let $X$ be a C*-correspondence associated to a finite graph with irreducible components $G_1, \dots, G_m$.
Let
\[
\La := \{\log \la_{G_s} \mid \textup{ $G_s$ is a $\la_{G_s}$-maximal component} \}.
\]
The finite KMS-simplex is constant for any $\be$ in the half-open half-closed intervals defined by $\La$. 
Moreover we have an infinite state at $\be = \log \la \geq 0$ if and only if there exists a $\la$-maximal component.
\end{proposition}

\begin{proof}
First we show that the possible values for $\tau \in \Tr(A)$ are $h_X^\tau < 0$ or $h_X^\tau = \log \la$ for some $\log \la \in \La$.
Suppose that $h_X^\tau > 0$ and by Proposition \ref{P:tr en gr} let $G_{s_0}$ such that $h_X^\tau = \log \la_{G_{s_0}}$. 
As $h_X^\tau > 0$ then $G_{s_0} \neq [0]$.
In order to reach contradiction assume that $\log \la_{G_{s_0}} \notin \La$.
Then there exists a $G_r$ that is communicated by $G_{s_0}$ such that $\la_{G_r} > \la_{G_{s_0}}$.
But then $G_r$ is also communicated by the support of $\tau$ giving the contradiction
\[
h_X^\tau \geq \log \la_{G_r} > \log \la_{G_{s_0}} = h_X^\tau.
\]

Next we show that $\Tr_\be(A) \neq \mt$ for all $\be > h_X$.
Let $G_{s_0}$ be the sink component for which $h_X = \log \la_{G_{s_0}}$.
Then any trace supported entirely on $G_{s_0}$ has entropy equal to $\log \la_{G_{s_0}}$ and thus is in $\Tr_\be(A)$ for all $\be > \log \la_{G_{s_0}} = h_X$.

Now let $\tau \in \Tr_{\be}(A)$ with $\be \in  (\log \la_1, \log \la_2]$ and $\log \la_1, \log \la_2 \in \La$.
On the one hand we have $h_X^{\tau} \in \La$ from the first paragraph.
On the other hand Corollary \ref{C:tr en} gives that $h_X^\tau < \be \leq \log \la_2$ and so $h_X^\tau \leq \log \la_1$.
But then $\tau \in \Tr_{\be'}(A)$ for any other $\be' > \log \la_1$.
Therefore
\[
\Tr_\be(A) = \Tr_{\be'}(A) \foral \be, \be' \in (\log \la_1, \log \la_2].
\]

For the second part first assume that there is an averaging trace $\tau$ at $\log \la > 0$.
Then $h_X^\tau = \log \la$.
Let $G_{s_1}, \dots, G_{s_q}$ be all components with which the support of $\tau$ communicates.
Then Proposition \ref{P:tr en gr} gives a $G_{s_0}$ such that
\[
\log \la_{G_{s_0}} = \max\{\log \la_{G_{s_1}}, \dots, \log \la_{G_{s_q}} \} = h_X^\tau = \log \la.
\]
Let a $G_r$ that is communicated by $G_{s_0}$.
Then $G_r$ is communicated also by the support of $\tau$ so that $G_r \in \{G_{s_1}, \dots, G_{s_q}\}$ which implies that 
\[
\log \la_{G_r} \leq \max\{\log \la_{G_{s_1}}, \dots, \log \la_{G_{s_q}} \} = \la_{G_{s_0}},
\]
giving that $G_{s_0}$ is $\la$-maximal.

Conversely suppose that $G_{s_0}$ is $\la$-maximal for $\la \geq 1$.
This means that the spectral radius of $H_{s_0}$, and thus of $H_{s_0}^t$, equals $\log \la_{G_{s_0}}$.
A variation of the Perron-Fr\"{o}benius Theorem gives an eigenvector $p \geq 0$ of $H_{s_0}^t$ at $\log \la_{G_{s_0}}$.
Then $p' = [0, \; p, \; 0]$ is an eigenvector of $G^t$ at $\log \la_{G_{s_0}}$ since
\[
G^t \cdot
p'
=
\begin{bmatrix}
\ast & 0 & 0 \\
\ast & H_{s_0}^t & 0 \\
\ast & 0 & \ast
\end{bmatrix}
\cdot
\begin{bmatrix}
0 \\
p \\
0
\end{bmatrix}
=
\begin{bmatrix}
0 \\
H_{s_0}^t p \\
0
\end{bmatrix}
=
e^{-\la_{G_{s_0}}}  \cdot p'.
\]
The $\ell^1$-normalization of $p'$ defines an averaging trace at $\log \la_{G_{s_0}}$ and the proof is complete.
\end{proof}

\begin{remark}\label{R:traces}
Our analysis identifies the simplices $\Tr_\be(A)$ for all $\be \geq h_X$.
To this end order the phase transitions in $\La$ by
\[
h_X = \log \la_1 < \cdots < \log \la_q = h_X^s.
\]
For any such $\log \la_n$ there are $\la_n$-maximal components $G_{n, 1}, \dots, G_{n, k_n}$.
Notice that no $G_{m, s'}$ can be communicated by any $G_{n, s}$ when $m > n$.
Set 
\[
V_n : = \{v \in G \mid \textup{ $v$ communicates with some $G_{n, s}$, $s = 1, \dots, k_n$}\}.
\]
Proposition \ref{P:tr en gr} yields
\[
h_X^\tau \leq \log \la_{n}
\textup{ iff }
\tau|_{V_j} = 0 \foral j > n.
\]
Therefore
\[
\Tr_\be(A)
=
\{\tau \in \Tr(A) \mid \tau|_{V_j} = 0 \foral j > n \}
\foral
\be \in (\log \la_{G_{s_n}}, \log \la_{G_{s_{n+1}}}].
\]
In particular there is a finite algorithm to decide if some $\de_v$ is in $\Tr_\be(A)$ and thus to identify $\Tr_\be(A)$:
\begin{itemize}[leftmargin=48pt, itemindent=0pt, itemsep=0pt]
\item[Step 0.] For $v \in G_{s_0}$ set $\La_v^{(0)} = \la_{G_{s_0}}$.

\item[Step 1.] If $G_{s_0 + 1}$ is communicated by $G_{s_0}$ then set $\La_v^{(1)} = \{\la_{G_{s_0}}, \la_{G_{s_1}}\}$.
Otherwise set $\La_v^{(1)} = \La_v^{(0)}$.

\item[Step $N$.] Suppose that $\La_v^{(N-1)} = \{\la_{G_{s_0}}, \dots, \la_{G_{s_{N-1}}}\}$.
If $G_{s_{N-1} + 1}$ is communicated by any of the $G_{s_0}, \dots, G_{s_{N-1}}$ then set $\La_v^{(N)} = \La_v^{(N-1)} \cup \{\la_{G_{s_{N-1} + 1}}\}$.
Otherwise set $\La_v^{(N)} = \La_v^{(N-1)}$.
\end{itemize}

\noindent
This algorithm terminates as we have finite components at a step, say $N_0$, and set $\la_v = \max \La_v^{(N_0)}$.
Then $h_X^\tau = \log \la_v$ and so $\de_v \in \Tr_\be(A)$ whenever $\log \la_v < \be$.
\end{remark}

\begin{theorem}\label{T:phase G}
Let $X$ be a C*-correspondence associated to a finite graph with irreducible components $G_1, \dots, G_m$.
Let
\[
\La := \{\log \la_{G_s} \mid \textup{ $G_s$ is $\la_{G_s}$-maximal}\}
=
\{h_X =\log \la_1 < \cdots < \log \la_q = h_X^s\}
\]
and let $V_1, \dots, V_q$ be the induced vertex sets.
Then the phase transitions occur at $\La$ and:
\begin{enumerate}[leftmargin=30pt, itemindent=0pt, itemsep=1pt]
\item For $\be \in (\log \la_n, \la_{n+1}]$ with $\log \la_n, \log \la_{n+1} \in \La$ we have an affine weak*-homeomorphism
\[
\Phi \colon \{\tau \in \Tr(A) \mid \supp \tau|_{V_j} = 0 \foral j > n \} \to \Eq_\be^{\fty}(\T_X).
\]
\item For $\be  = \log \la \in \La$ we have an affine weak*-homeomorphism
\[
\Psi \colon \{ \tau \in \Tr(A) \mid \textup{ $H_{s}^t \tau = e^{\be} \tau$ for some $G_s$ that is $\la$-maximal}\} \to \Eq_\be^{\infty}(\T_X).
\]
\end{enumerate}
\end{theorem}

\begin{proof}
By Theorem \ref{T:graph entropy} we have at least one averaging trace at each of the values in $\La$.
The same proof gives the domain of $\Psi$ (which has been shown to be a weak*-homeomorphism).
Remark \ref{R:traces} shows that the domain of $\Phi$ at $\be \in (\log \la_n, \log \la_{n+1}]$ is exactly
\[
\Tr_\be(A)
=
\{\tau \in \Tr(A) \mid \supp \tau|_{V_j} = 0 \foral j > n \},
\]
which is a weak*-closed subset of $\Tr(A)$.
Let
\[
p_W := \sum_{i \in W} p_{v_i} 
\qfor
W := (V_{n+1} \cup \cdots \cup V_q)^c.
\]
For $\tau \in \Tr(A)$ set $\tau_W(\cdot):= \tau(p_W \cdot p_W)$ and compute
\begin{align*}
\sum_{|\mu| = k, s(\mu) \in W} \tau(\sca{x_\mu, a x_\mu})
& =
\sum_{k=0}^N e^{-k\be} \sum_{|\mu| = k} \tau(p_W \sca{x_\mu, a x_\mu} p_W) 
 =
\sum_{k=0}^N e^{-k\be} \sum_{|\mu| = k} \tau_W(\sca{x_\mu, a x_\mu}).
\end{align*}
However by definition $\tau_W \in \Tr_\be(A)$ forcing
\[
\tau\bigg(\sum_{k=0}^{N} e^{-k\be} \sum_{|\mu| = k, s(\mu) \in W} \sca{x_\mu, a x_\mu}\bigg)
=
\sum_{k=0}^{N} e^{-k\be} \sum_{|\mu| = k} \tau_W(\sca{x_\mu, a x_\mu})
< c_{\tau_W, \be} \cdot \nor{a} <\infty
\]
Thus by the Banach-Steinhaus Theorem the sequence $\bigg(\sum_{k=0}^N e^{-k\be} \sum_{|\mu| = k, s(\mu) \in W} \sca{x_\mu, a x_\mu}\bigg)_N$ converges, for all $a \in A$.
The point here is that $W$ is given by $\be$, and makes the series to converge.

Suppose that $\tau_j \stackrel{\textup{w*}}{\longrightarrow} \tau$ in $\Tr_\be(A)$.
Then $\tau_j = (\tau_j)_W$ and $\tau = \tau_W$.
We can appeal to the Monotone Convergence Theorem to have that
\begin{align*}
\lim_j \sum_{k=0}^\infty e^{-k\be} \sum_{|\mu| = k} \tau_j(\sca{x_\mu, a x_\mu}) 
& = 
\lim_j \sum_{k=0}^\infty e^{-k\be} \sum_{|\mu| = k, s(\mu) \in W} \tau_j(\sca{x_\mu, a x_\mu}) \\
& = 
\lim_j \tau_j\bigg(\sum_{k=0}^\infty e^{-k\be} \sum_{|\mu| = k, s(\mu) \in W} \sca{x_\mu, a x_\mu}\bigg) \\
& =
\tau\bigg(\sum_{k=0}^N e^{-k\be} \sum_{|\mu| = k, s(\mu) \in W} \sca{x_\mu, x_\mu}\bigg) \\
& =
\sum_{k=0}^\infty e^{-k\be} \sum_{|\mu| = k, s(\mu) \in W} \tau(\sca{x_\mu, a x_\mu})\\
& =
\sum_{k=0}^\infty e^{-k\be} \sum_{|\mu| = k} \tau(\sca{x_\mu, a x_\mu}).
\end{align*}
Thus the parametrization $\Phi$ is weak*-continuous, and hence a weak *-homeomorphism.
\end{proof}

We continue with some examples to highlight the methods in the proofs.
To this end we will put some effort to compute the predicted values ad-hoc.
We start with an example that showcases two points:
\begin{enumerate}[leftmargin=30pt, itemindent=0pt, itemsep=1pt]
\item It may happen that a trace $\tau$ has $h_X^\tau \leq 0$ but it may not give an averaging trace at $0 = h_X^\tau$.
\item It may be the case that a trace is in $\Tr_\be(A)$ for all $\be > 0$, so that $\T_X$ has finite KMS-states at every $\be > 0$.
\end{enumerate}

\begin{example}
Consider the following graph on two vertices $v$ and $w$:
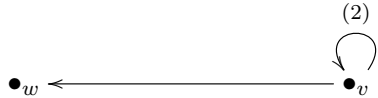
\[
\xymatrix@R=20pt@C=50pt{
\bullet_w & & \bullet_{v} \ar@{->}_{(2)}@(ur,ul) \ar[ll]
}
\]
It is immediate that the adjacency matrix satisfies
\[
G
=
\begin{bmatrix} 2 & 1 \\ 0 & 0 \end{bmatrix}
\qand
G^k
=
\begin{bmatrix} 2^k & 2^{k-1} \\ 0 & 0 \end{bmatrix}.
\]
If $\tau$ is a trace with $\tau(p_v) \neq 0$ then
\[
\sum_{|\mu| = k} \tau(\sca{x_\mu, x_\mu}) = \tau(p_v) \cdot (2^{k} + 2^{k-1}) = \tau(p_v) \cdot 3 \cdot 2^{k-1},
\]
and so $h_X^\tau = \log 2$.
On the other hand if $\tau(p_v) = 0$ then $\tau = \de_w$ with $c_{\de_w, \be} = 1$ for all $\be > 0$.
Moreover we see that $G^t$ has an eigenvector $[2/3, \; 1/3]$ for the eigenvalue $2$.
Therefore in this example we get
\[
\Tr_\be(A) = 
\begin{cases}
\{\de_w\} & \textup{ if } \be \in (0, \log 2], \\
\Tr(A) & \textup{ if } \be \in (\log 2, +\infty),
\end{cases}
\qand
\Avt_\be(A)
=
\begin{cases}
\{\frac{2}{3}\de_v + \frac{1}{3} \de_w\} & \textup{ if } \be = \log 2, \\
\mt & \textup{ otherwise}.
\end{cases}
\]
In particular we see that $h_X^{\de_v} = \log 2$ but $\de_{v}$ does not define an averaging trace at $\log 2$ (actually this holds for any trace that is supported on $v$, apart from $p = \frac{2}{3}\de_v + \frac{1}{3} \de_w$).

On the other hand we see that $\de_w$ defines a finite KMS-state for all $\be > 0$ given by
\begin{align*}
\Phi_\be(\de_w)(f)
& =
\begin{cases}
1 & \textup{ if $f = p_w$ or $f = 1$},\\
0 & \textup{ otherwise},
\end{cases}
\end{align*}
for all $f \in \T_X$.
By taking the weak*-limit as $\be \to 0$, the KMS-states theory then predicts a gauge-invariant tracial state on $\T_X$ of the same form (being constant with respect to $\be$).
This can be created directly.
Indeed let the representation of $\T_X$ on the $\ell^2$-path space of $G$, i.e., on the space generated by $\{\xi_\mu \mid \textup{ $\mu$ is a path in $G$}\}$.
Then we get that the vector state
\[
\Phi(f) = \sca{\xi_w, f \xi_w}, \foral f \in \T_X,
\]
has the aforementioned form.
\end{example}

The next two examples illustrate the contribution of just the sink irreducible components for getting $h_X$, and the contribution of all communicating components for getting a tracial entropy.

\begin{example}
Consider the following graph on three vertices:

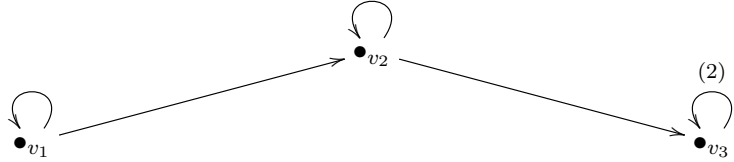
\vspace{5pt}
\begin{align*}
\xymatrix@R=20pt@C=50pt{
& & \bullet_{v_2} \ar@{->}@(ur,ul) \ar[drr] & & \\
\bullet_{v_1} \ar@{->}@(ur,ul) \ar[urr] & & & & \bullet_{v_3} \ar@{->}_{(2)}@(ur,ul)
}
\end{align*}
Its adjacency matrix then satisfies
\[
G
=
\begin{bmatrix}
1 & 1 & 0 \\
0 & 1 & 1 \\
0 & 0 & 2
\end{bmatrix}
\qand
G^k
=
\begin{bmatrix}
1 & k & a_k \\
0 & 1 & b_k \\
0 & 0 & 2^k
\end{bmatrix},
\]
for the sequences
\[
a_k
=
\begin{cases}
0 & \textup{ if } k=1,\\
1 & \textup{ if } k=2,\\
(k - 1) + 2 a_{k-1} & \textup{ if } k \geq 3,
\end{cases}
\qand
b_k = 2^{k+1} -1 \qfor k \geq 1.
\]
On one hand we have that $a_k \geq 2^{k-1}$ for $k \geq 2$.
On the other hand if $k \geq 2$ then
\begin{align*}
a_k 
& =
(k-1) + 2 a_{k-2}
 =
\cdots
\\
& =
(k-1) + 2(k-2) + 2^2(k-3) + \cdots + 2^{k-3} + 2^{k-3} a_2 \\
& \leq
k (1 + 2 + 2^2 + \cdots + 2^{k-3} + 2^{k-2})
=
k \cdot 2^{k-1}.
\end{align*}
Therefore we have
\[
2^{k-1} \leq k + a_k \leq k \cdot 2^{k}
\qand
2^k \leq 1 + b_k \leq 2^{k+1}.
\]
For a trace $\tau \in \Tr(A)$ let $p_i =\tau(p_{v_i})$ and set $p_{\max} = \max\{p_1, p_2, p_3\}$ and $p_{\min} =\min\{p_i \neq 0 \mid i=1,2,3\}$.
Thus we get that
\[
p_{\min} \cdot 2^{k-1}
\leq 
p_1(k + a_k) + p_2(1+ b_k) + p_3 2^k
=
\sum_{|\mu| = k} \tau(\sca{x_{\mu}, x_{\mu}})
\leq 
p_{\max} \cdot 3k \cdot 2^{k+1}.
\]
Consequently $c_{\tau, \be} < \infty$ if and only if $\be>\log 2$.
In particular $h_X^\tau = \log 2$ for all $\tau \in \Tr(A)$ and so $h_X = \log 2$.
We see here that $G$ has one sink irreducible component that contributes to $h_X$.
Moreover $G^t$ has two eigenvectors $[0, \; 1/2, \; -1/2]$ and $[0, \; 0, \; 1]$ at $1$ and $2$, respectively.
Thus we have
\[
\Tr_\be(A)
=
\begin{cases}
\mt & \textup{ if } \be = \log 2, \\
\Tr(A) & \textup{ if } \be \in (\log 2, +\infty),
\end{cases}
\qand
\Avt_\be(A)
=
\begin{cases}
\{\de_{v_3}\} & \textup{ if } \be = \log 2, \\
\mt & \textup{ if } \be \in (\log 2, +\infty).
\end{cases}
\]
Although $G^t$ has a smaller eigenvalue than $2$, the corresponding eigenvector is not positive and thus does not contribute to the KMS-simplex.
\end{example}

\begin{example}\label{E:inductive limit}
Consider the following graph $G$ on three vertices:
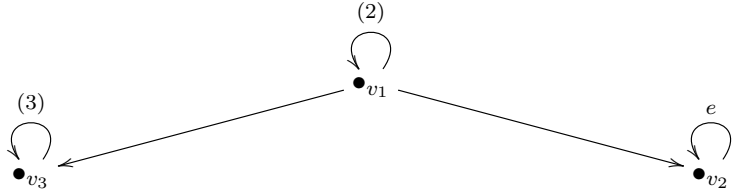
\begin{align*}
\xymatrix@R=20pt@C=50pt{
& & \bullet_{v_1} \ar@{->}_{(2)}@(ur,ul) \ar[drr] \ar[dll] & & \\
\bullet_{v_3} \ar@{->}_{(3)}@(ur,ul) & & & & \bullet_{v_2} \ar@{->}_{e}@(ur,ul)
}
\end{align*}
It follows that the adjacency matrix satisfies
\[
G
=
\begin{bmatrix}
2 & 1 & 1 \\
0 & 1 & 0 \\
0 & 0 & 3
\end{bmatrix}
\qand
G^k
=
\begin{bmatrix}
2^{k} & 2^{k} -1 & 3^{k} - 2^{k} \\
0 & 1 & 0 \\
0 & 0 & 3^k
\end{bmatrix}.
\]
Therefore if $\tau \in \Tr(A)$ with $p_i = \tau(p_{v_i})$ we have that
\[
\sum_{|\mu|=k} \tau(\sca{x_\mu, x_\mu})
=
p_1(3^{k} + 2^{k} - 1) + p_2 + p_3 3^{k}.
\]
Set $p_{\max} = \max\{p_1, p_2, p_3\}$ and $p_{\min} = \min\{p_i \neq 0 \mid i=1,2,3\}$.
We then see that if $p_1 \neq 0$, or if $p_1 =0$ but $p_3 \neq 0$ then
\[
p_{\min} 3^{k} \leq \sum_{|\mu| = k} \tau(\sca{x_\mu, x_\mu}) \leq p_{\max} 3^{k+1}.
\]
Thus $c_{\tau, \be} < \infty$ if and only if $\be > \log 3$.
Here we see that $h_X^{\tau} = \log 3$ even for $\tau$ with $\tau(p_{v_1}) \neq 0$ and $\tau(p_{v_3}) = 0$.
Such a state may not be supported on $v_3$, e.g., $\tau = \de_{v_1}$, but its support communicates with $v_3$ and thus its entropy is affected by it.

On the other hand if $p_1 = p_3 = 0$ then $\tau = \de_{v_2}$, and so $c_{\de_{v_2}, \be} < \infty$ if and only if $\be>0$, in which case $c_{\de_{v_2}, \be} = (1 - e^{-\be})^{-1}$.
Furthermore $G^t$ has three eigenvalues, namely $1,2$ and $3$ with corresponding $\ell^1$-eigenvectors to be $[0, \; 1, \; 0]$, $[1/3, \; 1/3, \; -1/3]$ and $[0, \; 0, \; 1]$.
The averaging traces correspond to positive eigenvectors and so we have that
\[
\Tr_\be(A)
=
\begin{cases}
\{\de_{v_2}\} & \textup{ if } \be \in (0, \log 3], \\
\Tr(A) & \textup{ if } \be \in (\log 3, +\infty),
\end{cases}
\qand
\Avt_\be(A)
=
\begin{cases}
\{\de_{v_3} \} & \textup{ if } \be = \log 3, \\
\mt & \textup{ otherwise.}
\end{cases}
\]
The eigenvector corresponding to the eigevalue $2$ does not contribute an averaging trace.

It is worth paying some more attention to $\{\de_{v_2}\}$.
By construction we see that
\begin{align*}
\de_{v_2}(\sca{x_f, p_v x_f})
& =
\begin{cases}
\de_{v_2}(p_{s(f)}) & \textup{ if $v =r(f)$}, \\
0 & \textup{ otherwise},
\end{cases} \\
& =
\begin{cases}
1 & \textup{ if $v_2 = s(f)$ and $v =r(f)$}, \\
0 & \textup{ otherwise},
\end{cases} \\
& =
\de_{v_2, v} \cdot \de_{e,f}.
\end{align*}
Thus $\de_{v_2}$ is an averaging trace at $\be = 0$ since
\[
\sum_{f \in G^{(1)}} \de_{v_2}(\sca{x_f, p_v x_f})
=
\de_{v_2, v} = \de_{v_2}(p_v).
\]
On the other hand for $\be \in(0, \log 3)$ we have a single finite KMS-state with the property
\begin{align*}
\Phi_\be(\de_{v_2})(L_{\nu_1} L_{\nu_2}^*)
& =
\de_{|\nu_1|, |\nu_2|} (1 - e^{-\be}) \sum_{k=0}^\infty e^{-k\be} \sum_{|\mu| = k} \de_{v_0}(L_\mu^* L_{\nu_2}^* L_{\nu_1} L_\mu) \\
& =
\begin{cases}
1 & \textup{ if  $\nu_1=\nu_2=e^\ell$ for some $\ell \in \bN$},\\
0 & \textup{ otherwise.}
\end{cases}
\end{align*}
The KMS-states theory predicts the existence of a KMS-state at $\be=0$ as the weak*-limit of the $\Phi_\be(\de_{v_2})$ for $\be \to 0$.
In particular there exists a tracial state of $\ca(G)$ such that $\Phi = \Phi \circ E$, for the conditional expectation $E$, for which
\[
\Phi(L_{\nu_1} L_{\nu_2}^*)
\begin{cases}
1 & \textup{ if  $\nu_1 = \nu_2 = e^\ell$ for some $\ell \in \bN$},\\
0 & \textup{ otherwise.}
\end{cases}
\]
We will show how this state can be induced directly.
To this end let the representation on the infinite path space of $G$.
For the infinite path $e^{(\infty)}$ on $v_2$ set the state
\[
\Psi(f) := \sca{\xi_{e^{(\infty)}}, E(f) \xi_{e^{(\infty)}}} 
\foral
f \in \ca(G).
\]
Then a direct computation yields
\begin{align*}
\Psi(L_{\nu_1} L_{\nu_2}^* L_{\nu_1'} L_{\nu_2'}^*)
& =
\de_{|\nu_1| + |\nu_1'|, |\nu_2| + |\nu_2'|} \sca{ L_{\nu_2} L_{\nu_1}^* \xi_{e^{(\infty)}}, L_{\nu_1'} L_{\nu_2'}^*\xi_{e^{(\infty)}}} \\
& =
\begin{cases}
1 & \textup{ if $\nu_1 =e^\ell$, $\nu_2 = e^k$, $\nu_1' = e^m$, $\nu_2' = e^n$, and $\ell + m = k + n$}, \\
0 & \textup{ otherwise}.
\end{cases}
\end{align*}
The symmetry on $\nu_i, \nu_i'$ shows that $\Psi$ is tracial.
By setting $\nu_1'= \nu_2' = \mt$, namely the void path, we get that $\Psi$ coincides with $\Phi$.
Note that $\Psi$ coincides with the state $\Psi(\de_{v_0})$ obtained by the inductive limit process of Theorem \ref{T:para 0} for the averaging trace $\de_{v_0}$ (even though $\be =0$).
\end{example}

Finally let us give an example with multiple phase transitions.
With a small tweak we get a case where the Cuntz-Pimsner algebra has plenty of finite KMS-states.

\begin{example}\label{E:collection}
Fix a collection of positive integers $\{1 < a_1< \dots < a_n\}$ and let the graph $G$ be 
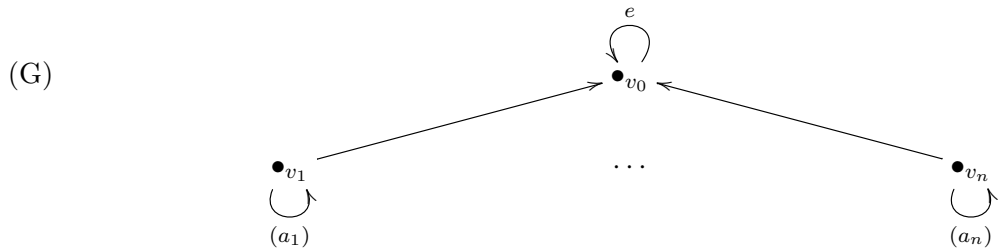
\begin{equation}\tag{G}
\xymatrix@R=20pt@C=50pt{
& & \bullet_{v_0} \ar@{->}_{e}@(ur,ul) & & \\
\bullet_{v_1} \ar@{->}_{(a_1)}@(dl,dr) \ar[urr] & & \cdots & & \bullet_{v_n} \ar@{->}_{(a_n)}@(dl,dr) \ar[ull]
}
\end{equation}
where $(a_j)$ denotes the number of cycles on the vertex $v_j$.
Its adjacency matrix satisfies
\[
G
=
\begin{bmatrix}
a_n & 0 & \cdots & 0 & 1 \\
0 & a_{n-1} & \cdots & 0 & 1 \\
\vdots & \vdots & \cdots & \vdots & \vdots \\
0 & 0 & \cdots & a_1 & 1 \\
0 & 0 & \cdots & 0 & 1
\end{bmatrix}
\qand
G^k
=
\begin{bmatrix}
a_n^k & 0 & \cdots & 0 & (a_n^{k} -1)/(a_n - 1) \\
0 & a_{n-1}^k & \cdots & 0 & (a_{n-1}^{k} -1)/(a_{n-1} - 1) \\
\vdots & \vdots & \cdots & \vdots & \vdots \\
0 & 0 & \cdots & 1 & (a_1^{k} -1)/(a_1 - 1) \\
0 & 0 & \cdots & 0 & 1
\end{bmatrix}
.
\]
Let $\tau$ be a trace on the vertices and set $P$ be the diagonal matrix with $p_i =\tau(p_{v_i})$.
Then
\[
\sum_{|\mu| = k} \tau(\sca{x_\mu, x_\mu})
=
\sum_{i,j \in [n]} p_i (G^k)_{ij}.
\]
If $r = \max\{ i \in [n] \mid p_i \neq 0\}$ then 
\[
\log a_r
=
\lim_k \frac{1}{k} \log (p_r a_r^{k-1})
\leq
h_X^\tau
\leq
\lim_k \frac{1}{k} \log (r^2 a_r^{k+1}) 
=
\log a_r
\]
and so $h_X^\tau = \log a_r$.
Hence any trace supported on $\{v_1, \dots, v_r\}$ defines a state in $\Eq_\be^{\fty}(\T_X)$ as long as $\be > \log a_r$.
Moreover notice that $G^t$ admits the vectors $\{e_i \mid i=1, \dots, n+1\}$ from the o.n.\ basis of $\bR^{n+1}$ as eigenvectors for the eigenvalues $a_n, \dots, a_2, a_1, 1$.
Thus for every $\log a_i$ we get that the Dirac measure $\de_{v_i}$ on $v_i$ is in $\Avt_{\log a_i}(A)$ and so $\Eq_{\log a_i}^\infty(\T_X) \neq \mt$.
Moreover we have that $c_{\de_{v_0}, \be} = 1$ for all $\be >0$.
Thus we conclude:
\begin{enumerate}[leftmargin=30pt, itemindent=0pt, itemsep=1pt]
\item The Dirac measure $\de_{v_0}$ induces a state in $\Eq_\be^{\fty}(\T_X)$ for all $\be >0$; therefore we get that $h_X = 0 < \log a_n = h_X^s$.
\item If $j > 1$ then any convex combination of $\{\de_{v_{j'}} \mid j' < j\}$ induces a state in $\Eq_{\log a_j}^{\fty}(\T_X)$ and $\de_{v_j}$ induces a state in $\Eq_{\log a_j}^\infty(\T_X)$.
\end{enumerate}
In this example we have $\Avt_\be(A) \neq \mt$ for a finite number of $\be$, whereas $\Tr_\be(A) \neq \mt$ for all $\be > 0$.

As in Example \ref{E:inductive limit} the Dirac measure $\de_{v_0}$ induces a tracial state $\Psi(\de_{v_0})$ on $\O_X$ and on $\T_X$ through the direct limit process of Theorem \ref{T:para 0}, or as the weak*-limit $\Psi(\de_{v_0}) =\lim_{\be \to 0} \Phi_\be(\de_{v_0})$.

With a small tweak we can produce a variant $G'$ for which $\Eq_\be^{\fty}(\O_{X'}) \neq \mt$ for any $\be > 0$.
Here $X'$ refers to the graph C*-correspondence related to the graph $G'$ given by:
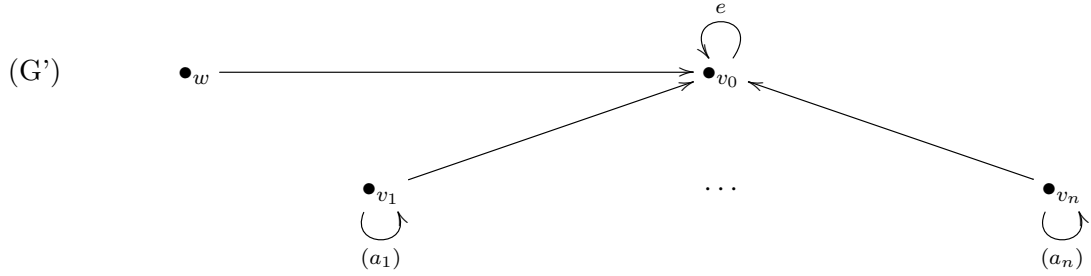
\begin{equation}\tag{G'}
\xymatrix@R=30pt@C=50pt{
\bullet_w \ar[rrr] & & & \bullet_{v_0} \ar@{->}_{e}@(ur,ul) & & \\
& \bullet_{v_1} \ar@{->}_{(a_1)}@(dl,dr) \ar[urr] & & \cdots & & \bullet_{v_n} \ar@{->}_{(a_n)}@(dl,dr) \ar[ull]
}
\end{equation}
Notice that all the entropies remain the same (there is only one path ending at $v_0$ of length $k$ that can be added).
If $\tau'$ is the Dirac measure on $w$ then $h_X^{\tau'} = 0$ and so $\tau' \in \Tr_\be(A)$ for all $\be >0$.
In light of Theorem \ref{T:para rel} we have that $\O_{X'}$ is the quotient of $\T_{X'}$ by $\K(\F(X') J)$ for $J = \ca(p_{v_0}, \dots, p_{v_n})$ and $\tau'|_J = 0$.
Therefore $\tau'$ induces a finite state on $\O_{X'}$ for all $\be > 0$.
\end{example}

\section{Examples from \cite{HLRS15}}\label{S:ex hlrs}

We square our analysis with the examples in \cite{HLRS15}.
The following refer to Examples 6.1--6.7, respectively.

\begin{example}
Let the graph
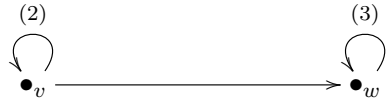
\[
\xymatrix@R=20pt@C=50pt{
\bullet_{v} \ar@{->}_{(2)}@(ur,ul) \ar[rr] & & \bullet_w \ar@{->}_{(3)}@(ur,ul)
}
\]
with adjacency matrix
\[
G = 
\begin{bmatrix}
2 & 1 \\
0 & 3
\end{bmatrix}
\]
with respect to the vertices $(w, v)$.
The components are $G_1 = [2]$ and $G_2 = [3]$ and so $h_X^s = \log 3$.
It has one sink component $G_2 = [3]$ and thus $h_X = \log 3$ as well.
We see that $G_1$ is not $2$-maximal, but $G_2$ is $3$-maximal.
Hence we have just one phase transition at $\log 3$.
Moreover $G_2^t$ has one eigenvector at $3$ given by $\de_w$.
By checking the communicating vertices, we have that
\[
h_X^\tau = \log 3 \foral \tau \in \Tr(A).
\]
Therefore for $\T_X$ we have 
\begin{itemize}[leftmargin=20pt, itemindent=0pt, itemsep=2pt]
\item For $\be > \log 3$ then $\Tr_\be(A) = \Tr(A) = \sca{\de_v, \de_w}$.
Thus $\Eq_\be(\T_X) = \Eq_\be^{\fty}(A)$ is a simplex of dimension one with extreme points $\Phi(\de_v)$ and $\Phi(\de_w)$.
\item For $\be = \log 3$ we have $\Tr_{\log 3}(A) = \mt$ and $\Avt_\be(A) =\{\de_v\}$.
Thus the KMS-simplex has dimension zero with one infinite state induced by $\de_v$.
\item For $\be < \log 3$ we have no KMS-states.
\end{itemize}
As $G$ has no sources we have that $\O_X =\O(A,X)$.
For $\O_X$ we thus have
\begin{itemize}[leftmargin=20pt, itemindent=0pt, itemsep=2pt]
\item For $\be = \log 3$ there is one infinite KMS-state induced by $\de_v$.
\item For $\be \neq \log 3$ we have no KMS-states.
\end{itemize}
\end{example}

\begin{example}
Let the graph
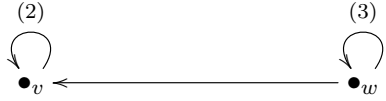
\[
\xymatrix@R=20pt@C=50pt{
\bullet_{v} \ar@{->}_{(2)}@(ur,ul) & & \bullet_w \ar@{->}_{(3)}@(ur,ul) \ar[ll]
}
\]
with adjacency matrix
\[
G = 
\begin{bmatrix}
3 & 1 \\
0 & 2
\end{bmatrix}
\]
with respect to the vertices $(w,v)$.
The graph has two components $G_1 = [3]$ and $G_2 = [2]$.
Thus $h_X^s = \log 3$.
The graph has one sink component $G_2 = [2]$ and so $h_X = \log 2$.
We see that $G_1$ is $3$-maximal and $G_2$ is $2$-maximal.
Hence we have two phase transitions at $\log 2$ and $\log 3$.
Moreover $H_1^t = G^t$ has an eigenvector at $3$ given by
\[
p = \frac{1}{2}(\de_w + \de_v)
\]
and $H_2^t = G_2^t$ has an eigenvector at $2$ given by $\de_v$.
By checking the communicating vertices, for $\tau \in \Tr(A)$ we have
\[
h_X^\tau = 
\begin{cases}
\log 2 & \textup{ if $\tau = \de_v$},\\
\log 3 & \textup{ otherwise}.
\end{cases}
\]
Therefore for $\T_X$ we have
\begin{itemize}[leftmargin=20pt, itemindent=0pt, itemsep=2pt]
\item For $\be > \log 3$ we have $\Tr_\be(A) = \Tr(A) = \sca{\de_v, \de_w}$.
Thus $\Eq_\be(\T_X) = \Eq_\be^{\fty}(A)$ is a simplex of dimension one with extreme points $\Phi(\de_v)$ and $\Phi(\de_w)$.
\item For $\be = \log 3$ we have $\Tr_{\log 3}(A) = \{\de_v\}$ and $\Avt_{\log 3}(A) = \{p'\}$.
So the KMS-simplex has dimension one.
\item For $\be \in (\log 2, \log 3)$ we have that $\Tr_\be(A) = \{\de_v\}$ and $\Avt_{\log 3}(A) = \mt$.
Hence we have only one finite KMS-state $\Phi(\de_v)$.
\item For $\be = \log 2$ we have $\Tr_{\log 2}(A) =\mt$ and $\Avt_{\log 2}(A) = \{\de_v\}$.
Hence we have one infinite KMS-state.
\end{itemize}
The graph has no sources and so $\O_X = \O(A,X)$ inherits the infinite KMS-states:
\begin{itemize}[leftmargin=20pt, itemindent=0pt, itemsep=2pt]
\item For $\be = \log 3$ there is one KMS-state induced by $p'$.
\item For $\be = \log 2$ there is one KMS-state induced by $\de_v$.
\item For $\be \neq \log 2, \log 3$ we have no KMS-states.
\end{itemize}
\end{example}

\begin{example}
Let the graph
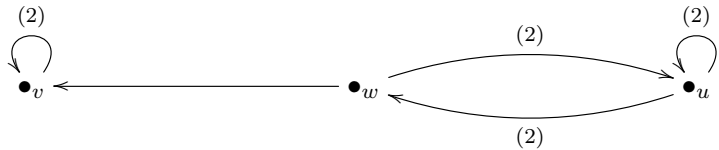
\[
\xymatrix@R=20pt@C=50pt{
\bullet_{v} \ar@{->}_{(2)}@(ur,ul) & & \bullet_w \ar[ll] \ar@{->}^{(2)}@/^1pc/[rr] & & \bullet_u \ar@{->}_{(2)}@(ur,ul) \ar@{->}^{(2)}@/^1pc/[ll]
}
\]
with adjacency matrix and components
\[
G = 
\begin{bmatrix}
2 & 2 & 0 \\
2 & 0 & 1 \\
0 & 0 & 2
\end{bmatrix},
\quad
G_1
=
\begin{bmatrix}
2 & 2 \\
2 & 0 
\end{bmatrix},
\quad
G_2
=
\begin{bmatrix}
2
\end{bmatrix},
\]
with respect to the order $(u, w ,v)$.
We have $\la_{G_1} = 1 + \sqrt{5} =: \ga$ and $\la_{G_2} = 2$.
As $\ga > 2$ we have $h_X^s = \ga$.
We have one sink component, namely $G_2$, and so $h_X = \log 2$.
We see that $G_1$ is $\ga$-maximal and $G_2$ is $2$-maximal.
Hence we have two phase transitions at $\log \ga$ and $\log 2$.
Since $H_1 = G$ we isolate the eigenvector of $G^t$ at $\ga$ given by
\[
p = \frac{1}{\ga + 1} (2 \de_u + (\ga - 2) \de_w + \de_v),
\]
while $H_2^t = G_2^t$ has an eigenvector at $2$ given by $\de_v$.
By checking the communicating vertices, for $\tau \in \Tr(A)$ we have
\[
h_X^\tau = 
\begin{cases}
\log 2 & \textup{ if $\tau = \de_v$},\\
\log \ga & \textup{ otherwise}.
\end{cases}
\]
Therefore for $\T_X$ we have
\begin{itemize}[leftmargin=20pt, itemindent=0pt, itemsep=2pt]
\item For $\be > \log \ga$ we have $\Tr_\be(A) = \Tr(A) = \sca{\de_u, \de_w, \de_v}$.
Thus $\Eq_\be(\T_X) = \Eq_\be^{\fty}(A)$ is a simplex of dimension two with extreme points $\Phi(\de_u)$, $\Phi(\de_w)$ and $\Phi(\de_v)$.
\item For $\be = \log \ga$ we have $\Tr_{\log \ga}(A) = \{\de_v\}$ and $\Avt_{\log \ga}(A) = \{p\}$.
So the KMS-simplex has dimension two with one finite state $\Phi(\de_v)$ and one infinite state $\Psi(p)$.
\item For $\be \in (\log 2, \log \ga)$ we have $\Tr_\be(A) = \{\de_v\}$ and $\Avt_\be(A) =\mt$.
So the KMS-simplex has only one finite state $\Phi(\de_v)$.
\item For $\be = \log 2$ we have $\Tr_{\log 2}(A) = \mt$ and $\Avt_{\log 2}(A) = \{\de_v\}$.
So the KMS-simplex has only one infinite state $\Psi(\de_v)$.
\end{itemize}
As we have no sources $\O_X = \O(A,X)$. Thus $\O_X$ inherits just the infinite-type states.
That is
\begin{itemize}[leftmargin=20pt, itemindent=0pt, itemsep=2pt]
\item For $\be = \log \ga$ it has one KMS-state induced from $p$.
\item For $\be = \log 2$ it has one KMS-state induced from $\de_v$.
\item For $\be \neq \log \ga, \log 2$ it has no KMS-states.
\end{itemize}
\end{example}

\begin{example}
Let the graph
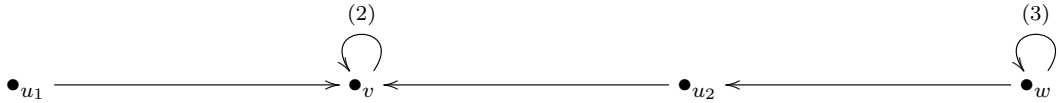
\[
\xymatrix@R=20pt@C=50pt{
\bullet_{u_1} \ar[rr] & & \bullet_{v} \ar@{->}_{(2)}@(ur,ul) & & \bullet_{u_2} \ar[ll] & & \bullet_{w} \ar[ll] \ar@{->}_{(3)}@(ur,ul)
}
\]
with adjacency matrix and components
\[
G = 
\begin{bmatrix}
3 & 1 & 0 & 0 \\
0 & 0 & 0 & 1 \\
0 & 0 & 0 & 1 \\
0 & 0 & 0 & 2
\end{bmatrix},
\quad
G_1
=
\begin{bmatrix}
3
\end{bmatrix},
\quad
G_2 = G_3 
=
\begin{bmatrix}
0
\end{bmatrix},
\quad
G_4
=
\begin{bmatrix}
2
\end{bmatrix},
\]
with respect to the order $(w, u_2, u_1, v)$.
We have $\la_{G_1} = 3$, $\la_{G_2} = \la_{G_3} = 0$ and $\la_{G_4} = 2$.
Therefore $h_X^s = \log 3$.
We have one sink component $G_4$ and so $h_X = \log 2$.
We see that $G_1$ is $3$-maximal and $G_4$ is $2$-maximal.
Hence we have two phase transitions at $\log 3$ and $\log 2$.
Since $H_1 = G$ we isolate the eigenvector of $G^t$ at $3$ given by
\[
p = \frac{1}{5}(3 \de_w + \de_{u_1} + \de_v),  
\]
and the eigenvector of $H_4^t = G_4^t$ at $2$ given by $\de_{v}$.
By checking the communicating vertices, for $\tau \in \Tr(A)$ we have
\[
h_X^\tau = 
\begin{cases}
\log 3 & \textup{ if $w \in \supp \tau$},\\
\log 2 & \textup{ otherwise}.
\end{cases}
\]
Therefore for $\T_X$ we have
\begin{itemize}[leftmargin=20pt, itemindent=0pt, itemsep=2pt]
\item For $\be > \log 3$ we have $\Tr_\be(A) = \Tr(A)$.
Thus $\Eq_\be(\T_X) = \Eq_\be^{\fty}(A)$ is a simplex of dimension three with extreme points $\Phi(\de_w)$, $\Phi(\de_{u_2})$, $\Phi(\de_{u_1})$ and $\Phi(\de_v)$.
\item For $\be = \log 3$ we have $\Tr_{\log 3}(A) = \sca{\de_{u_1}, \de_{u_2}, \de_v}$ and $\Avt_{\log 3}(A) = \{ p \}$.
So the KMS-simplex has dimension 3, with extreme points at the finite $\Phi(\de_{u_1})$, $\Phi(\de_{u_2})$, $\Phi(\de_{v})$ and at the infinite $\Psi(p)$.
\item For $\be \in (\log 2, \log 3)$ we have that $\Tr_\be(A) = \sca{\de_{u_1}, \de_{u_2}, \de_v}$ and $\Avt_\be(A) = \mt$.
So the KMS-simplex has dimension 2 with only finite states. 
\item For $\be = \log 2$ we have $\Tr_{\log 2}(A) = \mt$ and $\Avt_{\log 2}(A) = \{\de_v\}$.
So the KMS-simplex has zero dimension with one infinite state $\Psi(\de_v)$.
\item For $\be < \log 2$ we have no KMS-states.
\end{itemize}
Here we have one source so we need to be careful with the states that are inherited by $\O_X \neq \O(A,X)$.
For $\Tr_\be(A)$ we need to intersect with traces that are annihilated at $\{w, u_2, v\}$.
On the other hand we inherit the whole $\Avt_\be(A)$ as $A \hookrightarrow \O_X$.
Hence we have
\begin{itemize}[leftmargin=20pt, itemindent=0pt, itemsep=2pt]
\item For $\be > \log 3$ we have a single finite state induced by $\de_{u_1}$ and no infinite states.
\item For $\be = \log 3$ we have one finite state induced by $\de_{u_1}$ and one infinite state induced by $p$.
\item For $\be \in (\log 2, \log 3)$ we have a single finite state induced by $\de_{u_1}$ and no infinite states.
\item For $\be = \log 2$ we have a single infinite state induced by $\de_{u_1}$ and no finite states.
\item For $\be < \log 2$ we have no KMS-states.
\end{itemize}
\end{example}

\begin{example}
Let the graph
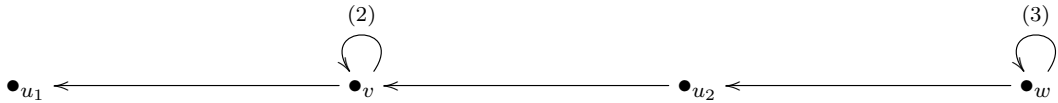
\[
\xymatrix@R=20pt@C=50pt{
\bullet_{u_1} & & \bullet_{v} \ar@{->}_{(2)}@(ur,ul) \ar[ll] & & \bullet_{u_2} \ar[ll] & & \bullet_{w} \ar[ll] \ar@{->}_{(3)}@(ur,ul)
}
\]
with adjacency matrix and components
\[
G = 
\begin{bmatrix}
3 & 1 & 0 & 0 \\
0 & 0 & 1 & 0 \\
0 & 0 & 2 & 1 \\
0 & 0 & 0 & 0
\end{bmatrix},
\quad
G_1
=
\begin{bmatrix}
3
\end{bmatrix},
\quad
G_2  
=
\begin{bmatrix}
0
\end{bmatrix},
\quad
G_3
=
\begin{bmatrix}
2 
\end{bmatrix},
\quad
G_4
=
\begin{bmatrix}
0
\end{bmatrix}
\]
with respect to the order $(w, u_2, v, u_1)$.
We have $\la_{G_1} = 3$, $\la_{G_2} = 0$, $\la_{G_3} = 2$ and $\la_{G_4} = 0$.
Therefore $h_X^s = \log 3$.
We have one sink component $G_4$ and so $h_X = 0$.
We see that $G_1$ is $3$-maximal, $G_3$ is $2$-maximal and $G_4$ is $0$-maximal.
Hence we have three phase transitions at $0$, $\log 2$ and $\log 3$.
For $G_1$ we have $H_1^t = G^t$ and we isolate the eigenvector at $3$ given by
\[
p = \frac{1}{16}(9 \de_w + 3 \de_{u_2} + 3 \de_{u_1} + \de_v).
\]
For $G_3$ we have an eigenvector at $2$ from
\[
H_3^t = \begin{bmatrix} 2 & 0 \\ 1 & 1 \end{bmatrix}
\quad
\textup{ with }
\quad
p' = \frac{1}{2} ( \de_v + \de_{u_1}).
\]
We also have the eigenvector $\de_{u_1}$ for $H_4^t = G_4^t$ at $0$.
By checking the communicating vertices, for $\tau \in \Tr(A)$ we have
\[
h_X^\tau = 
\begin{cases}
\log 3 & \textup{ if $w \in \supp \tau$},\\
0 & \textup{ if $\tau = \de_{u_1}$},\\
\log 2 & \textup{ otherwise}.
\end{cases}
\]
Therefore for $\T_X$ we have
\begin{itemize}[leftmargin=20pt, itemindent=0pt, itemsep=2pt]
\item For $\be > \log 3$ we have $\Tr_\be(A) = \Tr(A)$.
Thus $\Eq_\be(\T_X) = \Eq_\be^{\fty}(A)$ is a simplex of dimension three with extreme points $\Phi(\de_w)$, $\Phi(\de_{u_2})$, $\Phi(\de_v)$ and $\Phi(\de_{u_1})$.
\item For $\be = \log 3$ we have $\Tr_{\log 3}(A) = \sca{\de_{u_2}, \de_v, \de_{u_1}}$ and $\Avt_{\log 3}(A) = \{ p \}$.
So the KMS-simplex has dimension 3, with extreme points at the finite $\Phi(\de_{u_2})$, $\Phi(\de_v)$, $\Phi(\de_{u_1})$ and at the infinite $\Psi(p)$.
\item For $\be \in (\log 2, \log 3)$ we have that $\Tr_\be(A) = \sca{\de_{u_2}, \de_{v}, \de_{u_1}}$ and $\Avt_\be(A) = \mt$.
So the KMS-simplex has dimension 2 with only finite states. 
\item For $\be = \log 2$ we have $\Tr_{\log 2}(A) = \{\de_{u_1}\}$ and $\Avt_{\log 2}(A) = \{p'\}$.
So the KMS-simplex has dimension one with one extreme finite state $\Phi(\de_{u_1})$ and one extreme infinite state $\Psi(p')$.
\item For $\be \in (0, \log 2)$ we have $\Tr_\be(A) = \{\de_{u_1}\}$ and $\Avt_\be(A) = \mt$.
So the KMS-simplex has zero dimension with one finite state at $\Phi(\de_{u_1})$.
\item For $\be = 0$ we get a tracial state induced by $\Phi(\de_{u_1})$ (note that $\de_{u_1}$ does not satisfy the averaging condition).
\end{itemize}
Here we have no sources so $\O_X = \O(A,X)$ and inherits only the infinite type states.
Hence we have
\begin{itemize}[leftmargin=20pt, itemindent=0pt, itemsep=2pt]
\item For $\be = \log 3$ we have one infinite state induced by $p$.
\item For $\be = \log 2$ we have one infinite state induced by $p$.
\item For $\be \neq \log 2, \log 3$ we have no KMS-states.
\end{itemize}
\end{example}

\begin{example}
In this example one source is added to the previous graph to give 
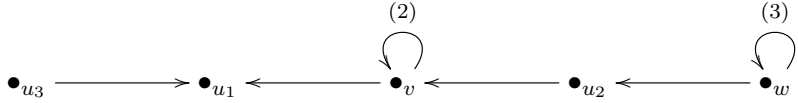
\[
\xymatrix@R=20pt@C=50pt{
\bullet_{u_3} \ar[r] & \bullet_{u_1} & \bullet_{v} \ar@{->}_{(2)}@(ur,ul) \ar[l] & \bullet_{u_2} \ar[l] & \bullet_{w} \ar[l] \ar@{->}_{(3)}@(ur,ul)
}
\]
The only change with the previous examplehere KMS-simplex by a finite state.
Indeed $u_3$ communicates just with $u_1$ which is a sink, and so the tracial entropy of $\de_{u_3}$ is zero.
As we have a source, then $\O_X \neq \O(A,X)$, and in particular $\Phi(\de_{u_3})$ descends to a finite state of $\O_X$ at every $\be \in [0, +\infty)$.
\end{example}

\begin{example}
In this example we consider the graph
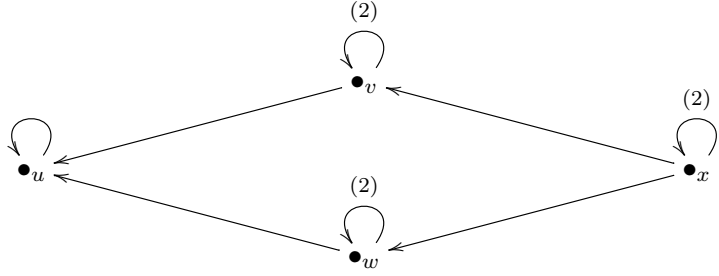
\[
\xymatrix@R=20pt@C=50pt{
& & \bullet_{v} \ar@{->}_{(2)}@(ur,ul) \ar[dll] & & \\
\bullet_{u} \ar@{->}@(ur,ul) & & & & \bullet_{x} \ar@{->}_{(2)}@(ur,ul) \ar[ull] \ar[dll]\\
& & \bullet_{w} \ar@{->}_{(2)}@(ur,ul) \ar[ull] & & 
}
\]
with adjacency matrix and components
\[
G = 
\begin{bmatrix}
2 & 1 & 1 & 0 \\
0 & 2 & 0 & 1 \\
0 & 0 & 2 & 1 \\
0 & 0 & 0 & 1
\end{bmatrix},
\quad
G_1 = G_2 = G_3
=
\begin{bmatrix}
2
\end{bmatrix},
\quad
G_4
=
\begin{bmatrix}
1
\end{bmatrix},
\]
with respect to the order $(x,v,w,u)$.
We have $\la_{G_1} = \la_{G_2} = \la_{G_3} = 2$ and $\la_{G_4} = 1$.
Therefore $h_X^s = \log 2$.
We have one sink component $G_4$ and so $h_X = 0$.
We see that the components $G_1, G_2, G_3$ are $2$-maximal and $G_4$ is $1$-maximal.
Hence we have two phase transitions at $0$ and $\log 2$.
Here we need to consider the eigenvectors for all $H_1^t, H_2^t, H_3^t, H_4^t$.
For $H_1^t = G^t$ we have two eigenvectors at $2$, namely $[0, \; 1, \; 0, \; 1]$ and $[0, \; -1, \; 1, \; 0]$, from which we isolate only the positive one.
In particular we see that for $G_2$ we obtain the ``same'' eigenvector at $2$ from
\[
H_2^t
=
\begin{bmatrix}
2 & 0 & 0 \\
0 & 2 & 0 \\
1 & 1 & 1
\end{bmatrix}
\quad
\textup{ with } 
\quad
p = \frac{1}{2}( \de_v + \de_u ).
\]
For $G_3$ we have one eigenvector $p'$ at $2$ from
\[
H_3^t
=
\begin{bmatrix}
2 & 0 \\
1 & 1 
\end{bmatrix}
\quad
\textup{ with } 
\quad
p' = \frac{1}{2}(\de_w + \de_u).
\]
Finally for $G_4$ we have the eigenvector $\de_u$ at $2$.
By checking the communicating vertices, for a trace $\tau$ we have
\[
h_X^\tau
=
\begin{cases}
0 & \textup{ if $\tau = \de_u$},\\
\log 2 & \textup{ otherwise}.
\end{cases}
\]
Therefore for $\T_X$ we have
\begin{itemize}[leftmargin=20pt, itemindent=0pt, itemsep=2pt]
\item For $\be > \log 2$ we have $\Tr_\be(A) = \Tr(A)$.
Thus $\Eq_\be(\T_X) = \Eq_\be^{\fty}(A)$ is a simplex of dimension three with extreme points $\Phi(\de_x)$, $\Phi(\de_{v})$, $\Phi(\de_w)$ and $\Phi(\de_{u})$.
\item For $\be = \log 2$ we have $\Tr_{\log 2}(A) = \{\de_u\}$ and $\Avt_{\log 2}(A) = \sca{p,  p'}$.
So the KMS-simplex has dimension two, with extreme points at the finite $\Phi(\de_{u})$ and the infinite $\Psi(p)$ and $\Psi(p')$.
\item For $\be \in (0, \log 2)$ we have that $\Tr_\be(A) = \{\de_u\}$ and $\Avt_\be(A) = \mt$.
So the KMS-simplex has dimension zero with a single finite state $\Phi(\de_u)$. 
\item For $\be = 0$ we get a tracial state induced by $\Psi(\de_{u})$.
\end{itemize}
Here we have no sources so $\O_X = \O(A,X)$, and so $\O_X$ inherits only the infinite type states.
Hence we have
\begin{itemize}[leftmargin=20pt, itemindent=0pt, itemsep=2pt]
\item For $\be = \log 2$ the KMS-simplex has dimension two with extreme points induced by $\{p, p'\}$.
\item For $\be = 0$ we have a tracial state induced by $\de_v$.
\item For $\be \neq 0, \log 2$ we have no KMS-states.
\end{itemize}
\end{example}

\section{Ground states and KMS${}_\infty$-states}

We follow \cite{LacRae10} and make a distinction between KMS${}_\infty$-states and ground states.
The following theorems make that difference clear.
The form of the ground states has been identified in \cite[Theorem 2.2]{LacNes04}.

\begin{theorem}\label{T:ground}
Let $X$ be a C*-correspondence of finite rank over $A$.
Then there exists an affine weak*-homeomoprhism $\tau \mapsto \vphi_\tau$ between the states $\tau \in \S(A)$ (resp. the tracial states $\tau \in \Tr(A)$) and the ground states (resp. the KMS${}_\infty$-states) of $\T_X$ such that
\begin{equation}\label{eq:ground}
\vphi_\tau(\pi(a)) = \tau(a) \foral a \in A
\qand
\vphi_\tau(t(\xi^{\otimes n}) t(\eta^{\otimes m})^*) = 0 \text{ when } n + m \neq 0.
\end{equation}
\end{theorem}

\begin{proof}
For a state $\tau \in \S(A)$ consider the GNS-representa\-tion $(H_\tau, x_\tau, \rho_\tau)$.
Let again $(\rho, v)$ be the induced representation of $\T_X$ on $H = \F X \otimes_{\rho_\tau} H_\tau$ and let $\vphi_\tau$ be the vector state given by
\[
\vphi_\tau(f) := \sca{x_0 \otimes x_\tau, (\rho \times v)(f) x_0 \otimes x_\tau}_{H} = \tau(p_0 f p_0).
\]
It is immediate that $\vphi_\tau$ satisfies the conditions of the statement.
This also shows that the map $\tau \mapsto \vphi_\tau$ is injective.

Next we show that equation (\ref{eq:ground}) characterizes the ground states for $\tau \in \S(A)$.
Then surjectivity follows by noting that if $\vphi$ is a ground state of $\T_X$ then $\vphi = \vphi_\tau$ for $\tau = \vphi \pi$.
Let $\vphi$ be a ground state and let $m \neq 0$.
Then the function
\[
r + i s \mapsto \vphi(t(\xi^{\otimes n}) \si_{r + is}(t(\eta^{\otimes m})^*)) = e^{-i m r} e^{m s} \vphi(t(\xi^{\otimes n}) t(\eta^{\otimes m})^*)
\]
has to be bounded for all $s > 0$.
This can happen only if $\vphi(t(\xi^{\otimes n}) t(\eta^{\otimes m})^*) = 0$.
Now if $m = 0$ and $n \neq 0$ then we get that $\vphi(t(\xi^{\otimes n})) = 0$ by taking adjoints.
In any case
\[
\vphi(t(\xi^{\otimes n}) t(\eta^{\otimes m})^*) = 0 \textup{ when } n + m \neq 0.
\]
Since $\si_z = \id$ on $\pi(a)$ we also get that $\vphi \pi \in \S(A)$ and so $\vphi$ satisfies equation (\ref{eq:ground}).
Conversely suppose that $\vphi$ satisfies equation (\ref{eq:ground}).
We have to show that, for any pair
\[
f = t(\xi^{\otimes n}) t(\eta^{\otimes m})^* 
\qand
g = t(\ze^{\otimes k}) t(y^{\otimes l})^*,
\]
the function $r + is \mapsto \vphi(f \si_{r + i s}(g))$ is bounded when $t > 0$.
Indeed we have that
\begin{align*}
|\vphi(f \si_{r + i s}(g))|^2 = e^{-(k - \ell) 2s} |\vphi(f g)|^2
\leq e^{-(k - \ell)2s} \vphi(f^*f) \vphi(g^*g).
\end{align*}
This is clearly bounded when $k - l \geq 0$.
Now if $k - \ell < 0$ then $l > 0$ and so
\[
\vphi(g^* g) = \vphi(t(y^{\otimes \ell} \langle \ze^{\otimes k}, \ze^{\otimes k} \rangle) t(y^{\otimes \ell})^*) = 0
\]
and thus $\vphi(f \si_{r + i s}(g)) = 0$, which completes the proof.

Now we pass to the KMS${}_\infty$-states.
Suppose that $\vphi$ is a KMS${}_\infty$-state.
Due to weak*-compactness (and after passing to subsequences), we may choose a sequence $\be_j \uparrow \infty$ such that $\text{w*-}\lim_j \vphi_{\tau, \be_j}$ converges to a KMS${}_\infty$-state $\vphi$.
Then $\vphi|_{\pi(A)}$ is tracial and when $n + m \neq 0$ then
\[
\vphi(t(\xi^{\otimes n}) t(\eta^{\otimes m}))
=
\lim_{\be_j \to \infty} e^{-\be_j n} \de_{n,m} \vphi_{\tau, \be_j} (t(\eta^{\otimes m})^* t(\xi^{\otimes n}))
=
0,
\]
so that $\vphi$ satisfies equation (\ref{eq:ground}).
For surjectivity let $\vphi$ be a KMS${}_\infty$-state and set $\tau = \vphi \pi$.
Let $\be_j \uparrow \infty$ and without loss of generality assume that $\be_j > h_X^\tau$ for all $j$.
Then we can form $\vphi_{\tau, \be_j} \in \Eq_{\be_j}^{\fty}(\T_X)$ arising from Theorem \ref{T:para}.
After passing to a subsequence let $\vphi_\tau = \text{w*-}\lim_j \vphi_{\tau, \be_j}$.
We will show that $\vphi = \vphi_\tau$.
For $n + m \neq 0$ we have that
\begin{align*}
\vphi(t(\xi^{\otimes n}) t(\eta^{\otimes m})^*)
& = 
0
=
\vphi_\tau(t(\xi^{\otimes n}) t(\eta^{\otimes m})^*).
\end{align*}
Hence it suffices to show that $\vphi_\tau \pi = \tau$.
Fix a unit decomposition $x = \{x_1, \dots, x_d\}$.
Then for $a \in A$ we have
\begin{equation}\label{eq:infty}
\vphi_{\tau, \be_j}(\pi(a))
=
c_{\tau, \be_j}^{-1} \tau(a) + c_{\tau, \be_j}^{-1} \sum_{k=1}^\infty e^{-k\be_j} \sum_{|\mu| = k} \tau(\sca{x_\mu, a x_\mu}).
\end{equation}
Take $\eps > 0$ so that $h_X^\tau + \eps < \be_1 \leq \be_j$.
Then there exists an $N \in \bN$ such that $\sum_{|\mu| = k} \tau(\sca{x_\mu,x_\mu}) \leq e^{k(h_X^\tau + \eps)}$ for all $k \geq N$.
Therefore we get that
\begin{align*}
1 
& \leq 
c_{\tau, \be_j}
 \leq 
1 + \sum_{k=1}^{N-1} e^{-k\be_j} \sum_{|\mu| = k} \tau(\sca{x_\mu,x_\mu})
+
\sum_{k= N}^{\infty} e^{-k \be_j} e^{k(h_X^\tau +\eps)} \\
& =
1 
+
e^{N(-\be_j + h_X^\tau + \eps)} \frac{1}{1 - e^{-\be_j + h_X^\tau + \eps}}
+ 
\sum_{k=1}^{N-1} e^{-k\be_j} \sum_{|\mu| = k} \tau(\sca{x_\mu,x_\mu}).
\end{align*}
However we have that
\[
\lim_{\be_j \to \infty} 
\left[e^{N(-\be_j + h_X^\tau + \eps)} \frac{1}{1 - e^{-\be_j + h_X^\tau + \eps}}
+
\sum_{k=1}^{N-1} e^{-k\be_j} \sum_{|\mu| = k} \tau(\sca{x_\mu,x_\mu})\right]
=
0
\]
which gives $\lim_{\be_j \to \infty} c_{\tau, \be_j} = 1$.
Combining with positivity of $\tau$ we also derive that
\begin{align*}
| \sum_{k=1}^\infty e^{-k \be_j} \sum_{|\mu| = k} \tau(\sca{x_\mu, a x_\mu})|
& \leq \\
& \hspace{-2cm} \leq
\nor{a} \sum_{k=1}^\infty e^{-k\be_j} \sum_{|\mu| = k} \tau(\sca{x_\mu, x_\mu}) \\
& \hspace{-2cm} \leq
\nor{a} \left[ e^{N(-\be_j + h_X^\tau + \eps)} \frac{1}{1 - e^{-\be_j + h_X^\tau + \eps}} + \sum_{k=1}^{N-1} e^{-k\be_j} \sum_{|\mu| = k} \tau(\sca{x_\mu,x_\mu}) \right]
\stackrel{\be_j \to \infty}{\longrightarrow} 0.
\end{align*}
Thus taking limits $\be_j \uparrow \infty$ in equation (\ref{eq:infty}) we conclude the required $\vphi(\pi(a))= \tau(a)$.
\end{proof}

Finally we have the analogues for the ground states and the KMS${}_\infty$-states for $J$-relative Cuntz-Pimsner algebras.

\begin{theorem}\label{T:infty rel}
Let $X$ be a C*-correspondence of finite rank over $A$.
Suppose that $J \subseteq \phi_X^{-1}(\K X)$.
Then the mapping $\tau \mapsto \vphi_\tau$ for
\[
\vphi_\tau(\rho(a)) = \tau(a) \foral a \in A
\qand
\vphi_\tau(v(\xi^{\otimes n}) v(\eta^{\otimes m})^*) = 0 \text{ when } n+m \neq 0
\]
defines an affine weak*-homeomorphism from the states on $A$ (resp. from the tracial states on $A$) that vanish on $J$ onto the ground states of $\O(J,X)$ (resp. onto the KMS${}_\infty$-states) of $\O(J,X)$.
\end{theorem}

\begin{acknow}
The author would like to thank Marios Bounakis, David Kimsey and Gerasimos Rigopoulos for discussions on links with Physics.
The author would like also to thank Sergey Neshveyev for his very helpful remarks and corrections on a preprint of this paper.
\end{acknow}


\end{document}